\documentclass{amsart}
\usepackage{geometry}                
\usepackage{graphicx}
\usepackage{amssymb}
\usepackage{epstopdf}
\newtheorem{prop}{Proposition}[section]
\newtheorem{lemma}[prop]{Lemma}
\newtheorem{Cor}[prop]{Corollary}
\newtheorem{theorem}[prop]{Theorem}
\DeclareMathOperator{\Hopf}{Hopf}
\DeclareMathOperator{\Dil}{Dil}
\DeclareMathOperator{\Vol}{Vol}
\DeclareMathOperator{\Area}{Area}
\DeclareMathOperator{\Length}{Length}
\DeclareMathOperator{\Deg}{Deg}
\DeclareMathOperator{\SH}{SH}
\DeclareMathOperator{\Sq}{Sq}
\DeclareMathOperator{\Average}{Average}
\DeclareMathOperator{\Bad}{Bad}
\DeclareMathOperator{\Ball}{Ball}
\DeclareMathOperator{\Bil}{Bil}
\DeclareMathOperator{\Lin}{Lin}
\DeclareMathOperator{\Bouquet}{Bouquet}
\DeclareMathOperator{\Diag}{Diag}
\DeclareMathOperator{\Diam}{Diam}
\DeclareMathOperator{\FlatNorm}{Flat Norm}
\DeclareMathOperator{\GL}{GL}
\DeclareMathOperator{\HC}{HC}
\DeclareMathOperator{\Lip}{Lip}
\DeclareMathOperator{\Map}{Map}
\DeclareMathOperator{\Mass}{Mass}
\DeclareMathOperator{\SO}{SO}
\DeclareMathOperator{\Triv}{Triv}
\DeclareMathOperator{\Dim}{Dim}
\DeclareMathOperator{\Dist}{Dist}

\title{Contraction of areas vs. topology of mappings}
\author{Larry Guth}

\begin{document}

\begin{abstract} We construct homotopically non-trivial maps from $S^m$ to $S^{m-1}$ with arbitrarily small
$k$-dilation for each $k > (m+1)/2$.  We prove that homotopically non-trivial maps from $S^m$ to $S^{m-1}$ cannot
have arbitrarily small $k$-dilation for $ k \le (m+1)/2$.

\end{abstract}

\maketitle


\section{Introduction}

The $k$-dilation of a map between Riemannian manifolds measures how much the
map stretches $k$-dimensional volumes.  If $F$ is a $C^1$ map, 
we say that $\Dil_k(F) \le \lambda$ if each
$k$-dimensional surface $\Sigma$ in the domain is mapped to an image with $k$-dimensional
volume at most $\lambda \Vol_k (\Sigma)$.  The 1-dilation is the same as the Lipschitz constant.
We will study how the $k$-dilation of $F$ is related to its homotopy class.  The $k$-dilation describes the local geometry
of $F$, and we want to understand how the local geometry of $F$ influences its global topological features.  
We focus on maps from the unit $m$-sphere to the unit $n$-sphere.  

We begin with the following question: if a map $F: S^m \rightarrow S^n$ has small $k$-dilation, does the map $F$ have
to be contractible?  If a map $F: S^m \rightarrow S^n$ has $\Dil_1 F < 1$, then it is contractible.  If $\Dil_1 F < 1$, then
the diameter of the image of $F$ is $< \pi$, and so $F$ is not surjective.  
In \cite{TW} Tsui and Wang proved that maps with small 2-dilation are also contractible.

\newtheorem*{rtwo}{Tsui-Wang inequality}
\begin{rtwo} If $F: S^m \rightarrow S^n$ has $\Dil_2 F < 1$, (and if $m,n \ge 2$), then $F$ is contractible.
\end{rtwo}

In contrast, we will show that maps with small 3-dilation may not be contractible.

\newtheorem*{introexam}{Example}
\begin{introexam} There is a sequence of homotopically non-trivial maps $F_j: S^4 \rightarrow S^3$ with
$\Dil_3(F_j) \rightarrow 0$.
\end{introexam}

Our goal is to study this phenomenon.  We study the following question.

\newtheorem*{prob}{Main Question}
\begin{prob} Suppose that $F_j: S^m \rightarrow S^n$ is a sequence of maps, all in the same homotopy class,
with $\Dil_k(F_j) \rightarrow 0$.  What can we conclude about the homotopy class of the $F_j$?
\end{prob}

Our main result describes the situation for maps from $S^m$ to $S^{m-1}$.

\newtheorem*{mthm}{Main Theorem}
\begin{mthm} Fix an integer $m \ge 3$.  If $k > (m+1)/2$, then there is a sequence of homotopically
non-trivial maps $F_j$ from $S^m$ to $S^{m-1}$ with $k$-dilation tending to zero.  On the
other hand if $k \le (m+1)/2$, then every homotopically non-trivial map from $S^m$ to $S^{m-1}$ has $k$-dilation
at least $c(m) > 0$. 
\end{mthm}

In the first half of our theorem, we have to construct some homotopically non-trivial
maps with tiny $k$-dilations.  Our construction gives the following more general result.

\newtheorem*{hp}{An h-principle for $k$-dilation}
\begin{hp} Suppose that $F_0$ is a map from $S^m$ to $S^n$ with $m > n$ and $k > (m+1)/2$.
Then for any $\epsilon > 0$, we can homotope $F_0$ to a map $F$ with $k$-dilation less than
$\epsilon$.
\end{hp}

For a given map $F$, the $k$-dilations are related to each other by the inequalities $\Dil_1(F) \ge \Dil_2(F)^{1/2}
\ge \Dil_3(F)^{1/3} \ge ...$  So $\Dil_1(F) \le 1$ implies $\Dil_2(F) \le 1$ etc.  As $k$ increases, maps
with small $k$-dilation become easier to find.  Our results show that there is a phase transition at $k = (m+1)/2$.  When
$k > (m+1)/2$, the condition $\Dil_k F \le 1$ behaves rather flexibly.  When $k \le (m+1)/2$ the condition behaves
more rigidly.

We recall some previous results about $k$-dilation and homotopy type of maps.  We start with results about the 1-dilation, which is
much better understood.  In the paper ``Homotopical effects
of dilatation'' (\cite{GHED}), Gromov investigated the relationship between the 1-dilation of a map and its homotopy type.  
If $F$ is a map from $S^m$ to $S^m$, then its degree is at most $\Dil_1(F)^m$, and this
bound is sharp up to a constant factor.  A more difficult result from \cite{GHED} says
that if $F$ is a map from $S^{2n-1}$ to $S^n$ (with $n$ even), then its
Hopf invariant is at most $C_n \Dil_1(F)^{2n}$, and this upper bound is also sharp up to a constant factor. 
Recently, DeTurck, Gluck, and Storm \cite{DGS} proved that each Hopf fibration has the minimal 1-dilation in its homotopy class.
The 1-dilations of maps in torsion homotopy classes have not been studied as much, partly because it's difficult
to formulate a good question.  For each torsion homotopy class in $\pi_m(S^n)$, the minimal 1-dilation is finite
and $\ge 1$.  There is no good candidate for a minimizer, and so finding the exact minimal value of the 1-dilation
looks hopeless. 

In \cite{GCC}, Gromov posed the question how the $k$-dilation of a map $F: S^m \rightarrow S^n$ is related to its homotopy
class.  The estimates above about the degree and the Hopf invariant generalize to $k$-dilation.
The argument about the degree generalizes to show that a map $F: S^m \rightarrow S^m$ has degree at most
$\Dil_m(F)$.  The argument about the Hopf invariant gives
the following inequality.   (See \cite{GCC} section 3.6 and \cite{GMS} pages 358-59.)

\newtheorem*{ho}{Hopf invariant inequality}
\begin{ho} (Gromov) Suppose that $F: S^{2n-1} \rightarrow S^n$, with $n$ even.  Then
the Hopf invariant of $F$ is controlled by the $n$-dilation of $F$ as follows:

$$ | \Hopf(F) | \le C(n) \Dil_n(F)^2. $$

In particular, if $\Hopf(F) \not= 0$, then $\Dil_n(F) \ge c(n) > 0$.  
\end{ho}

The proof of the Hopf invariant inequality is based on differential forms.  On the one hand, the
Hopf invariant can be written in terms of differential forms.  On the other hand, differential forms interact
nicely with $k$-dilation.  If $\alpha$ is a $k$-form, then $\| F^* \alpha \|_\infty \le \Dil_k(F)
\| \alpha \|_\infty$.  This allows us to bound homotopy invariants defined using
differential forms in terms of $k$-dilation.  It seems significantly 
harder to connect the $k$-dilation with torsion homotopy invariants.  

In the second half of the main theorem, 
we prove a lower bound for the $k$-dilation of maps in some torsion homotopy classes.  
Our proof involves Steenrod squares,
and the technique gives the following slightly more general estimate.

\newtheorem*{st}{Steenrod squares and $k$-dilation}
\begin{st} Let $F$ be a map from $S^m$ to $S^n$, and let $X$ be the cell complex
formed by attaching an (m+1)-cell to $S^n$ using $F$: $X = B^{m+1} \cup_F S^n$.
If the complex $X$ has a non-trivial Steenrod square and if $k \le (m+1)/2$, then
$\Dil_k(F) \ge c(m) > 0$.
\end{st}

\noindent The non-trivial homotopy class in $\pi_m(S^{m-1})$ induces a non-trivial Steenrod square.  There are also homotopy
classes inducing a non-trivial Steenrod square in $\pi_m(S^{m-3})$ and $\pi_m(S^{m-7})$ if $m$ is sufficiently large.

We will also see that different homotopy classes in the same group $\pi_m(S^n)$ can interact with $k$-dilation differently.  
Bechtluft-Sachs observed a related phenomenon in \cite{BS}.  Building on his observation, we will construct
maps with arbitrarily small $k$-dilation (for some $k$) in homotopy classes that are suspensions.  For
example, we will prove the following result.

\begin{prop} Suppose that $a \in \pi_m(S^n)$ is the suspension of a homotopy class in $\pi_{m-1}(S^{n-1})$, and
that $m > n$.  Then the class $a$ can be realized by maps with arbitrarily small $n$-dilation.
\end{prop}

We will apply this method to some specific homotopy groups.  We begin with the group $\pi_7(S^4)$ which
is isomorphic to $\mathbb{Z} \oplus \mathbb{Z}_{12}$.

\begin{prop} Each torsion element $a \in \pi_7(S^4)$ can be realized by maps with arbitrarily small 4-dilation,
and each non-torsion element cannot. 
\end{prop}

The non-torsion classes have non-zero Hopf invariant, and their 4-dilation cannot be too small by the Hopf
invariant inequality.  The non-torsion elements are all suspensions, and using the suspension structure, their
4-dilations can be made arbitrarily small.  In other homotopy groups, different torsion elements may behave differently.  For example, 
we consider $\pi_8(S^5)$ which is isomorphic to $\mathbb{Z}_{24}$.

\begin{prop} The homotopy class corresponding to $12$ in $\mathbb{Z}_{24} = \pi_8(S^5)$ can be realized by maps
with arbitrarily small 4-dilation.  The homotopy classes corresponding to odd numbers in $\mathbb{Z}_{24} = \pi_8(S^5)$
cannot be realized with arbitrarily small 4-dilation. 
\end{prop}

The odd classes involve non-trivial Steenrod squares, and so the second statement follows from the Steenrod square inequality.
The class 12 is a suspension from $\pi_5(S^2)$, and using the suspension structure we are able to construct maps with
arbitrarily small 4-dilation.

\subsection{Thick tubes} \label{introthicktubes}

The results and questions we mentioned above have some connections with the geometry of tubes in $\mathbb{R}^m$.  We mention these
in the introduction because they may be easier to visualize than the main theorem.  In particular, instead of talking about a non-trivial Steenrod square, we can talk about a tube with a twist in it.

An $m$-dimensional tube is an embedding $I$ from $S^1 \times B^{m-1}$ into $\mathbb{R}^m$.  We write
$S^1(\delta)$ for the circle of radius $\delta$ and $B^{m-1}(R)$ for the ball of radius $R$.  We say an embedding
is $k$-expanding if it increases the $k$-dimensional area of each $k$-dimensional surface.  Finally, we say that
a tube has $k$-thickness $R$ if the embedding $I$ is a $k$-expanding embedding
from $S^1(\delta) \times B^{m-1}(R)$ into $\mathbb{R}^m$, for some $\delta > 0$.  

Surprisingly, there are 3-dimensional tubes with 2-thickness 1 inside of arbitrarily small balls $B^3(\epsilon)$.

\newtheorem*{gz}{Thick tube example}
\begin{gz} For every $\epsilon > 0$, there is a 3-dimensional tube with 2-thickness 1 embedded
in the $\epsilon$-ball $B^3(\epsilon) \subset \mathbb{R}^3$.
\end{gz}

The first example of this type that I'm aware of was given by Zel'dovitch in the 1970's (see \cite{Ar}).
The example generalizes in a straightforward way to higher dimensions.

\begin{gz} (Higher dimensions) In dimension $m$, if $k \ge (m+1)/2$, then a tube with $k$-thickness
1 may be embedded into an arbitrarily small ball.
\end{gz}

In \cite{Ge} Gehring studied the geometry of linked tubes in $\mathbb{R}^3$.  Gehring was interested in a geometric invariant
called the conformal modulus of a tube.  His methods give the following result about thick tubes.

\newtheorem*{gl}{Gehring linking inequality}

\begin{gl} If $I_1$ and $I_2$ are disjoint 3-dimensional tubes
with 2-thickness 1 contained in the ball $B^3(\epsilon)$, and if $\epsilon > 0$ is sufficiently small,
then the tubes have linking number zero.
\end{gl}

In the early 1990's, Freedman and He \cite{FH} extended Gehring's work, proving estimates for general knots and links.  For example,
they proved that a 3-dimensional tube with 2-thickness 1 contained
in a small ball must be unknotted.

An embedding $I: S^1 \times B^{m-1} \rightarrow \mathbb{R}^m$ defines a framing of the normal bundle of
the core circle $I ( S^1 \times \{ 0 \} )$.  Any embedded circle in $\mathbb{R}^m$ also has a canonical framing
of its normal bundle (up to homotopy), induced by the ambient space.  The relationship between the two framings defines an element in $\pi_1 (\SO(m-1) )$ which is $\mathbb{Z}$ for $m=3$ and $\mathbb{Z}_2$ for $m > 3$.  We call this element the 
twisting number of the embedding $I$.  If $m=3$, then the twisting
number of $I$ is equal to the linking number of the circles $I ( S^1 \times \{ p \} )$ and $I( S^1 \times \{ q \} )$ for any $p, q \in B^2$.  
Therefore, the linking inequality above has the following corollary.

\newtheorem*{gt}{Twisting inequality in three dimensions}

\begin{gt} If $I$ is a 3-dimensional tube with 2-thickness 1 contained in the ball $B^3(\epsilon)$,
and if $\epsilon$ is sufficiently small, then the twisting number of $I$ is equal to zero.
\end{gt}

In summary, it's possible to embed a thick tube into a tiny ball, but we cannot put a twist in it.
In higher dimensions, $m > 3$, there is no linking number of tubes and every tube
is unknotted, but the twisting number is still defined modulo 2.  We will prove
the following higher-dimensional generalization of Gehring's twisting inequality.

\newtheorem*{hdt}{Twisting inequality in high dimensions}

\begin{hdt} If $I$ is an $m$-dimensional tube, with $k$-thickness 1, contained in the
ball $B^m(\epsilon)$, if $\epsilon > 0$ is sufficiently small, and if $k \le (m+1)/2$, then the tube $I$ has twisting number zero.
\end{hdt}

In particular, if $m$ is any odd dimension $m \ge 5$, and if $k = (m+1)/2$, then
we can embed a $k$-thick tube into an arbitrarily small ball, but we cannot put
a twist in it.

The twisting inequality in three dimensions is closely related to the inequality $| \Hopf(F) | \lesssim \Dil_2(F)^2$ for maps
$F: S^3 \rightarrow S^2$.  The twisting inequality in higher dimensions is closely related to the
Steenrod square inequality and the main theorem of the paper.

\subsection{On the proof of the lower bound for $k$-dilation using Steenrod squares} \label{introlowerbound}

The main new idea in this paper is a way to prove a lower bound for the $k$-dilation of maps in certain
torsion homotopy classes.  Here is an outline of the argument.

As a warmup, we consider maps $S^3 \rightarrow S^2$ with non-zero Hopf invariant.  The Hopf invariant
inequality implies that maps with tiny 2-dilation must have zero Hopf invariant.  The original proof uses
differential forms, but this proof is hard to generalize to maps $S^m \rightarrow S^{m-1}$ with $m \ge 4$,
because the relevant homotopy invariant cannot be written in terms of differential forms.

We describe an alternate proof that a map $F: S^3 \rightarrow S^2$ with tiny 2-dilation has zero Hopf invariant.  The
Hopf invariant is closely related to cup products which are closely related to Cartesian products.  We consider the
Cartesian product $F \times F: S^3 \times S^3 \rightarrow S^2 \times S^2$.  We can read the Hopf invariant from $F \times F$ as follows.  
There is a 4-chain $Z_0$ in $S^3 \times S^3$ with the interesting property that $F \times F(Z_0)$ is always a cycle.
To see how this may happen, notice that $F \times F$ maps $\Diag(S^3)$ to $\Diag(S^2)$.  The diagonal $\Diag(S^3)$ is one
of the components of $\partial Z_0$ and $F \times F$ collapses it to something 2-dimensional.  The same happens to the other components of $\partial Z_0$, and so $F \times F$ seals the boundary closed making a cycle.  We let $Z(F)$ denote the cycle
$F \times F(Z_0)$.  The homology class of $Z(F)$ is the Hopf invariant of $F$.  Now it is easy to check that if $\Dil_2(F)$ is tiny,
then $\Dil_4(F \times F)$ is also tiny, and so the cycle $Z(F)$ has tiny 4-volume, and so it is homologically trivial in $S^2 \times S^2$.

This approach generalizes to maps $S^m \rightarrow S^{m-1}$, or more generally to maps $S^m \rightarrow S^n$ when the cell complex has a non-trivial Steenrod square.  The Steenrod squares are closely connected with the following twisted product.  For any space $X$, consider the product $S^i \times X \times X$, and consider the involution $I(\theta, x_1, x_2) = (- \theta, x_2, x_1)$.  The involution acts freely, and the quotient is denoted $\Gamma_i X$.  The space $\Gamma_i X$ is a fiber bundle with fiber $X \times X$ and base $\mathbb{RP}^i$.  The construction is also functorial, and so our map $F: S^m \rightarrow S^n$ induces a map $\Gamma_i F: \Gamma_i S^m \rightarrow \Gamma_i S^n$.  (Eventually we will choose $i = 2n - m -1$.)  As above, there is a $2n$-chain $Z_0$ in $\Gamma_i S^m$ with the interesting property that $\Gamma_i F(Z_0)$ is always a cycle in $\Gamma_i S^n$.  The $2n$-cycle $Z(F) = \Gamma_i F(Z_0)$ is homologically non-trivial if and only if the cell complex $B^{m+1} \cup_F S^n$ has a non-trivial Steenrod square.

We suppose $k \le (m+1)/2$ and that $\Dil_k(F)$ is tiny.  The $k$-dilation of $F$ gives information about the geometry of the map $\Gamma_i F$.  For some 
small 
values of $k$, $\Dil_{2n} \Gamma_i F$ can be controlled in terms of $\Dil_k F$.  For example, if $F: S^m \rightarrow S^{m-1}$ has sufficiently 
small 2-dilation, then the volume of $Z(F)$ is small, so $Z(F)$ is null-homologous, and so $B^{m+1} \cup_F S^n$ has trivial Steenrod squares.  
The same argument works if $F: S^m \rightarrow S^{m-3}$ has sufficiently small 4-dilation.
However, for most values of $k$ in the range $k \le (m+1)/2$, $\Dil_k F$ does not control $\Dil_{2n} \Gamma_i F$.
The $k$-dilation of $F$ may be arbitrarily small and yet $\Dil_{2n} \Gamma_i F$ and $\Vol Z(F)$ may be arbitrarily large.

The construction of $\Gamma_i F$ does not treat all the directions equally.  The double cover of $\Gamma_i F$ is defined to be
$id \times F \times F: S^i \times S^m \times S^m \rightarrow S^i \times S^n \times S^n$, where $id$ is the identity map.  We can see that the different factors are not treated in the same way.  Because the situation is not isotropic, we can make a more refined estimate that treats different directions differently.  If $k \le (m+1)/2$ and $\Dil_k(F)$ is tiny, then the tangent plane of $Z(F)$ is usually constrained to lie in certain `good' directions.  
We will divide $Z(F)$ into a good
set, where its tangent plane is constrained to lie in `good' directions, and a bad set of small volume where the tangent plane may go in any direction.

It takes a little time to say in detail what are the `good' directions.  To give some sense, we mention that
on the good set, the tangent plane of $Z(F)$ must be nearly orthogonal to the $S^n \times S^n$ direction.  If we look at the double cover of $Z(F)$ in $S^i \times S^n \times S^n$ and project it to the $S^n \times S^n$ factor, the volume of the projection is tiny.  This implies that the double cover of $Z(F)$ is homologically trivial.  But it may happen that the double cover of a cycle is homologically trivial and the cycle is not.

As a warmup, we first consider the case of a cycle $X$ in $\Gamma_i S^n$ whose tangent plane lies in the good directions at every point.  This condition forces every loop in $X$ to be homotopically trivial in $\Gamma_i S^n$.  
This is the key fact about the good directions for tangent planes.  Therefore, the double cover of $X$ consists of two disjoint cycles $X_1$ and $X_2$ in $S^i \times S^n \times S^n$.  Because of the control on the tangent direction of $X$, each of these cycles is trivial, and so $X$ is homologically trivial.

But the proof of our theorem is more complicated because the cycle $Z(F)$ has a bad set of small volume where the tangent plane
may face in any direction.  As a result, the cycle $Z(F)$ may contain non-trivial loops - we only know that each non-trivial loop must go through this small bad set.  The double cover of $Z(F)$ consists of two large pieces connected by a small bridge - the bridge being the double cover of the bad set.  We will cut out the small bridge and replace it by two small caps.  After this surgery, we get two homologically trivial cycles in $S^i \times S^n \times S^n$.  Projecting one of them down to $\Gamma_i S^n$, we get a trivial cycle which is very close to $Z(F)$.  It consists of $Z(F)$ minus the bad set and plus the image of the cap.  As long as the cap is small, it follows that $Z(F)$ is homologically trivial.

\subsection{On the proof of the h-principle} \label{introhprin}

We will give two constructions of maps with small $k$-dilation.  The first construction uses suspensions.  The construction is short, and we give it
early in the paper.  The second construction is used to prove the h-principle for $k$-dilation.  That construction is longer, and we describe
the main idea here.

The following simple observation is helpful to construct maps with small $k$-dilation.  If a map sends a region of the domain into a $(k-1)$-dimensional
polyhedron, then on that region the map has $k$-dilation zero.  With this observation in mind, we sketch the construction in the h-principle.

We begin with a map $F_0: S^m \rightarrow S^n$.  We consider a fine triangulation of
the target $S^n$.  We define $U \subset S^m$ so that $F_0$ maps the complement of $U$ into the $(k-1)$-skeleton of the fine triangulation.
We know that $F_0$ automatically has $(k-1)$-dilation zero on $U^c$, and we only have to worry about $U$.  Of course the $k$-dilation of $F_0$ on
$U$ is probably not tiny, and we have to modify $F_0$.  

We will carefully construct a degree 1 map $G: S^m \rightarrow S^m$, 
and the map $F$ will be $F_0 \circ G$.  Since $G$ is homotopic to
the identity, $F$ will be homotopic to $F_0$.  We let $T$ be $G^{-1}(U)$.  The map $G$ sends $T$ to $U$ and $T^c$ to $U^c$.  Therefore, the map
$F$ sends $T^c$ into the $(k-1)$-skeleton of the fine triangulation.  By our simple observation, the k-dilation of $F$ on $T^c$ is automatically zero.  In our
construction, the restriction of $F$ to $T^c$ will be very complicated and may have a huge $(k-1)$-dilation, but we don't have to keep track of it.
We only have to worry about the $k$-dilation of $F$ on $T$.  In fact, we will be able to arrange that the Lipschitz constant of $F$ on $T$ 
is very small.

Telling a minor white lie, the set $U$ is a fiber bundle, where the base space is $S^n$ minus the $(k-1)$-skeleton of our triangulation, and
each fiber is a manifold of dimension $m - n > 0$.  We call directions tangent to the fiber `vertical', and the perpendicular directions horizontal.
The map $G$ will be a diffeomorphism from $T$ to $U$, so $T$ also has this fiber bundle structure.  The map $G$ will strongly shrink all the horizontal
directions and strongly stretch all the vertical directions.  Since $F_0$ annihilates all the vertical directions, the composition 
$G \circ F_0$ will have a small Lipschitz constant.

It's actually easier to imagine $G^{-1}$ going from $U$ to $T$ than to imagine $G$.  We take a second to switch our perspective.  We want an embedding
$G^{-1}$ from $U$ into $S^m$ which stretches all the horizontal directions and shrinks all the vertical directions.  We build this embedding by isotoping
the identity map.

To get a sense of how to do the isotopy, we need to describe $U$ a little bit more.  
The complement of the $(k-1)$-skeleton of our triangulation of $S^n$ is a small neighborhood of the dual 
$n-k$-dimensional polyhedron.  Now $U$ is a small neighborhood of the inverse image of this polyhedron, so $U$ is a small neighborhood of a polyhedron
of dimension $p = m - k$.  The condition $k > (m+1)/2$ is equivalent to $p < (m-1)/2$.  

If $p < (m-1)/2$, then any p-dimensional polyhedron 
embeds in $S^m$, and any two embeddings are isotopic.  These facts follow from a standard general position argument.  
Informally, we may say that a p-dimensional polyhedron in $S^m$ may
be isotoped rather freely -- there is no obstruction to isotoping it into any position that we like.  This intuition extends to our set $U$, which
is a small neighborhood of a polyhedron of dimension $p < (m-1)/2$.  We isotope $U$ so that the horizontal directions all expand and the vertical
directions all contract.  Because of the condition $p < (m-1)/2$, there is no obstruction that prevents us from isotoping $U$ in this way.  The fibers
have to shrink, but everything can slide out of their way to allow them to shrink.  As they shrink, it creates space which allows the horizontal
directions to become thicker.

To make this argument precise and rigorous, we use the h-principle for immersions and general position arguments.  The h-principle for
immersions allows us to build immersions of $U$ rather freely.  Using the condition $p < (m-1)/2$, we
can modify these immersions to embeddings.  This last argument is a quantitative version of the general position argument mentioned in
the last paragraph.

\subsection{Outline of the paper}

Now we give an outline of the paper.  We describe what material appears in each section, and also what kinds of tools and background each
section uses.  

In Section 2, we give background about $k$-dilation.  This section contains all the facts about $k$-dilation that we need
in the paper.

In Section 3, we construct homotopically non-trivial maps with small $k$-dilation using suspensions.  With this method, 
we can construct the maps in the main theorem for $m \le 7$, and we can give some other examples in classes that are suspensions.  

The next large chunk of the paper is concerned with proving the lower bound for $k$-dilation for homotopy classes that induce a non-trivial
Steenrod square.  This argument follows the outline in Section \ref{introlowerbound}.  The argument involves a combination of topology
and isoperimetric-type estimates.  On the topology side, we need to use some facts about Steenrod squares.  On the geometry side,
we need to use the Federer-Fleming deformation theorem in many places.  We will also introduce some small variations on the deformation
theorem.  Here is a more detailed outline.
In Section 4, we prove the Hopf invariant inequality using the cycle $Z(F)$ as described above.  In Section 5, we
generalize this setup with the cycle $Z(F)$ to all homotopy classes with a non-trivial Steenrod square.  This section requires some
background on Steenrod squares.  The material we need is in Section 4L of Hatcher's book {\it Algebraic Topology}, \cite{H}.
In Section 6, we describe
how the geometry of $\Gamma_i F$ is controlled by $\Dil_k F$.  We define and describe `directed volumes' of $Z(F)$, and describe the
good and bad directions.  We see that the volume of $Z(F)$ in bad directions is controlled by $\Dil_k F$.  In Section 7, we
discuss the surgery where the bad part of the double cover of $Z(F)$ is cut out and replaced by small caps.  We see that the estimates
that we need about the size of the caps follow from an inequality about chains in Euclidean space called the perpendicular pair inequality.  
The perpendicular
pair inequality is a geometric inequality in the same area as the Federer-Fleming isoperimetric inequality (the isoperimetric
inequality for cycles of arbitrary codimension).  
In Section 8, we prove the perpendicular pair inequality.  This section heavily uses the deformation theorem.
There is a review of the deformation theorem in an appendix at the end of the paper.

In Section 9, we discuss thick tubes and prove the estimates about twisted tubes in Section \ref{introthicktubes}.

In Sections 10-11, we prove the h-principle for $k$-dilation.  As a special case, this gives all the maps in the main theorem.  This argument
follows the outline in Section \ref{introhprin}.  In Section 10, we prove a quantitative embedding result using a general position argument.
This result allows us to isotope the set $U$ rather freely, as described above.  Section 10 uses the h-principle for immersions.  The
material we need is in the book {\it Introduction to the H-Principle} by Eliashberg and Mishachev, \cite{EM}.  In Section 11, we prove
the h-principle for $k$-dilation.

This finishes the proofs of all the results in the paper.  Section 12 gives more background about $k$-dilation.  
It discusses other results in the literature that are generally relevant to the paper.  Section 13 discusses open problems.  
Section 14 contains several appendices.

\vskip5pt 

{\bf Acknowledgements.} I would like to thank Tom Mrowka, Brian White, Misha Gromov, and Robert Young for helpful conversations
related to the paper.

\section{Basic facts about $k$-dilation} \label{basicfacts}

In this section, we recall the definition of $k$-dilation, some different ways of looking at it, and some of its immediate consequences.
We hope to give some feel for the condition $\Dil_k F \le \lambda$.  Readers who are already familiar with $k$-dilation can skip this
section.

If $F$ is a $C^1$ map from $(M^m,g)$ to $(N^n,h)$, then we say that $F$ has $k$-dilation $\le \lambda$ if, for each
$k$-dimensional submanifold $\Sigma^k \subset M$,

$$ \Vol_k (F(\Sigma)) \le \lambda \Vol_k(\Sigma). $$

The $k$-dilation of $F$ can be described in terms of the derivative $dF$.  Suppose that $y = F(x)$, and we consider
the derivative $dF_x: T_x M \rightarrow T_y N$.  We consider the singular values of the derivative $dF_x$.  The singular value
decomposition says that we can find an orthonormal frame $v_1, ..., v_m$ for $T_x M$ and an orthonormal frame $w_1, ..., w_n$ for
$T_y N$ so that $dF_x (v_j) = s_j w_j$.  Here $s_j \ge 0$ are the singular values of $dF_x$.  (If $m > n$, then $s_j = 0$ for all $j > n$.)  We 
organize the singular values so that $s_1 \ge s_2 \ge s_3 \ge ...$.  The map $dF_x$ maps the unit ball in $T_x M$ to an ellipsoid in $T_y N$ with
principal radii $s_1 \ge s_2 \ge ...$.  The norm of $dF_x$ is $s_1$, and the norms of the exterior products $\Lambda^k dF_x$ are also
described by the singular values as follows.

\begin{prop} If $dF_x$ has singular values $s_1 \ge s_2 \ge ...$, then the norm of $\Lambda^k dF_x$ is $\prod_{i=1}^k s_i$.
\end{prop}

\begin{proof} Recall that $v_i$ are an orthonormal basis of $T_x M$ and $w_i$ an orthonormal basis of $T_y N$ with $dF_x v_i = s_i w_i$.  
There is an orthonormal basis of $\Lambda^k T_x M$ given by the wedge products $v_{i_1} \wedge ... \wedge v_{i_k}$, with $i_1 < ... < i_k$.  
We compute $\Lambda^k dF_x$ in this basis.

$$\Lambda^k dF_x (v_{i_1} \wedge ... \wedge v_{i_k}) = s_{i_1} ... s_{i_k} w_{i_1} \wedge ... \wedge w_{i_k}. $$

Since $\{ w_{i_1} \wedge ... \wedge w_{i_k} \}$ is an orthonormal basis of $\Lambda^k T_y N$, we see that the singular values of 
$\Lambda^k dF_x$ are just the products $s_{i_1} ... s_{i_k}$ with $i_1 < ... < i_k$.  
In particular, we see that the operator norm $| \Lambda^k dF_x | = s_1 ... s_k$. \end{proof}

Now we can write the $k$-dilation in terms of $\Lambda^k dF$ and/or the singular values of the derivative.

\begin{prop} If $F: (M,g) \rightarrow (N,h)$ is a $C^1$ map, then the $k$-dilation of $F$ is equal to $\sup_{x \in M} | \Lambda^k dF_x|$.  If $s_1(x) \ge s_2(x) \ge ... $ are the singular values of $dF_x$, then the $k$-dilation of $F$ is also equal to $\sup_{x \in M} s_1(x) ... s_k(x)$.  
\end{prop} 

\begin{proof} Fix $x \in M$, and suppose as above that $dF_x v_i = s_i w_i$.  
If $\Sigma$ is a small $k$-dimensional disk centered at $x$ with tangent plane spanned by $v_1, ..., v_k$, then $\Vol_k F(\Sigma)$ is
approximately $(\prod_{j=1}^k s_j) \Vol_k(\Sigma)$.  Therefore, $\Dil_k(F) \ge \prod_{j=1}^k s_j$.

If $\Sigma$ is an oriented $k$-dimensional manifold, then at each point $x \in \Sigma$, the tangent space of $\Sigma$ defines a unit $k$-vector $V_\Sigma(x) \in \Lambda^k T_x M$.  (Any $k$-dimensional manifold can be written as a union of oriented manifolds, so we can restrict attention to the case of oriented manifolds $\Sigma$.)  The volume of $F(\Sigma)$ can then be described as

$$\Vol_k F(\Sigma) = \int_{\Sigma} | \Lambda^k dF_x V_\Sigma(x) | dvol(x) \le \sup_x | \Lambda^k dF_x | \Vol_k \Sigma. $$

Therefore, $\Dil_k F \le \sup_{x \in M} | \Lambda^k dF_x| = \sup_{x \in M} \prod_{j=1}^k s_j(x)$. \end{proof}

With these results, we can give a little discussion of the condition $\Dil_k F \le 1$.  A map with small 2-dilation can have an arbitrarily 
large value of $s_1(x)$ as long as $s_2(x)$ is small enough to make $s_1 s_2 \le 1$.  So the derivative $dF_x$ can stretch in one direction
arbitrarily strongly as long as it contracts in the perpendicular directions enough to compensate.  Similarly, if $\Dil_k F \le 1$, then the
derivative $dF_x$ can stretch in (k-1) directions arbitrarily strongly, as long as it contracts in all the perpendicular directions enough
to compensate.  Also, $dF_x$ may behave differently for different $x$.  For example, if $\Dil_k(F) \le 1$, there may be some $x$ where
$dF_x$ is roughly an isometry and other $x$ where $dF_x$ stretches a lot in one direction and contracts in the others.  
The condition $\Dil_k(F) \le 1$ gives local
information about $dF_x$ for each $x$, and it's not easy to understand what kinds of global behavior are allowed by this condition.  To help
get a feel for this, consider the following example.

\begin{prop} There are surjective maps from the unit 3-ball to the unit 2-ball with arbitrarily small 2-dilation.
\end{prop}

\begin{proof} (sketch) Fix $\epsilon > 0$.  We want to find a surjective map from $B^3$ to $B^2$ with 2-dilation $\le \epsilon$.
It's easy to find a (linear) map with 2-dilation $\le \epsilon$ onto $B^2(\epsilon^{1/2})$.  By modifying this map a little, 
we can make a map $F_0$ with 2-dilation $\le \epsilon$ onto $B^2(\epsilon^{1/2}/10)$ and which maps $\partial B^3$ to a point.

We can improve the situation as follows.  Let $r$ be a small radius which we choose later.  Take $r^{-3}$ disjoint $r$-balls in
the unit 3-ball.  Using a rescaling of the map $F_0$, we can define $F$ on each $r$-ball so that its image covers a disk of
radius $r' = (1/10) \epsilon^{1/2} r$, and so that $F$ collapses the boundary of each $r$-ball to a point.  The image of $F$
now includes $r^{-3}$ discs in $B^2$ of radius $r' \sim \epsilon^{1/2} r$.  Taking $r$ sufficiently small compared to $\epsilon$, 
we can cover all of $B^2$ with these disks.

So far, we have only defined $F$ on the union of the disjoint balls of radius $r$.  Now we have to extend $F$ to the region between them
in a way that matches with $F$ on the boundary.  At this point, it's helpful to know that $F$ mapped the boundary of each ball to a point.
We choose a 1-dimensional tree in $B^2$ that includes all of these points, and we extend $F$ to a map from the complement of the balls
to the tree.  Since the tree is contractible, such an extension exists.  And since the image is 1-dimensional, the 2-dilation of the extension
is zero (on the complement of the balls). \end{proof}

Let's make some comments on this construction.  Notice that the singular values of $dF_x$ behave differently at different points.  At the
points inside the balls, we have singular values $s_1 \sim s_2 \sim \epsilon^{1/2}$.  At the points between the balls, we have singular
value $s_1$ very large and $s_2 = 0$.  The key to the construction is to mix these two behaviors.  For context, a linear map $L: B^3 \rightarrow
B^2$ with 2-dilation $\epsilon$ cannot be surjective - the image will have area $\lesssim \epsilon$.  A non-linear map
$F: B^3 \rightarrow B^2$ with tiny 2-dilation can be surjective, and one crucial ingredient is that the derivative of the map should be
wildly different at different places.  (This example can also be made stronger.  
In \cite{K}, Kaufman gives an example of a $C^1$ surjective map from $B^3$ to $B^2$ with 2-dilation equal to zero.)

We've just seen that (non-linear) maps with small $k$-dilation can do things that linear maps with small $k$-dilation can't.  There are also
some limits to this phenomenon.  For example, if $F$ is a $C^1$ map with $\Dil_k F > 1$, then there is no sequence of maps with $\Dil_k F_j \le 1$
which converges to $F$ in $C^0$.  In particular, (non-linear) maps with $k$-dilation $\le 1$ cannot imitate a linear map with $k$-dilation $> 1$.  
In Section \ref{prevlower}, we prove this result and give some further background on $k$-dilation.

Using the analysis with singular values, we can see how the $k$-dilations for differents values of $k$ are related to each other.

\begin{prop} \label{dilkdill} If $k < l$, then $\Dil_k(F)^{1/k} \ge \Dil_l(F)^{1/l}$.
\end{prop}

\begin{proof} For each point $x$, we have $s_1(x) \ge s_2(x) \ge ...$.  Since $k < l$, for each point $x$ we have

$$(\prod_{j=1}^k s_j(x)) ^{1/k} \ge (\prod_{j=1}^l s_j(x) )^{1/l}.$$  

Taking the supremum over $x$, we get $\Dil_k(F)^{1/k} \ge \Dil_l(F)^{1/l}$.

\end{proof}

So the $k$-dilations of $F$ for different $k$ are related
to each other by the following inequalities.

$$ \Dil_1(F) \ge \Dil_2(F)^{1/2} \ge \Dil_3(F)^{1/3} \ge ... $$

\noindent As $k$ increases, the condition $\Dil_k F \le 1$ gets weaker, and 
finding a map with small $k$-dilation gets easier.  

We can also see a connection between $k$-dilation and differential forms.

\begin{prop} \label{kdildiff} If $F: (M,g) \rightarrow (N,h)$ is a $C^1$ map and $\alpha$ is a $k$-form on $N$, then

$$\| F^* \alpha \|_\infty \le \Dil_k(F) \| \alpha \|_\infty. $$

\end{prop}

\begin{proof} For any point $x \in M$, let $y = F(x)$.  Then $(F^* \alpha)_x = \Lambda^k dF_x^* (\alpha_y)$. The map $\Lambda^k dF_x^*$ is the 
adjoint of the map $\Lambda^k dF_x: \Lambda^k T_x M \rightarrow
\Lambda^k T_y N$.  Therefore, the operator norm $| \Lambda^k dF_x^* |$ is equal to the operator norm $|\Lambda^k dF_x| \le
\Dil_k(F)$.  Hence $|(F^* \alpha)_x| \le \Dil_k(F) |\alpha_y|$.  \end{proof}

Using this proposition, we can easily bound the degree of a map in terms of its $k$-dilations.  Recall that $S^m$ denotes the unit $m$-sphere.

\begin{prop} The degree of a map $F: S^m \rightarrow S^m$ obeys the bound $| \Deg F | \le \Dil_m F$.
\end{prop}

\begin{proof} Let $\omega$ be the volume form of $S^m$.  We write $| \Deg F | = | (\Vol S^m)^{-1} \int_{S^m} F^* \omega | \le \| F^* \omega \|_\infty
\le \Dil_m F$. \end{proof}

\section{Mappings with small $k$-dilation, the suspension method}

In this paper, we will give two different constructions of homotopically non-trivial maps with
arbitrarily small $k$-dilation.  This section contains a construction that is adapted to homotopy classes
that are suspensions.  The construction is short, and so we describe it right away.

This construction isn't strong enough to make all the mappings from the main theorem in the introduction or from
the h-principle for $k$-dilation stated in the introduction,
but it does give many interesting mappings.  It gives homotopically non-trivial maps from $S^4$ to $S^3$ with
arbitrarily small 3-dilation.  More generally it gives the following proposition.

\begin{prop} \label{cod1mapbaby} If $m \ge 4$, and if $k > (2/3) m$, then there are homotopically non-trivial maps from
$S^m$ to $S^{m-1}$ with arbitrarily small $k$-dilation.
\end{prop}

Eventually we will prove that there are homotopically non-trivial maps $S^m \rightarrow S^{m-1}$ with arbitrarily
small $k$-dilation for all $k > (m+1)/2$, which is the sharp range of $k$.  This proposition gives the sharp range of $k$ for
$m = 4, 5, 6,$ or $7$, but not for $m \ge 8$.  

Later in the paper, we will prove the h-principle for $k$-dilation using a different construction, based on general position arguments.  
This second construction is much longer.

The suspension method can sometimes do better than the general position method we use later.  In particular, the suspension
method can allow us to distinguish different
homotopy classes in the same group $\pi_m(S^n)$.  For example,
consider $\pi_7(S^4)$.  Maps from $S^7$ to $S^4$ with non-trivial Hopf invariant must have
4-dilation at least $c > 0$.  The homotopy group $\pi_7(S^4)$ is isomorphic to $\mathbb{Z} \oplus
\mathbb{Z}_{12}$.  The torsion elements are exactly the elements with Hopf invariant zero.  They
are all suspensions of classes in $\pi_6(S^3)$.  The suspension method applies to
all the torsion elements, and it proves that all of them can be realized by maps
with arbitrarily small 4-dilation.  We will discuss a few more examples below.  The information that we
use about the homotopy groups of spheres and the suspension maps between them may be found in \cite{T},
pages 39-42.

Now we give the construction of maps with arbitrarily small $k$-dilation.

\begin{prop} \label{suspmeth} Suppose that $a \in \pi_m(S^n)$ is the suspension of a homotopy class
in $\pi_p(S^q)$.  Then $a$ can be realized by maps with arbitrarily small $k$-dilation
for any $k > (q/ p) m$.
\end{prop}

\proof Let $a_0$ be a homotopy class in $\pi_p(S^q)$ so that $a$ is the $(m-p)$-fold suspension
of $a_0$.

Let $f_1$ be a map in the homotopy class $a_0$ from
$[0,1]^p$ to the unit q-sphere, taking the boundary of the domain
to the basepoint of $S^q$.  Let $f_2$ be a degree 1 map from
$[0,1]^{m-p}$ to the unit (m-p)-sphere, taking the boundary of the domain
to the basepoint of $S^{m-p}$.  We can assume both maps are smooth,
and we pick a number $L$ which is bigger than the Lipshitz constant
of either map.

Inside of the unit $m$-sphere, we can bilipschitz embed
a rectangle $R$ with dimensions $[0,\epsilon]^p \times
[0,\epsilon^{-\frac{p}{m-p}}]^{m-p}$.  More precisely, there is an embedding
$I$ which is locally $C(m)$-bilipschitz.  (See the appendix in Section \ref{bilipembed} for
a construction of this embedding.)

Now we construct a map $F$ from $R$ to $S^q
\times S^{m-p}$.  The map $F$ is a direct product of a map $F_1$ from
$[0, \epsilon]^p$ to $S^q$ and a map $F_2$ from $[0,
\epsilon^{-\frac{p}{m-p}}]^{m-p}$ to $S^{m-p}$.  The map $F_1$ is just a rescaling
from $[0, \epsilon]^p$ to the unit cube, composed with the map
$f_1$.  Similarly, the map $F_2$ is just a rescaling from $[0,
\epsilon^{-\frac{p}{m-p}}]^{m-p}$ to $[0,1]^{m-p}$, composed with the map $f_2$.  

When $k$ is bigger than $q$, the $k$-dilation of $F$ is less than $(L
\epsilon^{-1})^q (L \epsilon^{\frac{p}{m-p}})^{k-q}$.  Expanding this
expression gives $L^k \epsilon^{-q + (\frac{p}{m-p})(k-q)}$.  The
important part of the expression is the power of $\epsilon$,
which is equal to $(\frac{p}{m-p}) (k-q - \frac{q}{p} (m-p)) = (\frac{p}{m-p})( k - qm/p)$.  We have assumed that $k$
is greater than $qm/p$, and so the exponent of $\epsilon$
is positive.  For $\epsilon$ sufficiently small, the $k$-dilation
of $F$ is arbitrarily small.  

The map $F$ takes the boundary of $R$ to $S^q \vee S^{m-p}$.  We compose
$F$ with a smash map, which is a degree 1 map from $S^q
\times S^{m-p}$ to $S^{m + q - p} = S^n$, taking $S^q \vee S^{m-p}$ to the base
point.  The result is a map from $R$ to $S^{n}$ which takes the
boundary of $R$ to the basepoint.  We can easily extend this map to
all of $S^{m}$ by mapping the complement of $R$ to the basepoint
of $S^{n}$.  The resulting map is homotopic to $a$,
and it has arbitrarily small $k$-dilation.
\endproof

The following special case was stated in the introduction.

\begin{prop} Suppose that $a \in \pi_m(S^n)$ is the suspension of a homotopy class in $\pi_{m-1}(S^{n-1})$, and
that $m > n$.  Then the class $a$ can be realized by maps with arbitrarily small $n$-dilation.
\end{prop}

\begin{proof} This follows from Proposition \ref{suspmeth}.  We take $p = m-1$ and $q = n-1$.  
Since $m > n$, we have $n > \frac{n-1}{m-1} m = (q/p) m$.
\end{proof}

Next we apply Proposition \ref{suspmeth} to some particular homotopy classes.

First we consider maps from $S^m$ to $S^{m-1}$.  We prove Proposition \ref{cod1mapbaby}.

\begin{proof} For $m \ge 4$, the non-trivial homotopy class in $\pi_m(S^{m-1})$ is the (m-3)-fold
suspension of the Hopf fibration from $S^3$ to $S^2$.  Proposition \ref{suspmeth} gives a map in this homotopy class with
arbitrarily small $k$-dilation for $k > (2/3) m$. \end{proof}

Next we consider the torsion classes in $\pi_7(S^4)$.

\begin{prop} Each torsion homotopy class in $\pi_7(S^4)$ can be realized by maps with
arbitrarily small 4-dilation.
\end{prop}

\begin{proof} Each torsion class is the suspension of a class from $\pi_6(S^3)$.  We apply Proposition \ref{suspmeth} 
to get maps with arbitrarily small $k$-dilation for all
$k >  (3/6) 7 = 3.5$. \end{proof}

We remark that one element of $\pi_7(S^4)$ is actually a double suspension of a class in $\pi_5(S^2)$.  
We can realize this element with arbitrarily small $k$-dilation for $k > (2/5) 7 = 2.8$.  In particular, we can realize it by
maps with arbitrarily small 3-dilation.  I don't know whether the other torsion classes can be realized with arbitrarily
small 3-dilation.

There is one (non-zero) element $a \in \pi_8(S^5)$ which is a triple suspension of an element in $\pi_5(S^2)$.  It can be realized
with arbitrarily small 4-dilation, since $4 > (2/5) 8$.  
Recall that $\pi_8(S^5)$ is isomorphic to $\mathbb{Z}_{24}$.  The element $a$ above is the 2-torsion element - it corresponds to $12$ in $\mathbb{Z}_{24}$.  
All of the odd elements are detected by Steenrod squares.  By the Steenrod square inequality, they cannot be realized by maps with arbitrarily 
small $k$-dilation when $k \le 9/2$.  In particular, they cannot be realized with arbitrarily small 4-dilation.

The discussion of $\pi_8(S^5)$ applies more generally to the stable 3-stem.  For all $m \ge 8$, $\pi_m(S^{m-3})$ is isomorphic to $\mathbb{Z}_{24}$, and the suspension maps are isomorphisms.  Therefore, the two-torsion element can be realized by maps with arbitrarily small $k$-dilation for all $k > (2/5) m$.  The odd elements can be realized by maps with arbitrarily small $k$-dilation only if $k > (m+1)/2$.

In the open problems section at the end of the paper, there is some more discussion of which homotopy classes can be realized with arbitrarily small $k$-dilation.

The Freudenthal suspension theorem says that every map in $\pi_m(S^n)$ is a suspension as long as $m \le 2n - 2$.  (See \cite{H}, corollary 4.24 on page 360.)  The condition $m \le 2n -2$ is equivalent to $n > (m+1)/2$.  As long as $n > (m+1)/2$, the Freudenthal suspension theorem implies that every class in $\pi_m(S^n)$ is the suspension of a class from $\pi_{2m - 2n + 1}(S^{m-n+1})$.  Proposition \ref{suspmeth} then implies the following weak form of the h-principle for $k$-dilation.

\begin{prop} \label{weakhprin} If $n > (m+1)/2$, and if $k > \frac{m-n+1}{2m-2n+1} m$, then any map $F_0: S^m \rightarrow S^n$ can be homotoped to a map $F: S^m \rightarrow S^n$ with arbitrarily small $k$-dilation.
\end{prop}

Proposition \ref{weakhprin} is weaker than the h-principle, but for many values of $m$ and $n$ it's not a bad substitute.  For example, if $F_0$ is a map
from $S^{101}$ to $S^{88}$, then Proposition \ref{weakhprin} implies that $F_0$ can be homotoped to a map with arbitrarily small $k$-dilation for all
$k > (14/27) 101 = 52.37...$: in other words, for each integer $k > 52$.  On the other hand, the h-principle says that $F_0$ can be homotoped to a map with arbitrarily small $k$-dilation for each $k > (101 + 1)/2 = 51$.  

The suspension method is based on an observation of Bechtluft-Sachs in \cite{BS}.  
He was interested in the $L^p$ norms $\| \Lambda^k dF \|_p$.  
In \cite{R}, Riviere proved an $L^p$ version of the Hopf invariant inequality.

\newtheorem*{lpho}{$L^p$ bound for the Hopf invariant}
\begin{lpho} For any $C^1$ map  $F: S^{2n-1} \rightarrow S^{n}$, the Hopf invariant of $F$ is controlled
by

$$ |\Hopf(F) | \le C(n) \| \Lambda^n d F \|_{L^{\frac{2n-1}{n}}}^2. $$

In particular, if $F$ has non-zero Hopf invariant, then $\| \Lambda^n d F \|_{L^\frac{2n-1}{n}} > c(n) > 0$.
\end{lpho}

In particular, if $F: S^7 \rightarrow S^4$ with non-zero Hopf invariant, then $\| \Lambda^4 dF \|_{7/4} \ge c > 0$.
Bechtluft-Sachs observed that other homotopy classes in $\pi_7(S^4)$ behave differently.

\newtheorem*{bsex}{Bechtluft-Sachs example}  

\begin{bsex} For each torsion element $a \in \pi_7(S^4)$, 
there is a sequence of homotopically non-trivial maps $F_i: S^7 \rightarrow S^4$ with 
$\| \Lambda^4 d F \|_{7/4} \rightarrow 0$.
\end{bsex}

\section{The Hopf invariant and $k$-dilation}

Our lower bounds on $k$-dilation generalize the following inequality for maps from $S^3$ to $S^2$.

\begin{prop} \label{hopf2dil} If $F$ is a $C^1$ map from the unit sphere $S^3$ to the unit sphere $S^2$, then the Hopf invariant of $F$ is controlled in terms of the 2-dilation of $F$ by the inequality

$$ |\Hopf(F)| \lesssim \Dil_2(F)^2. $$

\end{prop}

In the first part of this section, we review the previous proofs of the proposition from the literature, and we discuss the difficulties of generalizing them to maps from $S^m$ to $S^{m-1}$ for $m > 3$.  In the second part of this section, we give a new proof of the proposition.  Our proof of $k$-dilation estimates for maps $S^m \rightarrow S^{m-1}$ generalizes the new proof for maps from $S^3$ to $S^2$.  

\subsection{Previous estimates about the Hopf invariant and 2-dilation}

The first proof of Proposition \ref{hopf2dil} was based on differential forms.  We begin by describing the Hopf invariant in the language of differential forms.  We let $\omega$ be a 2-form on $S^2$ with $\int_{S^2} \omega = 1$.  The pullback $F^* \omega$ is a closed 2-form on $S^3$, and therefore $F^* \omega$ is exact.  Let $\alpha$ be a 1-form on $S^3$ with $d \alpha = F^* \omega$.  Then the Hopf invariant of $F$ is $\int_{S^3} \alpha \wedge F^* \omega$.  (See \cite{BT}, page 230, for background on the Hopf invariant.)

To get a quantitative estimate for the Hopf invariant, we can go step by step, giving quantitative estimates for each character appearing in the story.  
We sketch the proof here, and give more details in the appendix in Section \ref{prevlower}.  
First, we can choose $\omega$ to be $(4 \pi)^{-1} darea_{S^2}$.  So pointwise, $|\omega| = (4 \pi)^{-1} \le 1$.  The 2-dilation interacts well with 2-forms, 
giving the estimate $\| F^* \omega \|_\infty \le \Dil_2(F) \| \omega \|_\infty \le \Dil_2(F)$ (see Proposition \ref{kdildiff}).  
The main part of the proof is to give estimates for $\alpha$.  This requires some analysis and/or geometry.  For example, using Hodge theory and elliptic estimates, we can find a choice of $\alpha$ with $\| \alpha \|_2 \lesssim \| F^* \omega \|_2 \lesssim \Dil_2(F)$.
With these bounds in hand,

$$ |\Hopf(F)| = \left| \int_{S^3} \alpha \wedge F^* \omega \right| \le \| \alpha \|_2 \| F^* \omega \|_2 \lesssim \Dil_2(F)^2. $$

\noindent It's hard to generalize this argument to maps $S^m \rightarrow S^{m-1}$ when $m > 3$.  For $m > 3$, $\pi_m(S^{m-1}) = \mathbb{Z}_2$.  The homotopy invariant here takes values in $\mathbb{Z}_2$, and I don't know any way to describe it using differential forms.

A second proof of Proposition \ref{hopf2dil} studies the fibers of the map $F$.  By a smoothing argument, we can deform $F$ to a $C^\infty$ map without significantly increasing its 2-dilation.  So we can assume without loss of generality that $F$ is $C^\infty$.  Sard's theorem guarantees that almost every $y \in S^2$ is a regular value.  When $y$ is a regular value, the fiber $F^{-1}(y)$ is a smooth compact 1-manifold (without boundary).  Each regular fiber has a canonical orientation, so each regular fiber is an integral 1-cycle in $S^3$.  Now the Hopf invariant can be described as the linking number of $F^{-1}(y_1)$ and $F^{-1}(y_2)$ for any two regular values of $F$.  Unwinding the definition of linking number this means the following.  Suppose that $\Sigma_1$ is an integral 2-chain in $S^3$ with $\partial \Sigma_1 = F^{-1}(y_1)$.  For almost every choice of $y_2$, $F^{-1}(y_2)$ will be transverse to $\Sigma_1$.  Then the Hopf invariant of $F$ is given by the intersection number of $F^{-1}(y_2)$ with $\Sigma_1$.  (See \cite{BT}, page 227-234 for a discussion of the different definitions of the Hopf invariant and why they are equivalent.)

To get a quantitative estimate for the Hopf invariant, we again go step by step through the argument and give quantitative estimates for each character as it appears.  By the coarea inequality, we can choose $y_1$ so that $\Length[ F^{-1}(y_1) ] \lesssim \Dil_2(F)$.  Next, by the isoperimetric inequality, we can choose $\Sigma_1$ so that $\Area(\Sigma_1) \lesssim \Length[ F^{-1}(y_1) ] \lesssim \Dil_2(F)$.  Now by the coarea inequality again, we can choose $y_2$ so that the number of points in $\Sigma_1 \cap F^{-1}(y_2)$ is $\lesssim \Dil_2(F) [\Area \Sigma_1] \lesssim \Dil_2(F)^2$.  The Hopf invariant is given by counting these intersection points with multiplicities $\pm 1$ determined by the orientations.  Therefore, $|\Hopf(F)|$ is at most the number of intersection points, which is $\lesssim \Dil_2(F)^2$.  (This argument first appeared in \cite{GFRM}.  The details are explained in the short paper \cite{GURH}.)

This argument also does not easily generalize to maps $S^m \rightarrow S^{m-1}$ for $m > 3$.  When $m > 3$, the relevant homotopy invariant cannot be described using a linking number.  It can be described using the fibers of the map $F$, together with their framing.  For a regular value $y \in S^{m-1}$, the fiber $F^{-1}(y)$ is a closed 1-manifold, and its normal bundle gets a framing coming from the isomorphism between the normal bundle and $T S^{m-1}_y$.  The map $F$ is homotopically non-trivial if and only if the framing of the normal bundle has a non-trivial twist.  (See Section \ref{twistedtubes} for a more detailed description.)

I tried to go step by step through the argument with the framing of the normal bundle and give quantitative estimates, but I couldn't make this approach work.  We can begin in the same way, by using the coarea formula to pick a point $y \in S^{m-1}$ so that $\Length[ F^{-1}(y) ] \lesssim \Dil_{m-1}(F)$.  The next character seems to be the framing of the normal bundle and the way that it twists as we move along the fiber $F^{-1}(y)$.  But this twisting depends on the second derivative of $F$, and so there is no way to bound the amount of local twisting in terms of any $\Dil_k(F)$.  Within this setting it's not clear to me what geometric quantity one should try to bound next.  Also notice that in bounding the length of the fiber $F^{-1}(y)$, we only required an estimate for $\Dil_{m-1}(F)$.  A bound of the form $\Dil_{m-1}(F) < \epsilon$ does not by itself imply that $F$ is null-homotopic - we need to invoke somewhere $\Dil_k(F)$ for $k \le (m+1)/2$.  

\subsection{A new method for bounding the Hopf invariant in terms of the 2-dilation} \label{newpfhopf}

Now we give a new proof of Proposition \ref{hopf2dil}, based on the connection between the Hopf invariant and cup products.  If $F: S^3 \rightarrow S^2$, we use $F$ as an attaching map to build a 4-dimensional cell complex $X$: $X := B^4 \cup_F S^2$.  The homotopy type of $X$ is a homotopy invariant of the map $F$.  In particular, the Hopf invariant of $F$ is related to the cup product structure of $X$.  The cohomology groups of $X$ are $H^2(X; \mathbb{Z}) = \mathbb{Z}$ with generator $a$ and $H^4(X; \mathbb{Z}) = \mathbb{Z}$ with generator $b$.  The cup product $a \cup a$ must be a multiple of $b$.  Let $H(X)$ be the integer so that  $a \cup a = H(X) b$.  The integer $H(X)$ is the Hopf invariant of the map $F$.  (See \cite{H} page 427 for a review of the Hopf invariant from this perspective.)

This definition does not seem at first sight like a good setting for a quantitative argument.  How can we connect this story to $\Dil_2(F)$?  What are the geometric quantities related to this story that we should try to estimate?

Cup products are closely connected with Cartesian products.  By unwinding the definition of the cup product in terms of Cartesian products, we can get a formulation which is easier to connect with the geometry of the map $F$.   Here is the formulation.  If $F: S^3 \rightarrow S^2$ is our map, we can look at the map $F \times F: S^3 \times S^3 \rightarrow S^2 \times S^2$.  Let $x_0$ be a basepoint of $S^3$ and $y_0 = F(x_0)$ be a basepoint of $S^2$.

Let $\Diag(S^3) \subset S^3 \times S^3$ denote the diagonal of $S^3 \times S^3$.   Let $\Bouquet(S^3)$ denote the cycle $S^3 \times \{ x_0 \} \cup \{x_0 \} \times S^3$.  Both $\Diag(S^3)$ and $\Bouquet(S^3)$ are integral cycles in $S^3 \times S^3$, and they are homologous.  Therefore, there is an integral 4-chain $Z_0$ with $\partial Z_0 = \Diag(S^3) - \Bouquet(S^3)$.  We consider $F \times F(Z_0)$.  By definition, $F \times F (Z_0)$ is an integral 4-chain in $S^2 \times S^2$, but in fact $F \times F (Z_0)$ is essentially a 4-cycle.  The reason is that $F \times F$ maps $\Diag(S^3)$ into $\Diag(S^2)$ and $\Bouquet(S^3)$ into $\Bouquet(S^2)$.  Therefore, $F \times F$ maps $\partial Z_0$ (which is 3-dimensional) into a 2-dimensional complex.  Therefore, $F \times F (Z_0)$ is essentially an integral 4-cycle in $S^2 \times S^2$.  

If we work with Lipschitz chains, then $F \times F(Z_0)$ is not literally a cycle.  But as we have seen, the boundary of $F \times F (Z_0)$ is a 3-cycle lying in a 2-dimensional polyhedron.  Therefore, we can pick an integral 4-chain $\nu$ with zero volume and with $\partial \nu = \partial \left( F \times F(Z_0) \right)$.  We define $Z(F)$ to be the 4-cycle $F \times F(Z_0) - \nu$.  The cycle $Z(F)$ is connected to the Hopf invariant
by the following proposition.

\begin{prop} \label{Z(F)hopf}   
The homology class of $Z(F)$ in $H_4(S^2 \times S^2; \mathbb{Z})$ is equal to $\Hopf(F) [S^2] \times [S^2]$.
\end{prop}

Using this proposition, we can give a short proof that $|\Hopf(F)| \lesssim \Dil_2(F)^2$.

\begin{proof} The Hopf invariant of $F$ is equal to the homology class of $Z(F) = F \times F(Z_0) - \nu$.  Since $\nu$ has zero 4-volume, we see that $|\Hopf(F)| \lesssim
\Vol_4 (F \times F(Z_0) ) \le \Dil_4(F \times F) \cdot \Vol_4(Z_0)$.   In our construction $Z_0$ does not depend on $F$, so $\Vol_4(Z_0) \lesssim 1$.  

It suffices to check that $\Dil_4(F \times F) \le \Dil_2(F)^2$.   Let $(x, x')$ be a point of $S^3 \times S^3$.  Let $S_1 \ge ... \ge S_4$ be the singular values of $d (F \times F)$ at $(x, x')$.  (So $S_1, S_2, S_3, S_4$ are functions on $S^3 \times S^3$.)   The 4-dilation of $F \times F$ is $\sup_{(x,x') \in S^3 \times S^3} S_1 S_2 S_3 S_4$.  Now let $s_1(x) \ge s_2(x)$ be the singular values of $dF_x$.  Since the derivative $d (F \times F)$ at $(x, x')$ is just $dF_x \times dF_{x'}$, the singular values $S_1, S_2, S_3, S_4$ are equal to the numbers $s_1(x), s_2(x), s_1(x'), s_2(x')$ arranged in decreasing order.  In particular, $S_1 S_2 S_3 S_4 = s_1(x) s_2(x) s_1(x') s_2(x') \le \Dil_2(F)^2$. \end{proof}

Now we prove Proposition \ref{Z(F)hopf}.

\begin{proof}  As we recalled above, $H^2(X; \mathbb{Z}) = \mathbb{Z}$ with generator $a$.  Let $[X]$ be the 
generator of $H_4(X; \mathbb{Z})$, defined so that $b( [X] ) = 1$.  We know that $a \cup a$ evaluated on $[X]$ is the Hopf invariant of $F$.

One definition of the cup product involves the diagonal embedding $\Diag: X \rightarrow X \times X$.  The cup product $a \cup a$ is the pullback of the cross product $a \times a$ defined on $X \times X$.  Therefore, the Hopf invariant of $F$ is the evaluation of $a \times a$ on $\Diag(X)$.

Recall that the space $X$ is formed by attaching $B^4$ to $S^2$ by the map $F: S^3 \rightarrow S^2$.  The space $X$ has a natural basepoint $x$, which is equal to the basepoint of $S^2 \subset X$.  We define $\Bouquet(X) \subset X \times X$ to be $X \times \{ x \} \cup \{x \} \times X$.  Now $S^2 \subset X$, and $Z(F) \subset S^2 \times S^2 \subset X \times X$, so $Z(F)$ is a 4-cycle in $X \times X$.  Next we determine its homology class.

\begin{lemma} \label{Z(F)hom} The 4-cycle $Z(F)$ is homologous to $\Diag(X) - \Bouquet(X)$ as integral 4-cycles in $X \times X$.
\end{lemma}

\begin{proof}

Recall that $Z_0$ is a 4-chain in $S^3 \times S^3$ with $\partial Z_0 = \Diag(S^3) - \Bouquet(S^3)$.  We consider $Z_0 \subset S^3 \times S^3 \subset
\bar B^4 \times \bar B^4$, where $\bar B^4$ denotes the closed 4-ball.  
The boundary of $\Diag(\bar B^{4})$ is $\Diag(S^3)$.  We choose a basepoint of $S^3$ and make it also a basepoint of $\bar B^4$, so that
the boundary of $\Bouquet(\bar B^{4})$ is $\Bouquet(S^3)$.  Therefore, $Z_0 - \Diag(\bar B^{4}) + \Bouquet(\bar B^{4})$ is a 4-cycle 
in $\bar B^4 \times \bar B^4$.  Since $\bar B^4 \times \bar B^4$ is contractible, this cycle is null-homologous, and so there is a 5-chain $Y_0$ with 

$$ \partial Y_0 = Z_0 - \Diag(\bar B^4) + \Bouquet(\bar B^4). $$

Let $\Psi: \bar B^4 \rightarrow X$ be the characteristic map of $B^4$ to $X = B^4 \cup_F S^2$.  In other words, the restriction of $\Psi$
to the boundary of $\bar B^4$ is the attaching map $F: S^3 \rightarrow S^2 \subset X$, and $\Psi$ is the inclusion map from the interior of $B^4$ into
$X$.  We choose base points so that $\Psi$ maps the basepoint of $\bar B^4$ to the basepoint of $X$.  
We consider $\Psi \times \Psi: \bar B^4 \times \bar B^4 \rightarrow X \times X$.  The image $\Psi \times \Psi (\partial Y_0)$ is a null-homologous
cycle in $X \times X$.  This null-homologous cycle is essentially equal to $Z(F) - \Diag(X) + \Bouquet(X)$.  More precisely,

\begin{itemize}

\item $\Psi \times \Psi (Z_0)$ is flat equivalent to $Z(F)$.

\item $\Psi \times \Psi (\Diag(\bar B^4))$ is flat equivalent to $\Diag(X)$.

\item $\Psi \times \Psi (\Bouquet(\bar B^4))$ is flat equivalent to $\Bouquet(X)$.

\end{itemize}

We review flat chains and cycles in Appendix \ref{appenflat}.  We say that two Lipschitz chains are flat equivalent if they define
the same flat chain.  The main point that we need is that if two Lipschitz cycles are flat equivalent, then they are homologous.  This follows 
easily from the definitions, and we review it in the appendix.  Given the three flat equivalences we just mentioned, it follows that 
$Z(F) - \Diag(X) + \Bouquet(X)$ is null-homologous, which is what we wanted to prove.

The three flat equivalences we just mentioned are straightforward.  First, $Z(F) = F \times F(Z_0) - \nu$, where $\nu$ is a chain of volume
zero.  Recall that the restriction of $\Psi$ to $S^3 = \partial B^4$ is just $F$.  Hence the restriction of $\Psi \times \Psi$ to $S^3 \times S^3$
is $F \times F$, and so we see $Z(F) = \Psi \times \Psi(Z_0) - \nu$.  Since $\nu$ has zero volume, $Z(F)$ is flat equivalent to $\Psi \times \Psi(Z_0)$.

Next we consider the chain $\Psi(\bar B^4)$.  If we consider it as a Lipschitz chain, then it is not literally a cycle, but its boundary
lies in $S^2 \subset X$.  Therefore, we can find an 4-chain $\nu'$ of zero volume so that $\Psi(\bar B^4) - \nu'$ is a Lipschitz
cycle in $X$.  It is homologous to the fundamental homology class $[X]$.  Any two Lipschitz cycles in this homology class are flat equivalent (see
the appendix).  So $\Psi(\bar B^4)$ is flat equivalent to the cycle $X$.  Similarly, $\Psi \times \Psi (\Diag(\bar B^4))$ is flat equivalent to $\Diag(X)$,
and $\Psi \times \Psi (\Bouquet(\bar B^4))$ is flat equivalent to $\Bouquet(X)$.  \end{proof}

The Hopf invariant of $F$ is $a \times a ( \Diag(X))$.  By the last lemma, this is equal to $a \times a( Z(F) + \Bouquet(X))$.  For any point $p \in X$, $a \times a ( X \times \{ p \}) = a \times a ( \{ p \} \times X) = 0$.  Therefore, $a \times a (\Bouquet(X)) = 0$.  Hence $\Hopf(F) = a \times a (Z(F))$.  Now $Z(F)$ is a cycle in $S^2 \times S^2 \subset X \times X$.  The restriction of $a$ to $S^2$ is just $[ S^2 ]^*$, the generator of $H^2(S^2; \mathbb{Z})$.  Therefore, $Z(F)$ is homologous to $\Hopf(F) [S^2] \times [S^2]$.  \end{proof}

This argument gives another proof that $|\Hopf(F)| \lesssim \Dil_2(F)^2$.  In the next sections, we will generalize this proof
to homotopically non-trivial maps $S^m \rightarrow S^{m-1}$ for $m > 3$.  When $m > 3$, we will need to consider Steenrod squares instead of cup squares.  In the next section, we review Steenrod squares and define an analogous cycle $Z(F)$ in that setting.

\section{Mappings detected by Steenrod squares}

Suppose that $F: S^m \rightarrow S^n$ is a $C^1$ map.  We can use $F$ as an attaching
map to build a cell complex $X = B^{m+1} \cup_F S^n$.  We assume that $m > n$.  In that case,
the cohomology of $X$ has the following structure: $H^n(X; \mathbb{Z}_2) = \mathbb{Z}_2$ with
generator $a$, $H^{m+1}(X; \mathbb{Z}_2) = \mathbb{Z}_2$ with generator $b$, and $H^d(X; \mathbb{Z}_2)$
vanishes for all other dimensions $d > 0$.  
The Steenrod square $\Sq^{m+1-n}$ maps $H^n(X; \mathbb{Z}_2)$ to $H^{m+1}(X; \mathbb{Z}_2)$.
We define the Steenrod-Hopf invariant of $F$ by the formula

$$ \Sq^{m+1-n}(a) = \SH(F) b. $$

The Steenrod-Hopf invariant takes values in $\mathbb{Z}_2$.  It is a homotopy invariant of the map $F$ (because
homotopic maps $F_1$ and $F_2$ produce homotopy-equivalent complexes $X_1$ and $X_2$).  

We recall some fundamental topological facts about the Steenrod-Hopf invariant.  These facts are explained
in more detail in Hatcher's book \cite{H}, page 489.  This is a very nice reference about Steenrod squares,
containing all the background material we need in this paper.

If $m = 2n-1$, then the Steenrod square
$\Sq^{m+1-n} = \Sq^n$ is the cup square.  In this case $\SH(F)$ is the mod 2 reduction
of the Hopf invariant of $F$.  The Hopf invariant is equal to 1 for the three Hopf fibrations
($S^3 \rightarrow S^2$, $S^7 \rightarrow S^4$, and $S^{15} \rightarrow S^8$).  So 
$\SH(F) = 1$ for the three Hopf fibrations.

Because Steenrod squares behave well with respect to suspensions, the 
Steenrod-Hopf invariant is preserved by suspensions.   Suppose that $\Sigma F: S^{m+1} \rightarrow
S^{n+1}$ is the suspension of $F: S^m \rightarrow S^n$.  The complex
formed by $\Sigma F$, $B^{m+2} \cup_{\Sigma F} S^{n+1}$, is the suspension of $B^{m+1} \cup_F
S^n$.  Since the Steenrod squares commute with the suspension isomorphism, we conclude that
$\SH (\Sigma F) = \SH(F)$.  In particular, $\SH(F) = 1$ for suspensions of the Hopf fibrations.

Therefore, the map $SH: \pi_m(S^n) \rightarrow \mathbb{Z}_2$ is surjective whenever
$n = m-1$ and $m \ge 3$; or $n= m-3$ and $m \ge 7$; or $n = m-7$ and $m \ge 15$.
(By a difficult theorem of Adams, the Steenrod-Hopf invariant is trivial for all other
$m > n$ - see page 490 of \cite{H} and the references therein.)

Our lower bound for $k$-dilation is the following theorem:

\begin{st} Let $F$ be a $C^1$ map from $S^m$ to $S^n$ with $\SH(F) \not= 0$.  If $k \le (m+1)/2$, then
$\Dil_k(F) \ge c(m) > 0$.
\end{st}

We have to review the topological proof that a map with $\SH(F) \not= 0$
is non-contractible and try to organize it in order to get quantitative information about $\Dil_k(F)$.
In the following subsection, we give an alternate description of $\SH(F)$, which we will
be able to connect with $\Dil_k(F)$.  We will check that the alternate definition agrees with
the definition above, which involves reviewing the construction of Steenrod squares.

\subsection{The cycle $Z(F)$} \label{cyclez(f)}

The Steenrod squares are closely related to the following topological operation.  Given
a space $X$ and an integer $i \ge 0$, first consider the product $S^i \times X \times X$.  On this product, there
is a natural $\mathbb{Z}_2$ action, sending $(\theta, x_1, x_2)$ to $(- \theta, x_2, x_1)$.
This action is free and the quotient space is denoted $\Gamma_i X$.  The space $\Gamma_i
X$ is a fiber bundle over $\mathbb{RP}^i$ with fiber $X \times X$.  

The operation $\Gamma_i$ is functorial - if we have a map $F: X \rightarrow Y$, then
there is an induced map $\Gamma_i F: \Gamma_i X \rightarrow \Gamma_i Y$.  The
induced map is defined as follows.  First we map $S^i \times X \times X$ to $S^i
\times Y \times Y$ using the map $id \times F \times F$.  (Here $id$ denotes the
identity map.)  This map is equivariant with respect to the $\mathbb{Z}_2$ action
on the domain and on the range.  Therefore, it descends to a map $\Gamma_i F$ 
between the quotient spaces.

In particular, our map $F: S^m \rightarrow S^n$ induces a map $\Gamma_i F$
from $\Gamma_i S^m$ to $\Gamma_i S^n$ for every $i$.

If $W \subset X$ is a mod 2 cycle, then there are several cycles in $\Gamma_i X$
that we can canonically build from $W$.  One of these is the diagonal cycle
$\Diag(W)$.  In each fiber of $\Gamma_i X \rightarrow \mathbb{RP}^i$, the fiber
of $\Diag(W)$ is a diagonal copy of $W \subset W \times W \subset X \times X$.  A second
example is the bouquet cycle $\Bouquet (W)$.  This is defined canonically as long as $X$
has a basepoint $x$.  In each fiber of $\Gamma_i X \rightarrow \mathbb{RP}^i$, the
fiber of $\Bouquet(W)$ is $W \times \{ x \} \cup \{x \} \times W \subset X \times X$.
If $W$ is a mod 2 d-cycle, then $\Diag(W)$ and $\Bouquet(W)$ are mod 2 (d+i)-cycles.

Remark: If $W$ is an integral d-cycle, it is not necessarily possible to choose orientations 
in order to make $\Diag(W)$ and $\Bouquet(W)$ into integral cycles.  

\begin{lemma} \label{homgi} If $i <  m$, then $H_d(\Gamma_i S^m) = 0$ for $m+i < d < 2m$
(with any coefficient group).
\end{lemma}

\proof There is a natural cell structure
on $\Gamma_i S^m$ which comes from the usual cell structure on $S^m$ (with 2 cells) and
the usual cell structure on $\mathbb{RP}^i$ with $i+1$ cells.  Since $i < m$, 
the cell structure has one cell
in each dimension $0, ..., i$, two cells in each dimension $m, ..., m+i$, and one cell 
in each dimension $2m, ..., 2m + i$.  In particular, we see that $H_d(\Gamma_i S^m)$
vanishes for $m+i < d < 2m$. \endproof

\begin{lemma} If $i < m$, then $\Diag(S^m)$ and $\Bouquet(S^m)$ are homologous - they
belong to the same homology class in $H_{m+i}(\Gamma_i S^m; \mathbb{Z}_2)$.
\end{lemma}

\proof First, the diagonal of $S^m \times S^m$
is homologous to the bouquet $\{ x \} \times S^m \cup S^m \times \{ x \}$, where
$x$ denotes the basepoint of $S^m$.  We let $T_0$ denote a homology between
them.  Now we consider $B^i \times T_0$, which is an (m+i+1)-chain in $B^i
\times S^m \times S^m$.  Here we think of $B^i$ as a hemisphere of $S^i$, so
that we can project $B^i \times S^m \times S^m$ onto $\Gamma_i S^m$.  The
boundary of $B^i \times T_0$ is equal to $B^i \times \Diag(S^m) + B^i \times
\Bouquet(S^m) + V$ where $V$ is a chain in $\partial B^i \times S^m \times S^m$.
Projecting $B^i \times T_0$ into $\Gamma_i S^m$, we get a chain with boundary
$\Diag(S^m) + \Bouquet(S^m) + V'$, where $V'$ is an (m+i)-chain lying in
$\Gamma_{i-1} S^m \subset \Gamma_i S^m$.  Because $\Diag(S^m)$ and $\Bouquet(S^m)$ are
each cycles, $V'$ must also be a cycle.  But as
we saw in Lemma \ref{homgi}, $H_{m+i}(\Gamma_{i-1} S^m) = 0$.  Hence $V'$ is
homologous to zero and the
diagonal $\Diag(S^m)$ is homologous to $\Bouquet (S^m)$. \endproof

At this point, we choose $i = 2n - m -1$.  Since $m > n$, we see that 
$i < n < m$.  Since $i < m$, we can find an (m+i+1)-chain $Z_0$ in $\Gamma_i S^m$ with
$\partial Z_0 = \Diag(S^m) + \Bouquet(S^m)$.  The dimension of the chain $Z_0$ is $m+i+1 = 2n$.

Our map $F: S^m \rightarrow S^n$ induces a map $\Gamma_i F:
\Gamma_i S^m \rightarrow \Gamma_i S^n$.  The map $\Gamma_i F$
maps $\Diag(S^m)$ to $\Diag(S^n)$.  We pick basepoints of $S^m$ and $S^n$ so that
$F$ sends basepoint to basepoint.  With these basepoints, $\Gamma_i F$ maps $\Bouquet(S^m)$ to
$\Bouquet(S^n)$.  Therefore, $\Gamma_i F$ maps $\partial Z_0$ into $\Diag(S^n) \cup \Bouquet(S^n)$.
Now $\partial Z_0$ is a cycle of dimension $i + m$.  On the other hand, $\Diag(S^n) \cup \Bouquet(S^n)$
is a polyhedron of dimension $i + n < i + m$.  Therefore $\Gamma_i F(Z_0)$ is essentially a cycle.

Although $\Gamma_i F (Z_0)$ is not literally a Lipschitz cycle, we have seen that the boundary of $\Gamma_i F (Z_0)$ is an (m+i)-cycle lying in a lower-dimensional polyhedron.  Therefore, we can pick a mod 2 (m+i+1)-chain $\nu$ with zero volume and with $\partial \nu = \partial \Gamma_i F (Z_0)$.  We define $Z(F)$ to be the cycle $\Gamma_i F (Z_0) - \nu$.

Next we study the homology class of $Z(F)$ in $H_{2n}(\Gamma_i S^n; \mathbb{Z}_2)$.  First we calculate this homology group.

\begin{lemma} \label{homgi2} Recall that $i = 2n - m -1$ and $m >n$.  The homology group $H_{2n}(\Gamma_i S^n; \mathbb{Z}_2) = \mathbb{Z}_2$.  The non-trivial homology class is represented by
a fiber $S^n \times S^n$ of the fiber bundle $S^n \times S^n \rightarrow \Gamma_i S^n \rightarrow \mathbb{RP}^i$.  
\end{lemma}

\begin{proof} We use the cell structure of $\Gamma_i S^n$ as in the proof of Lemma \ref{homgi}.  Since $i  = 2n - m - 1 < n$, this cell structure has exactly one cell in dimension $2n$.  Recall that $\Gamma_i S^n$ is a fiber-bundle over $\mathbb{RP}^i$ with fiber $S^n \times S^n$.  The $2n$-cell corresponds to a fiber of the fiber bundle - its closure is a fiber $S^n \times S^n$.  The cell structure also has exactly one $2n + 1$-cell.  Its closure is the restriction of the fiber bundle to a copy of $\mathbb{RP}^1 \subset \mathbb{RP}^i$.  The boundary of this $(2n+1)$-cell gives two copies of the $2n$-cell, and so the boundary operator (working modulo 2) is zero.  \end{proof}

Now we determine the homology class of $Z(F)$ and see how it connects to the Steenrod-Hopf invariant.

\begin{prop} \label{homz(f)=sh(f)} Let $Z_0 \subset \Gamma_i S^m$ be any $2n$-chain with boundary $\Diag(S^m) - \Bouquet(S^m)$, and define
the $2n$-cycle $Z(F) \subset \Gamma_i S^n$ as above.  Then the $2n$-cycle $Z(F)$ is homologous to $\SH(F) [S^n \times S^n]$, where $S^n \times S^n$
is a fiber of the fiber bundle $\Gamma_i S^n$.  In particular, the Steenrod-Hopf invariant $\SH(F)$ is non-zero if and only if
the cycle $Z(F)$ is non-trivial in $H_{2n}(\Gamma_i S^n; \mathbb{Z}_2)$.
\end{prop}

This Proposition is a generalization of Proposition \ref{Z(F)hopf}.

\proof  Recall that $X = B^{m+1} \cup_F S^n$.  The cohomology group $H^n (X; \mathbb{Z}_2)$
is isomorphic to $\mathbb{Z}_2$ and it has generator $a$.  We let $[X]$ be a generator of $H_{m+1}(X; \mathbb{Z}_2) = \mathbb{Z}_2$.  
The Steenrod-Hopf invariant $\SH(F)$ is equal to
the evaluation $\Sq_i a [X]$.  Now we unwind the definition of Steenrod squares to understand this evaluation
better.  We follow the construction of Steenrod squares in Hatcher, \cite{H}, 
pages 501-4.

The class $a$ induces a map $\Phi$ from $X$ to $K(\mathbb{Z}_2, n)$, which is well-defined
up to homotopy.  From now on, we 
abbreviate $K(\mathbb{Z}_2, n) = K$.  Therefore, we get a sequence of maps

$$ \mathbb{RP}^i \times X \rightarrow \Gamma_i X \rightarrow \Gamma_i K. $$

The first map is the diagonal inclusion, and the second map is $\Gamma_i \Phi$.  

The space $K$ comes with a fundamental cohomology class $\alpha \in H^n(K; \mathbb{Z}_2)$,
and $\Phi^* \alpha = a$.  Now in $\Gamma_i K$ there is a (2n)-dimensional cohomology
class $\beta$, whose restriction to each fiber $K \times K$ is $\alpha \times \alpha$ and whose
restriction to $\Bouquet(K) \subset \Gamma_i K$ vanishes.  This element is constructed
in Hatcher, pages 503-504.  (Hatcher constructs the element $\lambda(i)$ on the space $\Lambda K$.  He has already defined a map from $\Gamma_i K$ to $\Lambda K$, and the class $\beta$ is the pullback of $\lambda(i)$ to $\Gamma_i K$).

Let $\omega \in H^1( \mathbb{RP}^i; \mathbb{Z}_2)$ be the non-trivial cohomology class.  We pull back the cohomology class $\beta$  to $\mathbb{RP}^i \times X$, and expand it using the Kunneth formula.  The definition of the Steenrod squares is that this pullback is equal to

$$\sum_{j= 0}^n \omega^j \otimes \Sq_j a. $$

Using the diagram of maps above, we see that 

$$ \SH(F) = \Sq_i a  [ X]  =  \Diag^* \Gamma_i \Phi^* (\beta) [ \mathbb{RP}^i \times X] = \Gamma_i \Phi^* (\beta) [ \Diag(X)].$$

We have an inclusion $S^n \subset X$
and hence $\Gamma_i S^n \subset \Gamma_i X$.  So our cycle $Z(F)$ is a $2n$-cycle
in $\Gamma_i X$.

\begin{lemma}  The cycle $Z(F)$ is homologous to $\Diag(X) - \Bouquet(X)$ in $\Gamma_i X$.
\end{lemma}

This lemma is a generalization of Lemma \ref{Z(F)hom}.

\begin{proof}

Recall that $Z_0$ is a chain in $\Gamma_i S^m$ with $\partial Z_0 = \Diag(S^m) + \Bouquet(S^m)$.  We think of the sphere $S^m$ as the boundary of the closed
ball $B^{m+1}$, and so $\Gamma_i S^m \subset \Gamma_i \bar B^{m+1}$.  Therefore, we can think of $Z_0$ as a chain in $\Gamma_i \bar B^{m+1}$.  
The boundary of $\Diag(\bar B^{m+1})$ is $\Diag(S^m)$ and the boundary of $\Bouquet(\bar B^{m+1})$ is $\Bouquet(S^m)$.  Therefore, $Z_0 - \Diag(\bar B^{m+1}) + \Bouquet(\bar B^{m+1})$ is an (m+i+1)-cycle in $\Gamma_i \bar B^{m+1}$.  Since $\Gamma_i \bar B^{m+1}$ is homotopic to $\mathbb{RP}^i$, this cycle is null-homologous, and so there is a chain $Y_0$ with 

$$ \partial Y_0 = Z_0 - \Diag(\bar B^{m+1}) + \Bouquet(\bar B^{m+1}). $$

Let $\Psi: \bar B^{m+1} \rightarrow X$ be the characteristic map of $B^{m+1}$ to $X = B^{m+1} \cup_F S^n$.  In other words, the restriction of $\Psi$
to the boundary of $\bar B^{m+1}$ is the attaching map $F: S^m \rightarrow S^n \subset X$, and $\Psi$ is the inclusion map from the interior of $B^{m+1}$ into
$X$.  We choose base points so that $\Psi$ maps the basepoint of $\bar B^{m+1}$ to the basepoint of $X$.  
We consider $\Gamma_i \Psi: \Gamma_i \bar B^{m+1} \rightarrow \Gamma_i X$.  The image $\Gamma_i \Psi (\partial Y_0)$ is a null-homologous
cycle in $\Gamma_i X$.  This null-homologous cycle is essentially equal to $Z(F) - \Diag(X) + \Bouquet(X)$.  More precisely,

\begin{itemize}

\item $\Gamma_i \Psi Z_0$ is flat equivalent to $Z(F)$.

\item $\Gamma_i \Psi \Diag(\bar B^{m+1})$ is flat equivalent to $\Diag(X)$.

\item $\Gamma_i \Psi \Bouquet(\bar B^{m+1})$ is flat equivalent to $\Bouquet(X)$.

\end{itemize}

We review flat chains and cycles in Appendix \ref{appenflat}.  We say that two Lipschitz chains are flat equivalent if they define
the same flat chain.  The main point that we need is that if two Lipschitz cycles are flat equivalent, then they are homologous.  This follows 
easily from the definitions, and we review it in the appendix.  Given the three flat equivalences we just mentioned, it follows that 
$Z(F) - \Diag(X) + \Bouquet(X)$ is null-homologous, which is what we wanted to prove.

The three flat equivalences we just mentioned are straightforward.  First, $Z(F) = \Gamma_i F(Z_0) - \nu$, where $\nu$ is a chain of volume
zero.  Recall that the restriction of $\Psi$ to $S^m = \partial B^{m+1}$ is just $F$.  Hence the restriction of $\Gamma_i \Psi$ to $\Gamma_i S^m$
is $\Gamma_i F$, and so we see $Z(F) = \Gamma_i \Psi(Z_0) - \nu$.  Since $\nu$ has zero volume, $Z(F)$ is flat equivalent to $\Gamma_i \Psi(Z_0)$.

Next we consider the chain $\Psi(\bar B^{m+1})$.  If we consider it as a Lipschitz chain, then it is not literally a cycle, but its boundary
lies in $S^n \subset X$, and $n < m$.  Therefore, we can find an (m+1)-chain $\nu'$ of zero volume so that $\Psi(\bar B^{m+1}) - \nu'$ is a Lipschitz
cycle in $X$.  It is homologous to the fundamental homology class $[X]$.  Any two Lipschitz cycles in this homology class are flat equivalent (see
the appendix).  So $\Psi(\bar B^{m+1})$ is flat equivalent to the cycle $X$.  Similarly, $\Gamma_i \Psi \Diag(\bar B^{m+1})$ is flat equivalent to $\Diag(X)$,
and $\Gamma_i \Psi \Bouquet(\bar B^{m+1})$ is flat equivalent to $\Bouquet(X)$.  \end{proof}

We now know that $\SH(F) = \Gamma_i \Phi^*(\beta) [\Diag(X)] = \Gamma_i \Phi^*(\beta) [Z(F) + \Bouquet(X)]$.  
The cohomology class $\Gamma_i \Phi^* (\beta)$ vanishes on $\Bouquet (X)$ because $\Gamma_i \Phi$
maps $\Bouquet(X)$ to $\Bouquet(K)$, and $\beta$ vanishes on $\Bouquet(K)$.  Therefore, $\SH(F) =  \Gamma_i \Phi^* (\beta) [ Z(F)]$.  

The cycle $Z(F)$ lies in $\Gamma_i S^n$.  So next we consider the restriction of $\Gamma_i \Phi^*(\beta)$ to
$\Gamma_i S^n$.  We recall from Lemma \ref{homgi2} that $H^{2n}(\Gamma_i S^n; \mathbb{Z}_2) = \mathbb{Z}_2$, and a non-trivial representative is given by the fiber $S^n \times S^n$.  Recall that $\Gamma_i \Phi^* (\beta)$ restricted to the fiber $X \times X$ is $a \times a$.
Therefore, $\Gamma_i \Phi^*(\beta) (S^n \times S^n) = 1$.  Therefore, $Z(F)$ is homologous to $\SH(F) S^n \times S^n$.  \endproof

To summarize, our definition of the Steenrod-Hopf invariant in terms of the cycle $Z(F)$
agrees with the standard definition in terms of Steenrod squares on $X$.

Describing $\SH(F)$ in terms of the homology class of $Z(F)$ makes it more approachable
geometrically.  Next we will prove estimates about the geometry of $Z(F)$ in terms of $\Dil_k(F)$.
If $k \le (m+1)/2$, and if $\Dil_k(F)$ is sufficiently small, then we will be
able to use these estimates to construct a null-homology of $Z(F)$.

\section{Directed volume} \label{secdirvol}

In this section, we study the geometry of the cycle $Z(F)$ constructed in Section \ref{cyclez(f)}.   We will estimate the volume
of $Z(F)$.  If the volume of $Z(F)$ were sufficiently small, it would follow that $Z(F)$ was null-homologous and that $\SH(F) = 0$.  But it turns out that even if $\Dil_k(F)$ is very small (and $k \le (m+1)/2$), the volume of $Z(F)$ may still be arbitrarily large.  This point is the main difficulty in our proof.  We will get more information about the shape of $Z(F)$ by studying its directed volumes in different directions.  Later we will use this information to show that $Z(F)$ is homologically trivial.  We begin the section by defining directed volumes.

To get some intuition for directed volumes, we start with the simple case that $X$ is a compact submanifold of Euclidean space $\mathbb{R}^N$.  Suppose that $X$ has dimension $d$, and let $J$ be a d-tuple of integers from the set $1, .., N$.  Let $\mathbb{R}^J$ denote the
d-dimensional coordinate plane corresponding to $J$, and let $\pi_J$ denote
the orthogonal projection from $\mathbb{R}^N$ to $\mathbb{R}^J$.  Let $| \pi_J^{-1} (q) \cap X |$ denote the
number of points in $\pi_J^{-1}(q) \cap X$.  Then the $J$-volume of $X$ is given by the formula

$$ \Vol_J(X) := \int_{\mathbb{R}^J} | \pi_j^{-1}(q) \cap X | dq. $$

If $X$ is an oriented submanifold, then we can integrate differential forms over it.  We can
then redefine the $J$-volume as

$$\Vol_J (X) := \sup_{\| w \|_{\infty} \le 1 }  \int_X w(x) dx_J. $$

Next we want to define the directed volume of $C^1$ chains in $\mathbb{R}^N$.  
Suppose that $f$ is a $C^1$ map from the simplex $\Delta^d$ to $\mathbb{R}^N$.  Since the map $f$
may not be an embedding, we need to be slightly more careful in defining $\Vol_J f$.  

Let us recall the definition of $\Vol f$, written in a slightly non-standard way which generalizes for our purposes.  

We define the $k$-dilation of $f$ at a point $x$ by the formula

$$ \Dil_k f(x) = \sup_\omega  | \Lambda^k df_x^* \omega | ,$$

\noindent where the $\sup$ is taken over all $\omega \in \Lambda^k T^* \mathbb{R}^N$ with $| \omega | \le 1$.  

Then if $f: \Delta^d \rightarrow \mathbb{R}^N$ is a $C^1$ map, we define $\Vol_d f := \int_{\Delta} \Dil_d f(x) dx$.

If $J$ is a $k$-tuple of numbers from 1 to $N$, we can define $\Dil_J f(x)$ in a similar way:

$$ \Dil_J f (x) = | \Lambda^k df_x^* e_J^* |, $$

\noindent where $e_J^* \in \Lambda^k T^* \mathbb{R}^N$.  If $J = \{ j_1, ..., j_k \}$, then $e_J^* = e_{j_1}^* \wedge ... \wedge e_{j_k}^*$.  Here
$e_1, ..., e_N$ are the standard orthonormal basis of $T \mathbb{R}^N$ and $e_1^*, ..., e_N^*$ are the dual basis of $T^* \mathbb{R}^N$.  

If $f: \Delta^d \rightarrow \mathbb{R}^N$ and $J$ is a d-tuple, then we can define $\Vol_J f := \int_{\Delta} \Dil_J f(x) dx$.  

Now if $T = \sum_i c_i f_i$ is a mod 2 d-chain, then we define $\Vol_d T = \sum |c_i| \Vol_d f_i$, and $\Vol_J T = \sum_i |c_i| \Vol_J f_i$, where
$|1| = 1$ and $| 0 | = 0$.  (We can do the same for chains with coefficients in a group $G$ as long as we pick a norm on $G$.)

Some of this structure survives to Riemannian manifolds and products of Riemannian manifolds.  

If $f: \Delta^d \rightarrow (M,g)$, then define $\Dil_k f(x) = \sup_\omega |df_x^* \omega|$, where $\omega \in \Lambda^k T^*_{f(x)} M$ and $| \omega | \le 1$.  Then we can define $\Vol_d f := \int_\Delta \Dil_d f(x) dx$.  This agrees with the standard definition of the volume.  We define the volume of a chain $T$ as 
above.

Consider a product manifold $M = A \times B \times C$ with a product Riemannian metric.  (In this paper, we will work with products of three factors, but the definition works equally well with any number of factors.)  If $f$ is a $C^1$ map from $\Delta^d$ to $M$, we define $\Dil_{(a,b,c)} f(x)$ by the formula

$$\Dil_{(a,b,c)} f (x) := \sup | df_x^* (\alpha \wedge \beta \wedge \gamma) |, $$

\noindent where $\alpha \in \Lambda^a T^*A$ with $|\alpha| \le 1$, $\beta \in \Lambda^b T^* B$ with $| \beta | \le 1$, and $\gamma \in \Lambda^c T^* C$ with $| \gamma | \le 1$.  

If $a + b + c = d$ then we can define the $(a,b,c)$-volume of $f$ by $\Vol_{(a,b,c)} (f) := \int_{\Delta} \Dil_{(a,b,c)} f(x) dx$.  We can define
the (a,b,c)-volume of a chain by $\Vol_{(a,b,c)} (\sum c_i f_i) = \sum |c_i| \Vol_{(a,b,c)}(f_i)$.  

\begin{lemma} If $T$ is a d-chain in $(M^N, g)$, and $M$ is a product manifold $A \times B \times C$ with a product Riemannian metric,
then $\Vol_d T$ is comparable to $\sum_{a + b + c = d} \Vol_{(a,b,c)}(T)$.  
\end{lemma}

\begin{proof} It suffices to check that for each $f: \Delta^d \rightarrow M$ and each $x$, $\Dil_d f(x)$ is comparable to $\sum_{a + b + c = d} \Dil_{(a,b,c)} f(x)$.  

It follows from the definition that $\Dil_{(a,b,c)} f(x) \le \Dil_d f(x)$ for each $(a,b,c)$, because $\alpha \wedge \beta \wedge \gamma \in \Lambda^d T^* M$ and has norm $\le 1$.

On the other hand, if $\omega$ is an element of $\Lambda^d T^* M$ with $|\omega| \le 1$, then we can expand $\omega$ as a sum of $C(N)$ 
terms $\alpha_i \wedge \beta_i \wedge \gamma_i$, where $\alpha_i \in \Lambda^{a_i} T^* A$, $\beta_i \in \Lambda^{b_i} T^*B$, etc., and $|\alpha_i|, |\beta_i|, |\gamma_i| \le 1$.  Therefore, $\Dil_d f(x) \le C(N) \sum_{a + b + c = d} \Dil_{(a,b,c)} f(x)$.  \end{proof}

It's worth mentioning the special case of polyhedral chains.  Suppose
first that we triangulate $A$, $B$, and $C$, and take the product polyhedral structure on $M$.
Then each d-cell of the structure is a product of simplices $\Delta^a \times \Delta^b \times \Delta^c$ with $a + b + c = d$, where $\Delta^a \subset A$, etc.  So each d-cell can be assigned a ``direction"
$(a,b,c)$ telling how many dimensions of the cell come from $A$, from $B$, and from $C$.  A polyhedral d-chain is a linear combination of these d-cells.  Since we are working mod 2, we can think of a polyhedral d-chain as just a subset of these cells.  The $(a,b,c)$-volume of a polyhedral chain is just the total volume of all d-cells in $T$ with ``direction" $(a,b,c)$.

The directed volumes $\Vol_{(a,b,c)}(T)$ behave well with respect to product maps.

\begin{lemma} \label{dirvolstre} Suppose $M_1 = A_1 \times B_1 \times C_1$ and $M_2 = A_2 \times B_2 \times C_2$ are Riemannian products.  Suppose that $\Phi: M_1 \rightarrow M_2$ is a product of maps $\Phi = \Phi_A \times \Phi_B \times \Phi_C$, where $\Phi_A: A_1 \rightarrow A_2$, etc.  Suppose that
$T$ is a d-chain in $M_1$ and that $a + b + c = d$.

Then $\Vol_{(a,b,c)}(\Phi(T)) \le (\Dil_a \Phi_A) (\Dil_b \Phi_B) (\Dil_c \Phi_C) \Vol_{(a,b,c)}(T)$.

\end{lemma}

\begin{proof} It suffices to prove this inequality for a map $f: \Delta^d \rightarrow M_1$.  It suffices to prove that $\Dil_{(a,b,c)} (\Phi \circ f) (x)
\le  (\Dil_a \Phi_A) (\Dil_b \Phi_B) (\Dil_c \Phi_C) \Dil_{(a,b,c)} f(x) $.  Let $\alpha \in \Lambda^a T^* A_2$ with $| \alpha | \le 1$ and analogously $\beta$ and $\gamma$.

$$| d (\Phi \circ f)_x^* (\alpha \wedge \beta \wedge \gamma) | = | df_x^* (\alpha' \wedge \beta' \wedge \gamma') |,$$

\noindent where $\alpha' = d \Phi_A^* \alpha$, and analogously $\beta'$ and $\gamma'$.  Now $| \alpha' | \le \Dil_a \Phi_A$, and analogously $|\beta'|$ and $| \gamma'|$.  Therefore, 

$$ | df_x^* (\alpha' \wedge \beta' \wedge \gamma') | \le (\Dil_a \Phi_A) (\Dil_b \Phi_B) (\Dil_c \Phi_C) \Dil_{(a,b,c)} f(x).$$

\end{proof}

Finally, we adapt this idea to twisted products $\Gamma_i S^n$.  The double cover of $\Gamma_i S^n$
is $S^i \times S^n \times S^n$.  We take the product of unit sphere metrics on $S^i \times S^n \times S^n$.
We can define the $(a,b,c)$-volume of a chain in $S^i \times S^n \times S^n$.  Let $I$ be the involution of
$S^i \times S^n \times S^n$ defined by

$$ I (\theta, x_1, x_2) = (- \theta, x_2, x_1). $$

Recall that $\Gamma_i S^n$ is the quotient of $S^i \times S^n \times S^n$ by the involution $I$.  
The antipodal map on $S^i$ has no effect on
directional volumes.  Switching the two $S^n$ factors does.  So we see that $\Vol_{(a,b,c)} (I(T))
= \Vol_{(a,c,b)} (T)$.  Therefore, in $\Gamma_i S^n$ we cannot make a meaningful distinction
between $\Vol_{(a,b,c)}$ and $\Vol_{(a,c,b)}$.  But except for this ambiguity, we can define
directional volumes.  For a chain $T$ in $\Gamma_i S^n$, let $\tilde T$ denote the double cover of $T$ in $S^i \times S^n \times S^n$ and define

$$\Vol_{(a,b,c)}(T) := (1/2) \Vol_{(a,b,c)} (\tilde T) = (1/2) \Vol_{(a,c,b)} (\tilde T). $$

The directed volumes behave well with respect to the maps $\Gamma_i F$.

\begin{lemma} \label{dirvolstr2} If $F: S^m \rightarrow S^n$, and $T$ is a d-chain in $\Gamma_i S^m$, then 

$$\Vol_{(a,b,c)} \Gamma_i F (T) \le \Dil_b(F) \Dil_c(F) \Vol_{(a,b,c)} (T).$$

\end{lemma}

\begin{proof} Let $\tilde T$ be the double cover of $T$ in $S^i \times S^m \times S^m$.  Then the double cover of $\Gamma_i F (T)$
is $(id \times F \times F)(\tilde T)$.  Therefore, $\Vol_{(a,b,c)} \Gamma_i F(T)$ is bounded by $ (1/2) \Vol_{(a,b,c)} (id \times F \times F)(\tilde T)$.
By Lemma \ref{dirvolstre}, this is $ \le (1/2) \Dil_b (F) \Dil_c(F) \Vol_{(a,b,c)} \tilde T = \Dil_b (F) \Dil_c (F) \Vol_{(a,b,c)} T$.
\end{proof}

We introduced this language because the directed volumes of $Z(F)$ are related to the $k$-dilation of $F$.

\begin{prop} \label{dirvolz(f)} Let $F: S^m \rightarrow S^n$ be a $C^1$ mapping and let $Z(F)$ be the
mod 2 cycle in $\Gamma_i S^n$ defined in Section \ref{cyclez(f)}.  (Recall that $i = 2n - m -1$.)  Then the
directional volumes of $Z(F)$ are bounded by the following inequality.

$$\Vol_{(a,b,c)}(Z(F)) \le C(m) \Dil_b (f) \Dil_c(f) .$$
\end{prop}

\proof  Recall that $Z(F)$ is $\Gamma_i F (Z_0) + \nu$, where $Z_0$ is a $2n$-chain in $\Gamma_i S^m$ and
$\nu$ is a $2n$-chain with zero volume.  The chain $\nu$ contributes zero directed volume in any direction.
Using Lemma \ref{dirvolstr2}, we see that $\Vol_{(a,b,c)}(Z(F)) \le \Dil_b(F) \Dil_c(F) \Vol_{(a,b,c)}(Z_0)$.  
But $Z_0$ is independent of $F$, and so $\Vol_{(a,b,c)}(Z_0) \le C(m)$.  \endproof

Estimating all the directed volumes of $Z(F)$ allows us to estimate its total volume.

\begin{Cor} If $\Dil_{m+1-n}(F) \le 1$, then the volume of $Z(F)$ is bounded as follows:

$$\Vol Z(F) < C(m) \Dil_{m+1-n} (F)^{\frac{m+1}{m+1-n}}.$$

\end{Cor}

\proof  The possible directed volumes of $Z(F)$ are given by directions
$(a,b,c)$ where $a + b + c = 2n$, and $a \le 2n-m-1$, $b \le n$, and $c \le n$.
From the first inequality, we see that $b+c \ge m+1$.  Since $b,c \le n$ we conclude that
$b$ and $c$
are each at least $m+1 -n$.  Let $D = \Dil_{m+1-n}(F)^{\frac{1}{m+1-n}}$.  Then
$\Dil_b(F) \le D^b$ for all $b \ge m+1-n$ (by Proposition \ref{dilkdill}).  
So for every $(a,b,c)$,
$\Vol_{(a,b,c)}(Z(F)) \le C(m) D^b D^c$.  Since $D \le 1$, this quantity is $\le C(m) D^{m+1}$, which is the
right-hand side.  Since every directed volume obeys the desired bound, so
does the total volume. \endproof

Our bound for the volume of $Z(F)$ has the following corollary connecting $\SH(F)$ with some $k$-dilations of $F$.

\begin{Cor} If the Steenrod Hopf invariant $\SH(F)$ is non-zero, then $\Dil_{m+1-n}(F) \ge c(m) > 0$.

\end{Cor}

\proof If $\Dil_{m+1-n}(F)$ is very small, then the volume of $Z(F)$ is very small.  By the Federer-
Fleming deformation theorem it follows that $Z(F)$ is null-homologous.  Hence $\SH(F) = 0$. \endproof

This estimate is much weaker than the one we want to prove, but it still has some content.  For example,
if $F$ is a homotopically non-trivial map from $S^m$ to $S^{m-1}$, then the corollary says that $\Dil_2(F)$
is bounded below.  The sharp theorem says that $\Dil_{k}(F)$ is bounded below for $k \le
\frac{m+1}{2}$.  The corollary gives the sharp value of $k$ when $m = 3$ or $4$, but not when $m \ge 5$.
If we look at a map from $S^m$ to $S^{m-3}$, with non-trivial Steenrod-Hopf
invariant,  then the corollary says
that $\Dil_4(F)$ is bounded below.  The theorem says that $\Dil_k(F)$ is bounded below
for all $k \le \frac{m+1}{2}$.  The corollary gives the sharp value of $k$ when $m = 7,8$, but not when $m \ge 9$.  

This corollary includes our first new lower bound on $k$-dilation: a map $F: S^8 \rightarrow S^5$ with non-zero Steenrod-Hopf invariant must have $\Dil_4(F) \ge c > 0$.  

Now we suppose that $\Dil_k(F)$ is tiny for some $k \le (m+1)/2$, and
we wish to prove that $Z(F)$ is null-homologous.  We cannot bound the total volume of $Z(F)$, but we can
bound the volume in some directions.

\begin{lemma} \label{contdir} Suppose that $\Dil_k(F) \le 1$.  If $b$ and $c$ are each $\ge k$, then
the directed volume $\Vol_{(a,b,c)} Z(F)$ is  $\lesssim \Dil_k(F)^2$.
\end{lemma}

\proof By Proposition \ref{dirvolz(f)}, the directed volume $\Vol_{(a,b,c)} Z(F)$ is $\lesssim 
\Dil_b(F) \Dil_c(F)$.  Since $b,c \ge k$, we know that $\Dil_b(F)^{1/b} \le \Dil_k(F)^{1/k}$ (by Proposition \ref{dilkdill}).  
Since $\Dil_k(F) \le 1$, we see that $\Dil_b(F)$
and $\Dil_c(F)$ are both $\le \Dil_k(F)$. \endproof

We call a direction $(a,b,c)$ bad if $| b-c | \le 1$ and good if $| b-c | \ge 2$.  If $k \le (m+1)/2$, then the directed volume of $Z(F)$
in the bad directions is controlled by the following lemma.

\begin{lemma} \label{dirvolbad} If $k \le (m+1)/2$, and if $(a,b,c)$ is a bad direction in $\Gamma_i S^n$, then $b$ and $c$ are $\ge k$.
Therefore, if $\Dil_k F \le 1$, then

$$ \Vol_{(a,b,c)} (Z(F)) \lesssim \Dil_k(F)^2. $$

\end{lemma}

\begin{proof} Once we know that $b,c \ge k$, then the estimate follows from Lemma \ref{contdir}.
Since $Z(F)$ is a cycle of dimension $2n$, $a + b + c = 2n$.  We know that
$a \le i = 2n - m -1$.  Therefore, $ b+c \ge m+1$.  Since $(a,b,c)$ is a bad direction, $|b -c | \le 1$. 
It is just an elementary computation to check that $b$ and $c$ are at least $k$.

There are two cases depending on whether $m$ is even or odd.

If $m$ is even, then $k \le m/2$.  We know that $2b + 1 \ge b+c \ge m+1$, and we see that $b \ge m/2 \ge k$.  By a symmetrical
argument, $c \ge k$.  

If $m$ is odd, and $b+c = m+1$, then we must have $b = c$.  In this case, $b = c = (m+1)/2 \ge k$.  If $m$ is odd
and $b + c \ge m+2$, then we see that $2 b+1 \ge b+c \ge m+2$, and so $b \ge (m+1)/2 \ge k$.  By a symmetrical argument, $c \ge k$.

\end{proof}

We know that $Z(F)$ has only a small volume in bad directions, and we want to prove that $Z(F)$ is null-homologous.  As
a toy problem, let's consider a polyhedral $2n$-cycle $X$ with zero volume in the bad directions.  The next proposition shows
that such a cycle is null-homologous.  To make the situation precise, suppose we choose any triangulation of $S^n$.  Then let us
choose any triangulation of $S^i$ which is invariant with respect to the antipodal map.  Taking the product, we get a polyhedral structure on
$S^i \times S^n \times S^n$ which is invariant with respect to our involution $I(\theta, x_1, x_2) = (- \theta, x_2, x_1)$.  So it descends
to give a polyhedral structure on $\Gamma_i S^n$.  Each face in the polyhedral structure has a direction $(a,b,c)$ well defined up to the equivalence $(a,b,c) \sim (a,c,b)$.  Each face is either good or bad (because $|b-c| \le 1$ if and only if $|c-b| \le 1$).  

\begin{prop} \label{nobadmodel}
Let $X$ be a polyhedral (2n)-cycle in $\Gamma_i S^n$ which does not contain any $2n$-faces in bad directions.
Then $X$ is homologically trivial.
\end{prop}

\proof Let $\tilde X$ be the double cover of $X$ in $S^i \times S^n \times S^n$.  So $\tilde X$
is a polyhedral cycle with no faces in the bad directions.  The good directions $(a,b,c)$
all have $b \not= c$, so we can divide them into two categories: the directions where $b < c$, and
the directions where $b > c$.  We let $\tilde X_1$ be the chain given by adding all the faces
of $\tilde X$ where $b < c$, and we let $\tilde X_2$ be the chain given by adding the faces
where $b > c$.   So $\tilde X = \tilde X_1 + \tilde X_2$.

Now comes the crucial point.  Because of the estimate $|b-c|\ge 2$ for good directions, the
boundaries $\partial \tilde X_1$ and $\partial \tilde X_2$ are {\it disjoint}!  If we consider a face $\Box_1$ in
$\tilde X_1$ lying in direction $(a_1, b_1, c_1)$, then we know that $b_1 \le c_1 - 2$.
Now consider a face of $\partial \Box_1$, and say that it lies in direction $(\bar a_1, \bar b_1, \bar c_1)$.  The vector
$(\bar a_1, \bar b_1, \bar c_1)$ can be found by taking $(a_1, b_1, c_1)$ and subtracting 1 from one of the three
 entries.  Therefore, $\bar b_1 < \bar c_1 $.  We repeat the analysis for a face $\Box_2$ in $\tilde X_2$.  It has
 direction $(a_2, b_2, c_2)$ with $b_2 \ge c_2 + 2$.  A face in the boundary of $\Box_2$ has direction $(\bar a_2, \bar b_2,
 \bar c_2)$, and $\bar b_2 > \bar c_2$.  Hence $\partial \tilde X_1$ and $\partial \tilde X_2$ have no faces in common.  Since $\partial
\tilde X_1 + \partial \tilde X_2 = \partial \tilde X = 0$, we see that $\tilde X_1$ and $\tilde X_2$
are each cycles!

Now $\tilde X$ is a double cover of $X$.  Each face of $X$ lifts to two faces of $\tilde X$,
one lying in $\tilde X_1$ and one lying in $\tilde X_2$.  For example, if $X$ is 8-dimensional, 
a face of $X$ in the direction $(2,2,4) = (2,4,2)$ lifts to two faces of $\tilde X \subset S^2 \times S^4
\times S^4$, one in direction $(2,2,4)$ and one in direction $(2,4,2)$.  So we see that the projection
of $\tilde X_1$ onto $X$ is a degree 1 map.

But it's easy to check that $\tilde X_1$ is null-homologous in $S^i \times S^n \times S^n$.  Since
$\tilde X_1$ has dimension $2n$, and $i < n$, $\tilde X_1$ could be homologically non-trivial
only if its projection to $S^n \times S^n$ had non-zero degree.  But $\tilde X_1$ is a (2n)-cycle 
with no volume in the $(0, n, n)$ direction - so its projection to $S^n \times S^n$ has measure
zero. \endproof

The cycle $X$ in the last proposition had no volume in the bad directions.  Our actual cycle $Z(F)$ has a
small but non-zero volume in the bad directions.  The proposition does not directly apply to $Z(F)$, but one may
still hope that $Z(F)$ is rather similar to $X$.  In the next two sections, we will modify the above argument to show that $Z(F)$ is homologically trivial.

The key point in the proof of Proposition \ref{nobadmodel} was that the double cover of $X$ split into separate cycles $\tilde X_1$ and $\tilde X_2$.  The double cover of $Z(F)$ will not literally split into two separate cycles.  Instead, the double cover will look like two large pieces joined by a thin bridge.  We will have to prove a suitable estimate about the shape of the bridge.

The estimate on the shape of the bridge involves fairly hard work.  I made some attempt to find a softer argument.  For example, I tried to find a way to approximate the cycle $Z(F)$ by a polyhedral cycle $X$ with no volume in the bad directions.  But I couldn't find any way to do this.  It's still not clear to me whether this fairly hard work is necessary...

In the next section, we formulate an inequality - called the perpendicular pair inequality - that allows us to control the geometry of the thin bridge.

\section{Perpendicular Pair Inequality}

\newtheorem*{perppair}{Perpendicular Pair Inequality}

\begin{perppair} Suppose that $z$ and $w$ are mod 2 $(n-1)$-cycles in $\mathbb{R}^N$, and suppose that $y$
is an $n$-chain with $\partial y = z + w$.  Finally, suppose that $z$ and $w$ are ``perpendicular"
to each other in the following sense: for any coordinate $(n-1)$-tuple $J$, either $\Vol_J(z) = 0$ or
$\Vol_J(w) = 0$.

Then, we can find a chain $y'$ with $\partial y' = z$ and with Hausdorff content $\HC_n(y') \le C(n,N) \Vol_n (y)$.

In addition, $y'$ lies in the $R$-neighborhood of $z$ for $R \le C(n,N) \Vol_n(y)^{1/n}$.

\end{perppair}

The directed volume $\Vol_J(z)$ is defined in Section \ref{secdirvol}.

It's an open question whether we can bound the volume of $y'$ by $C(n,N) \Vol_n(y)$.  A bound on the Hausdorff content is weaker
than a bound on the volume.  For our application to $k$-dilation estimates, this Hausdorff content estimate turns out
to be just as useful as a volume estimate would have been.

Here is an outline of this section.
First we give a review of Hausdorff content.  In particular, we prove that a cycle with sufficiently small
Hausdorff content is homologically trivial.  Next we give the proof of our $k$-dilation lower bound using
the perpendicualr pair inequality.  At the end, we give some more context for the perpendicular pair inequality
by comparing it with an open problem in geometric measure theory raised by L. C. Young in the early 60's.

We prove the perpendicular pair inequality in the next section.

\subsection{Review of Hausdorff content}

Let $E$ be a subset of Euclidean space $\mathbb{R}^N$ or of some Riemannian
manifold.  The Hausdorff contents of $E$ measure how difficult it is to cover $E$ with
balls.  To compute the d-dimensional Hausdorff content of $E$, denoted $\HC_d(E)$,
we consider all covers of $E$ by (countably many) balls: $E \subset \cup_i B(x_i, R_i)$.
The ``cost"
associated to a given cover is $\sum_i R_i^d$.  The infimal cost over all covers is 
the d-dimensional Hausdorff content of $E$.

In our case $E$ will be an $n$-dimensional chain or cycle.  The Hausdorff content obeys
$\HC_n(E) \le C(n) \Vol_n(E)$, which one proves by covering $E$ by small balls.  On the
other hand, the Hausdorff content of $E$ may be much smaller than the volume, especially
if $E$ is ``crumpled up" so that a medium ball can cover a large volume of $E$.  

There is a version of the Federer-Fleming deformation theory using Hausdorff
content in place of volume.  We need the following result:

\begin{prop} \label{fedflemhcont} Let $(M^N,g)$ be a compact Riemannian manifold.  For any dimension $n \le N$, there is a constant
$\epsilon > 0$, depending on $M^N$, $g$, and $n$, such that every $n$-cycle $z$ in $M$ with $\HC_n(z) < \epsilon$
is homologically trivial.  (The proposition holds for homology with any coefficients.)
\end{prop}

\proof The proof is based on the Federer-Fleming pushout lemma - adapted to Hausdorff content.

\begin{lemma} Let $\Delta^N$ denote a unit equilateral Euclidean simplex, and let $E \subset
\Delta$ denote a set.  For each point $p$ in the interior of $\Delta$, let $\pi_p$ denote 
the outward radial projection from $\Delta - \{ p \}$ to $\partial \Delta$.  Let $\Delta_{1/2} \subset \Delta$ denote
a concentric simplex of one half the side-length, centered at the center of mass of $\Delta$.
For any dimension $0 \le d \le N$, the following inequality holds

$$\Average_{p \in \Delta_{1/2}} \HC_d [\pi_p(E \setminus \{p\}) ]  \le C(N) \HC_d(E). $$

Also, if $\HC_d(E)$ is sufficiently small, then we can choose $p$ outside of $E$ so that

$$\HC_d [ \pi_p(E) ] \le C(N) \HC_d(E). $$

\end{lemma}

\proof If $d > N-1$, then $\HC_d (\partial \Delta) = 0$, and the inequality is trivially true.  So we can assume $d \le N-1$.

Let $\{ B(x_i, R_i) \}$ denote any cover of $E$ with balls.  Let $p \in \Delta_{1/2}$.  Consider a ball $B(x,R)$ and define the radius
 $R' =  C(N) \Dist(p,x)^{-1} R $.  We claim that if $C(N)$ is large enough, then 
the outward projection $\pi_p [ B(x, R) \setminus \{ p \} ]$ is contained in
$B(\pi_p(x), R')$.  

First we consider the case that $\Dist(p, x) \le 10 R$.  In this case, $R' > 1$, and so $B(x', R')$ contains $\partial \Delta \supset \pi_p 
[ B(x, R) \setminus \{ p \}]$.  So we can now suppose that $\Dist(p, x) > 10 R$.

Let $ y \in B(x,R)$ $\sigma \subset B(x,R)$ be a segment from $x$ to $y$.  Then $\pi_p(\sigma)$ is a (piecewise smooth) curve in $\partial \Delta$ from
$\pi_p(x)$ to $\pi_p(y)$.  We will prove that the length of $\pi_p(\sigma)$ is $\le R'$.
Suppose that $z$ and $z'$ are on $\sigma$ and that $\pi_p(z)$ and $\pi_p(z')$ lie in the same $(N-1)$-face $F \subset \partial \Delta^N$.  
It suffices to prove that $\Dist( \pi_p(z), \pi_p(z')) \lesssim \Dist(p,x)^{-1} \Dist(z,z')$.  

Consider the triangle $T$ with vertices $p, \pi_p(z),$ and 
$\pi_p(z')$.  The angle of $T$ at the vertex $p$ is the angle between the rays $[p,z]$ and $[p, z']$.  
Because $z,z' \in B(x,R)$ and $\Dist(p,x) \ge 10R$ this angle is $\lesssim \Dist(p,x)^{-1} \Dist(z,z')$.  The lengths of all sides of $T$ are $\lesssim 1$.
Because $p \in \Delta_{1/2}$ the segments $[p, \pi_p(z)]$ and $[p, \pi_p(z')]$ hit the face $F$ at an angle $\gtrsim 1$.  Therefore, the other two
angles of $T$ are each $\gtrsim 1$.  It now follows by trigonometry that the length of $[\pi_p(z), \pi_p(z')]$ is at most a constant factor times
the angle at the vertex $p$.  This proves the desired bound.

Now let $B(x,R)$ be a ball, and let $R' = C(N) \Dist(x,p)^{-1} R$, as above.  Since $d \le N-1$, 

$$ \Average_{p \in \Delta_{1/2}} (R')^d = \Average_{p \in \Delta_{1/2}} C(N) R^d \Dist(x,p)^{-d} \lesssim R^d. $$

Consider a covering of $E$ by balls $B(x_i, R_i)$.  This inequality holds for each ball $B(x_i, R_i)$ in the covering.  
 Therefore, the average value of $\sum_i (R_i')^d$ is 
$\le C(N) \sum_i R_i^d$.  And so the average value of
$\HC_d[ \pi_p (E \setminus \{ p \} ) ]$ is bounded by $C'(N) \HC_d(E)$. 

If $d \le N$ and $\HC_d(E)$ is sufficiently small, then the measure of $E$ is less than half the measure of $\Delta_{1/2}$.  Therefore, we can choose $p \in \Delta_{1/2} \setminus E$ so that $\HC_d( \pi_p(E) ) \le C(N) \HC_d(E)$.  

\endproof

Now we prove the proposition.  Pick a triangulation of $(M^N, g)$.  The metric $g$ restricted to each simplex
is $L$-bilipschitz to the unit equilateral Euclidean simplex, where $L$ is a constant depending on $(M^n, g)$.   
Let $z \subset M^N$ be an $n$-cycle with $\HC_n(z) < \epsilon$.  We use the push-out
lemma on each $N$-face of $M^N$ to homotope $z$ into the $(N-1)$-skeleton of $M^N$.  Then we use it again on each (N-1)-face of $M$ to push $z$ into the (N-2)-skeleton, and so on, until
$z$ is pushed into the $n$-skeleton.  Each homotopy may increase the $n$-dimensional Hausdorff
content by a constant factor, so we end with a cycle of Hausdorff content $ < C \epsilon$.  Now if $\epsilon$
is too small, this cycle does not cover any $n$-simplex of $(M^N, g)$, and so it is null-homologous.  \endproof

\subsection{The proof of the $k$-dilation lower bound}

We now give the proof of the $k$-dilation lower bound using the perpendicular pair inequality.  We will 
prove the perpendicular pair inequality in the next section.

\begin{st} Let $F$ be a $C^1$ map from $S^m$ to $S^n$.  If the Steenrod-Hopf invariant
$\SH(F)$ is non-zero, and $k \le (m+1)/2$, then $\Dil_k(F) \ge c(m) > 0$.
\end{st}

\proof We recall the setup from Section \ref{cyclez(f)}.  Let $Z_0$ be any mod 2 chain in $\Gamma_i S^m$ 
with boundary $\Diag(S^m) + \Bouquet(S^m)$.  Then $\Gamma_i F(Z_0)$ is essentially a cycle $Z(F)$ in $\Gamma_i S^n$.  
According to Proposition \ref{homz(f)=sh(f)}, the homology class of $Z(F)$ determines $\SH(F)$.  In particular, $\SH(F) = 0$
if and only if $Z(F)$ is null-homologous.

We fix $ k \le (m+1)/2$.  We may assume that $\Dil_k(F) \le \epsilon = \epsilon(m)$.  If $\epsilon$ is sufficiently small, we have to show
that $\SH(F) = 0$.  It suffices to show that $Z(F)$ is null-homologous.  In Section \ref{secdirvol}, we proved some estimates about the directed volumes of
$Z(F)$.  According to Lemma \ref{contdir}, if $b, c \ge k$, then $\Vol_{(a,b,c)} (Z(F)) \lesssim \Dil_k(F)^2 \lesssim \epsilon^2$.  A direction $(a,b,c)$ is called bad if $|b - c| \le 1$.  According to Lemma \ref{dirvolbad}, $\Vol_{(a,b,c)} (Z(F)) \lesssim \Dil_k(F)^2 \lesssim \epsilon^2$ for each bad direction.

Also we showed in Proposition \ref{nobadmodel} that if $X$ is a polyhedral $2n$-cycle in $\Gamma_i S^n$ with zero volume in the bad directions, then $X$ is null-homologous.  Our plan is to modify the proof of Proposition \ref{nobadmodel}.  The main issue is that $Z(F)$ has a small but non-zero volume in the bad directions.  We will solve this main issue using the perpendicular pair inequality.  A minor issue is that $Z(F)$ is not a polyhedral cycle.  In order to set things up well, we have to slightly modify the definition of $Z(F)$ so it has something like a polyhedral structure.
We begin by doing this small modification.

Earlier, we didn't think about how to choose $Z_0$, but now let's consider that point.  We will choose a triangulation
of $S^m$ and a $\mathbb{Z}_2$-invariant triangulation of $S^i$.  Taking the product of these, we get a $\mathbb{Z}_2$-invariant polyhedral structure on $S^i \times S^m \times S^m$, and taking the quotient gives a polyhedral structure on $\Gamma_i S^m$.

Now it would be very convenient if we could choose $Z_0$ to be a polyhedral chain with respect to this polyhedral structure.  But this is impossible, because $\Diag(S^m)$ is not a polyhedral cycle, and $\partial Z_0$ needs to be $\Diag(S^m) + \Bouquet(S^m)$.  The best we can do is to write $Z_0$ as a polyhedral chain plus a small chain.  

For any $\delta > 0$, we do the following construction.  We choose a polyhedral structure on $\Gamma_i S^m$ as above, using fine triangulations of
$S^i$ and $S^m$.  (The fine triangulation needs to depend on $\delta$.)  Then we let $Z_\delta$ be a $2n$-chain with $\partial Z_\delta = \Diag(S^m) + \Bouquet(S^m)$, obeying the following estimates.  The volume of $Z_\delta$ is $\lesssim 1$ (independent of $\delta$).  Most of $Z_\delta$ is polyhedral with respect to our triangulation.  The remainder of $Z_\delta$ has volume $\le \delta$.  In other words, $Z_\delta = Z_\delta' + Z_\delta''$ where $Z_\delta'$ is polyhedral and $Z_\delta''$ has volume $\le \delta$.  We can find such a $Z_\delta$ by taking a chain $Z_0$ as above, choosing sufficiently fine triangulations, and applying the deformation theorem.

As in Section \ref{cyclez(f)}, $\Gamma_i F (Z_\delta)$ is essentially a $2n$-cycle in $\Gamma_i S^n$.  More precisely, the boundary of $\Gamma_i F(Z_\delta)$ lies in the lower-dimensional set $\Bouquet(S^n) \cup \Diag(S^n)$. Therefore, there is a $2n$-chain $\nu_\delta$ in $\Gamma_i S^n$ with volume zero and with $\partial \nu_\delta = 
\partial \left[ \Gamma_i F(Z_\delta) \right]$.  We define $Z_\delta(F)$ to be the sum $\Gamma_i F(Z_\delta) + \nu_\delta$.  Now $Z_\delta(F)$ is
a $2n$-cycle in $\Gamma_i S^n$.  Because $Z_\delta$ is a chain with boundary $\Bouquet(S^m) + \Diag(S^m)$, Proposition \ref{homz(f)=sh(f)} 
says that 
$Z_\delta(F)$ is null-homologous if and only if $\SH(F) = 0$. Assuming $\epsilon$ and $\delta$ are sufficiently small, we have to prove 
that $Z_\delta(F)$ is null-homologous.

The point of this small modification is that $Z_\delta$ consists mostly of polyhedral faces, and polyhedral faces are easier to analyze. 
Next we divide $Z_\delta(F)$ into pieces in good and bad directions.

Recall that $Z_\delta = Z_\delta' + Z_\delta''$, where $Z_\delta'$ is polyhedral and $Z_\delta''$ has volume $< \delta$.  
If $Q = \Delta^a \times \Delta^b \times \Delta^c$ is a polyhedral face of $\Gamma_i S^m$, we say that $Q$ is good if $|b-c| \ge 2$
and bad if $|b-c| \le 1$.  (Here $\Delta^a$ is a simplex in $S^i$, and $\Delta^b$ and $\Delta^c$ are simplices in $S^m$.)
We let $Z_\delta(good)$ be the union of the good faces in $Z_\delta'$.  We define $G$ to be $\Gamma_i F(Z_\delta(good))$.  The chain
$G$ is the ``good part'' of $Z_\delta(F)$.  We define $B = Z_\delta(F) - G$.  So $G$ and $B$ are $2n$-chains in $\Gamma_i S^n$ with $Z_\delta(F) = G + B$.

\begin{lemma} If $\epsilon$ and $\delta$ are sufficiently small, then we can guarantee that $\Vol_{2n} B$ is as small as we like.
\end{lemma}

\begin{proof} The chain $B$ has a few pieces, but they are each easy to bound.  We let $Z_\delta'(bad)$ be the union of all bad faces in $Z_\delta'$.  Now $B$ is equal to the following sum:

$$B = \Gamma_i F (Z_\delta'(bad)) + \Gamma_i F (Z_\delta'') + \nu_\delta. $$

The first term is the most interesting.  If $Q$ is a face in $Z_\delta'(bad)$, then $Q$ lies in direction $(a,b,c)$ with $a + b + c = 2n$ and $|b-c| \le 1$.  By Lemma \ref{dirvolstr2}, $\Vol \Gamma_i F(Q) \le \Dil_b (F) \Dil_c(F) \Vol (Q)$.  Since $Q$ is bad, Lemma \ref{dirvolbad} implies that $b,c \ge k$.  
Now since $\Dil_k(F) \le \epsilon \le 1$, we have $\Dil_b(F) \Dil_c(F) \le \epsilon^2$.  Hence $\Gamma_i F (Z_\delta'(bad))$ has volume $\le \epsilon^2 \Vol Z_\delta' \lesssim \epsilon^2$.

The volume of $\Gamma_i F (Z_\delta'') \le \Dil_1(F)^{2n} \delta$.  By making $\delta$ sufficiently small, we can make this term as small as we like.

Finally, $\nu_\delta$ has volume zero.

\end{proof}

So our cycle $Z_\delta(F)$ has the form $G + B$, where $G$ is a chain with volume only in the good directions and $B$ is small.  Next we study the double cover of $Z_\delta(F)$.  To simplify the notation, we abbreviate $Z_\delta(F)$ by $Z$ for the rest of the argument.

Now we consider the double cover $\tilde Z \subset S^i \times S^n \times S^n$.  We have $\tilde Z = \tilde G + \tilde B$, where $\tilde G$ is the double cover of $G$ and $\tilde B$ is the double cover of $B$.  

If $Q$ is a face in $Z_\delta'$, then the double cover of $Q$ in $S^i \times S^m \times S^m$ consists of two faces, $\tilde Q_1$ and $\tilde Q_2$.  Each of these faces is a product of simplices.  Suppose that $\tilde Q_1$ is a product $\Delta^{a_1} \times \Delta^{b_1} \times \Delta^{c_1}$ with $\Delta^{a_1} \subset S^i$, $\Delta^{b_1}$ in the first copy of $S^m$, and $\Delta^{c_1}$ in the second copy of $S^m$.  Then the other preimage face, $\tilde Q_2$, is $I (\tilde Q_1)$, where $I$ is our involution of $S^i \times S^m \times S^m$.  The face $\tilde Q_1$ is in direction $(a_1,b_1,c_1)$, and the face $\tilde Q_2$ is in direction $(a_2,b_2,c_2)=(a_1,c_1,b_1)$. 

If $Q$ is a good face then $|b_1 - c_1| = |b_2 - c_2| \ge 2$.  By choosing our labels appropriately, we can assume that $b_1 \le c_1 - 2$ and $b_2 \ge c_2 + 2$.  

The chain $\tilde G$ in $S^i \times S^n \times S^n$ is a sum of contributions from the good faces $Q \subset Z_\delta'$ as follows:

$$ \tilde G = \sum_{Q \subset Z_\delta, Q \textrm{ good }} (id \times F \times F) (\tilde Q_1) + (id \times F \times F)(\tilde Q_2). $$

So we can divide $\tilde G$ into two pieces as follows.

$$ \tilde G_1 = \sum_{Q \subset Z_\delta, Q \textrm{ good }} (id \times F \times F) (\tilde Q_1). $$

$$ \tilde G_2 = \sum_{Q \subset Z_\delta, Q \textrm{ good }} (id \times F \times F)(\tilde Q_2). $$

Let us compare our situation with the situation in the proof of Proposition \ref{nobadmodel}.  In Proposition \ref{nobadmodel}, we have a polyhedral $2n$-cycle $X$ in $\Gamma_i S^n$, with no volume in the bad directions.  We consider the double cover $\tilde X$ in $S^i \times S^n \times S^n$, and we divide it into pieces $\tilde X = \tilde X_1 + \tilde X_2$, where $\tilde X_1$ consists of the faces in good directions $(a,b,c)$ with $b < c$ and $\tilde X_2$ consists of faces in the good directions with $c < b$.  The cycle $Z$ is analogous to $X$.  The chains $\tilde G_1$ and $\tilde G_2$ are analogous to $\tilde X_1$ and $\tilde X_2$.  The chain $\tilde G_1$ lies only in good directions $(a,b,c)$ with $b< c$, and the chain $\tilde G_2$ lies only in good directions $(a,b,c)$ with $c < b$.  But the situation is more complicated because $Z$ has a small non-zero volume in bad directions.  So we have $\tilde Z = \tilde G_1 + \tilde G_2 + \tilde B$, and we know that $\tilde B$ has small volume.

In the proof of Proposition \ref{nobadmodel}, the crucial point was that $\tilde X_1$ and $\tilde X_2$ were each cycles.  In our case, $\tilde G_1$ and $\tilde G_2$ are not cycles.  Instead, I like to imagine $\tilde Z$ as two large pieces ($\tilde G_1$ and $\tilde G_2$) connected by a thin bridge ($\tilde B$).  We would like to cut out $\tilde B$ and separately cap off $\tilde G_1$.  In other words, we would like to find a small chain $Y_1$ so that $\tilde G_1 + Y_1$ make a cycle.  Then we can use this cycle the way we used $\tilde X_1$ in the proof of Proposition \ref{nobadmodel}.

This is the most delicate part of our argument.  We already know that $\tilde B$ has small volume, but this is not enough to be able to find a small $Y_1$.  For example, imagine that $\tilde Z$ was a large sphere, $\tilde B$ was a small neighborhood of the equator, $\tilde G_1$ was the part of the Northern hemisphere to the North of $\tilde B$ and $\tilde G_2$ was the part of the Southern hemisphere to the South of $\tilde B$.  Then $\tilde B$ may have arbitrarily small volume, and yet the boundary of $\tilde G_1$ cannot be filled in with a small chain. In order to find a small cap $Y_1$, we need more geometric information than just a bound on the volume of $\tilde B$.

The key point is that $\partial \tilde G_1$ and $\partial \tilde G_2$ are perpendicular, which allows us to apply the perpendicular pair inequality.  The perpendicular pair inequality exactly gives us the small chain $Y_1$ that we need.

Here are the details.  We consider $S^i \times S^n \times S^n \subset \mathbb{R}^{i+1} \times \mathbb{R}^{n+1} \times \mathbb{R}^{n+1} = \mathbb{R}^N$.  So we can think of $\tilde B$, $\partial \tilde G_1$, $\partial \tilde G_2$, etc. as chains and cycles in $\mathbb{R}^N$.  We have to check that $\partial \tilde G_1$ and $\partial \tilde G_2$ are perpendicular in the sense defined in the perpendicular pair inequality.

If $J$ is a set of numbers from $1$ to $N$, let $a(J)$ denote the number of directions in $J$ from the first factor $\mathbb{R}^{i+1}$, let $b(J)$ denote the number of directions of $J$ from the second factor $\mathbb{R}^{n+1}$, and let $c(J)$ denote the number of directions in $J$ from the third factor $\mathbb{R}^{n+1}$.  

\begin{lemma} If $T_0$ is a d-chain in $S^i \times S^m \times S^m$, and $T = (id \times F \times F)(T_0) \subset \mathbb{R}^N = \mathbb{R}^{i+1} \times \mathbb{R}^{n+1} \times \mathbb{R}^{n+1}$,
then 

$$ \Vol_{J}(T) \le \Dil_{b(J)}(F) \Dil_{c(J)}(F) \Vol_{(a(J),b(J),c(J))}(T_0) . $$

\end{lemma}

\begin{proof} Consider the product structure $\mathbb{R}^N = \mathbb{R}^{i+1} \times \mathbb{R}^{n+1} \times \mathbb{R}^{n+1}$.  Using this product structure, we can define $\Vol_{(a,b,c)}(T)$ for a d-chain $T$ in $\mathbb{R}^N$.

The directed volume $\Vol_J T \le \Vol_{(a(J), b(J), c(J))}(T)$, which follows by plugging in the definitions.

Applying Lemma \ref{dirvolstre}, we see that

$$ \Vol_{(a(J), b(J), c(J))}(T) \le \Dil_{b(J)}(F) \Dil_{c(J)}(F) \Vol_{(a(J),b(J),c(J))}(T_0). $$
\end{proof}

\begin{lemma} Let $\Delta^a \times \Delta^b \times \Delta^c$ be a face of our triangulation of $S^i \times S^m \times S^m$. 
If
$\Vol_J (id \times F \times F)(\Delta^a \times \Delta^b \times \Delta^c) > 0$, then $a(J) = a$, $b(J) = b$, and $c(J) = c$.  
\end{lemma}

\begin{proof} Note that $\Vol_{(a', b', c')} (\Delta^a \times \Delta^b \times \Delta^c) > 0$ only if
$(a', b', c') = (a,b,c)$.  Applying the previous lemma finishes the argument. \end{proof}

The boundary of $\tilde G_1$ is the sum $\sum_{Q} (id \times F \times F) (\partial \tilde Q_1)$, where the sum
goes over all the good faces $Q \subset Z_\delta$.  Pick a particular face $\tilde Q_1$, lying in direction $(a_1, b_1, c_1)$
with $b_1 \le c_1 - 2$.  Suppose a face of $\partial \tilde Q_1$ has direction $(\bar a_1, \bar b_1, \bar c_1)$.  The direction $(\bar a_1, \bar b_1, \bar c_1)$ is obtained by subtracting 1 from one of the three entries in the vector $(a_1, b_1, c_1)$.  Therefore, $\bar b_1 < \bar c_1$.
So if $\Vol_J(\partial \tilde G_1) > 0$, then we must have $b(J) < c(J)$.   But by the same argument, if $\Vol_J
(\partial \tilde G_2) > 0$, then we must have $b(J) > c(J)$.  For any $(2n-1)$-tuple $J$, either $\Vol_J(\partial \tilde G_1) = 0$
or  $\Vol_J(\partial \tilde G_2) = 0$.  So the two cycles $\partial \tilde G_1$ and $\partial \tilde
G_2$ are perpendicular in the sense of the perpendicular pair inequality.

Now we can apply the perpendicular pair inequality.  We let $\partial \tilde G_1$ and $\partial \tilde G_2$ play the roles of
$z$ and $w$, and we let $\tilde B$ play the role of $y$.  The hypotheses of the perpendicular pair inequality are satisfied
because $\partial \tilde G_1$ and $\partial \tilde G_2$ are perpendicular
and $\partial \tilde B = \partial \tilde G_1 + \partial \tilde G_2$.  Also note that $\Vol_{2n} (\tilde B)$ is as small as we like.
The perpendicular pair inequality tells us that there is a chain $Y \subset \mathbb{R}^N$ with $\HC_{2n}(Y)$
as small as we like and $\partial Y = \partial \tilde G_1$.  The chain $Y$ may not be contained in $S^i \times S^n \times
S^n$, but it is contained in the $R$-neighborhood of $S^i \times S^n \times S^n$ for $R \lesssim \Vol_{2n}(\tilde B)^{\frac{1}{2n}}$.
Since $R$ is tiny, we may retract $Y$ into $S^i \times S^n \times S^n$ without changing its Hausdorff content
much.  Hence there is a mod 2 chain $Y_1 \subset S^i \times S^n \times S^n$ with $\HC_{2n}(Y_1)$ tiny
and $\partial Y_1 = \partial \tilde G_1$.

Now we let $\tilde Z_1 = Y_1 + \tilde G_1$.  We note that $\tilde Z_1$ is a mod 2 (2n)-cycle in $S^i \times S^n \times S^n$.  
We claim that the cycle $\tilde Z_1$ is homologically trivial in $S^i \times S^{n} \times S^{n}$.  Since $\tilde Z_1$ 
is a (2n)-cycle, we just have to check that its projection to $S^{n} \times S^{n}$ has degree zero.  The 
projection of $\tilde G_1$ to $S^{n} \times S^{n}$ has measure zero, because the direction
$(0, n, n)$ is a bad direction.  On the other hand, $Y_1$ has tiny (2n)-dimensional
Hausdorff content, so the projection of $Y_1$ to $S^{n} \times S^{n}$ has tiny volume.  Hence
the projection of $\tilde Z_1$ to $S^{n} \times S^{n}$ is not surjective and has degree zero.  So
we see that $\tilde Z_1$ is homologically trivial in $S^i \times S^{n} \times S^{n}$.

Now let $\pi: S^i \times S^{n} \times S^{n} \rightarrow \Gamma_i S^{n}$ be the double
cover map.  Clearly $\pi (\tilde Z_1)$ is homologically trivial.  Now we break up the original cycle
$Z_\delta(F) = Z$ as a sum of cycles: $Z = \pi (\tilde Z_1) + \left( Z - \pi (\tilde Z_1) \right)$.  The first summand is homologically trivial.  Recall that
$Z = G + B$.  
Now $\pi (\tilde Z_1) = \pi (\tilde G_1 + Y_1) = G + \pi (Y_1)$.  So $Z - \pi (\tilde Z_1) = B - \pi (Y_1)$.
The chain $B$ has tiny volume and the chain $\pi (Y_1)$ has tiny $2n$-dimensional Hausdorff content.  Hence $Z - \pi (\tilde Z_1)$ has tiny $2n$-dimensional Hausdorff content.  By Proposition \ref{fedflemhcont}, it follows that $Z - \pi(\tilde Z_1)$ is null-homologous.
Therefore $Z = Z_\delta(F)$ is null-homologous.  Therefore $\SH(F) = 0$.  \endproof

\subsection{Context for the perpendicular pair inequality}

The perpendicular pair question is similar to a well-known open problem raised by L. C. Young in the 1960s.
In \cite{YL}, Young constructed an integral 1-cycle $z$ in $\mathbb{R}^4$ with the following
surprising property.  There is an integral 2-chain $y$ with $\partial y = 2 z$ and with area 2, but any integral chain $y'$ with $\partial y' = z$ has area strictly bigger than 1.  In fact, any integral chain $y'$ with $\partial y' = z$ has area $> 1.3$.   Notice that $y/2$ is a real chain with $\partial (y/2) = z$ and with mass 1.  But in Young's example, any integral chain $y'$ with $\partial y' = z$ has mass $> 1.3 > 1 = \Mass(y/2)$.  

Young raised the question of how large this effect could be.

\newtheorem*{2tprob}{Young's problem}

\begin{2tprob} Suppose that $z$ is an integral $(n-1)$-cycle in $\mathbb{R}^N$.  Suppose that $y$
is an integral $n$-chain with $\partial y = 2z$.  Does it follow that there is another integral chain
$y'$ with $\partial y' = z$ and $\Mass(y') \le C(n,N) \Mass(y)$?
\end{2tprob}

The perpendicular pair problem looks similar to Young's problem.  We can put them in a common framework as follows.
Suppose that $\partial y = z - w$.  Can we find a chain $y'$ with $\partial y' = z$ and with the size of $y'$
comparable to the size of $y$?  In general, the answer is certainly no.  For example, we may have $z = w$ and $y=0$.
But if $z$ and $w$ are very different from each other, it seems intuitive that filling $z$ and $w$ separately may
be approximately as good as filling $z-w$.  In Young's problem, $w = - z$.  In the perpendicular pair problem, we
know that $w$ and $z$ are perpendicular.  We can formulate a version of Young's problem for perpendicular pairs.

\newtheorem*{perppairprob}{Perpendicular Pair Problem}

\begin{perppairprob} Suppose that $z$ and $w$ are (integral or mod 2) $(n-1)$-cycles in $\mathbb{R}^N$, and suppose that $y$
is an $n$-chain with $\partial y = z + w$.  Finally, suppose that $z$ and $w$ are ``perpendicular"
to each other in the following sense: for any coordinate $(n-1)$-tuple $J$, either $\Vol_J(z) = 0$ or
$\Vol_J(w) = 0$.

Does it follow that there is a chain $y'$ with $\partial y' = z$ and $\Vol_n(y') \le C(n,N) \Vol_n (y)$?

\end{perppairprob}

It seems to me that these problems are closely related.  Young's problem is difficult, and I believe that the perpendicular
pair problem is difficult also.

\section{Proof of the perpendicular pair inequality}

In this section, we prove the perpendicular pair inequality.  First we recall the statement.

\begin{perppair} Suppose that $z$ and $w$ are mod 2 $(n-1)$-cycles in $\mathbb{R}^N$, and suppose that $y$
is an $n$-chain with $\partial y = z + w$.  Finally, suppose that $z$ and $w$ are ``perpendicular"
to each other in the following sense: for any coordinate $(n-1)$-tuple $J$, either $\Vol_J(z) = 0$ or
$\Vol_J(w) = 0$.

Then, we can find a chain $y'$ with $\partial y' = z$ and with $\HC_n(y') \le C(n,N) \Vol_n (y)$.

Also, $y'$ lies in the $R$-neighborhood of $z$ for $R \le C(n,N) \Vol_n(y)^{1/n}$.

\end{perppair}

This inequality can probably be extended to integral cycles or mod p cycles, but we only need the mod 2 case.  Focusing
on mod 2 makes the exposition a little bit cleaner, because we don't have to keep track of signs.

\subsection{The thick region}

For any number $\alpha > 0$, and any ball $B = B(x,R) \subset \mathbb{R}^N$, 
we say that $y$ is $\alpha$-thick in $B$ if
$\Vol(y \cap B) \ge \alpha R^n$.  Otherwise, we say that $y$ is $\alpha$-thin in $B$.
Now the thick region $T_\alpha(y)$ is defined to be the union of all the balls $B$ where
$y$ is $\alpha$-thick.

A standard covering argument shows that $T_\alpha(y)$ has controlled Hausdorff content.

\begin{lemma} For any $\alpha > 0$, 

$$\HC_n [ T_\alpha(y) ] \le 5^n \alpha^{-1} \Vol_n (y) . $$

\end{lemma}

\proof The set $T_\alpha(y)$ is the union of all thick balls.  By the Vitali covering lemma, we can
find disjoint thick balls $B_i$ so that $5 B_i$ covers $T_\alpha(y)$.  Hence $\HC_n(T_\alpha(y))$
is bounded by $\sum_i (5 R_i)^n$, where $R_i$ denotes the radius of $B_i$.  But since each
$B_i$ is $\alpha$ thick, $R_i^n \le \alpha^{-1} \Vol (y \cap B_i)$.  Since the $B_i$ are disjoint,
we see that $\sum_i (5 R_i)^n \le 5^n \alpha^{-1} \Vol(y)$. \endproof
 
\subsection{Outline of the construction}

Our construction is based on applying the deformation theorem to $z$ at a dyadic sequence of scales.

By a minor approximation argument, we can reduce to the case that $z$, $w$, and $y$ are all cubical chains
in the cubical lattice with some tiny scale $s_0$.  We give this approximation argument in Section \ref{minorapprox}.
For now, we give the proof of the perpendicular pair inequalities for the case of cubical chains.

Then we consider a dyadic sequence of scales
$s_i = 2^i s_0$.  We use the deformation theorem to deform $z$ to a cubical cycle at each scale.
We let $z_i$ be a Federer-Fleming deformation of $z$ at scale $s_i$.  (So $z$ itself is $z_0$.)  
Each cycle $z_i$ is a finite sum of cubical $(n-1)$-faces of the lattice with side length $s_i$.  We will prove
that when $i$ is sufficiently large, $z_i$ is just the zero cycle.  Let us define $i_{final}$ so
that $z_{i_{final}} = 0$.

Next we build a sequence of $n$-chains $A_i$ with $\partial A_i = z_{i-1} - z_i$.
We define the chain $y'$ as $y' = \sum_{i=1}^{i_{final}} A_i$.  An easy calculation shows
that $\partial y' = z_0 - z_{i_{final}} = z$.

Our main goal is to do this construction in such a way that each $z_i$ and each $A_i$ is contained in $T_\alpha(y)$, for
some $\alpha > 0$ depending only on the dimension $N$.  (This requires a slightly modified version of the Federer-Fleming
deformation theory adapted to the situation.)
Then the Hausdorff content of $y'$ will be bounded by the Hausdorff content of $T_\alpha(y) \lesssim \Vol_n(y)$.  

This outline is based on arguments from \cite{Y}.  In \cite{Y}, R. Young uses a multiscale argument of this type 
to prove isoperimetric inequalities on the Heisenberg group.

In this section we write $A \lesssim B$ to mean $A \le C(N) B$.

\subsection{Intersection number lemma}

In this section, we use the
perpendicular hypothesis to bound some intersection numbers.

Let $R$ denote a rectangle of dimension $N - n + 1$ parallel to the coordinate axes.  If $z$ is transverse
to $R$, then we can define the mod 2 intersection number $[ z \cap R ] \in \mathbb{Z}_2$ as the number of points in the
intersection $z \cap R$ taken modulo 2.  

\newtheorem*{inl}{Intersection number lemma}

\begin{inl} Let $\alpha > 0$ and $s > 0$ be any numbers.
Let $z$ and $w$ be a perpendicular pair of $(n-1)$-cycles.  Let $\partial y = z + w$.  
Let $R_0$ be an axis parallel rectangle with dimension $N-n+1$.  Suppose that
the sidelengths of $R_0$ are at most $s$.  

Let $v$ be a vector of length at most $s$, and let $R_v$ denote the translation of $R_0$
 by $v$.  We will pick a possible translation vector
$v \in B^N(s)$ randomly (with respect to the usual volume form on the ball).

Suppose that the ball around the center of $R_0$ with radius $N s$ is 
$\alpha$-thin.  Then the intersection number $[z \cap R_v]$ is equal to 0 most of the time.
More precisely, the set of vectors $v \in B^N(s)$ so that $[z \cap R_v] \not= 0$ has
probability measure at most $C(n,N) \alpha$.
\end{inl}

\proof  The intersection number $[ (z+w) \cap R_v]$ is equal to the intersection number
$[ y \cap \partial R_v]$.  The boundary $\partial R_v$ consists of $2^{N-n+1}$ faces, which
are each rectangles contained in the ball of radius $Ns$ around the center of $R_0$.  Since $y$
is small, it is usually disjoint from all these faces.  By standard integral geometry, the
probability that $y$ intersects any of the faces of $\partial R_v$ is $\le C(n,N) \alpha$.  

Here are details of the integral geometry argument.  
Let $F$ denote a face of the boundary of $R_0$.  Let $F_v$ denote the translation of $F$ by $v$, which is a face of
the boundary of $R_v$.  Let $F^{\perp}$ denote the plane perpendicular to
$F$ and let $\pi$ denote the orthogonal projection from $\mathbb{R}^N$ to $F^\perp$.  
The projection $\pi(F_v)$ is a single point.  Let $B$ denote the ball of radius $Ns$
around the center of $R_0$.  All faces $F_v$ are contained in this ball.  By assumption, the volume
of $y \cap B$ is at most $C(n,N) \alpha s^n$.  Therefore, the projection $\pi(y \cap B)$
has volume at most $C(n,N) \alpha s^n$.  Now if $y$ intersects $F_v$, then the point $\pi(F_v)$ must
lie in $\pi(y \cap B)$.  We note that $\pi(F_v)$ is just $\pi(F) + \pi(v)$.  So $F_v$ intersects $y$ only if
$\pi(v)$ is contained in the small set $\pi(y \cap B) - \pi(F)$.  The set of $v$ obeying this condition has
probability at most $C(n,N) \alpha$.  

With high probability, $0 = [y \cap \partial R_v] = [ (z+w) \cap R_v ] = [ z \cap R_v ] + [w \cap R_v ]$.
But the two intersection numbers $[z \cap R_v]$ and $[w \cap R_v]$ can 
(almost) never cancel because $z$ and $w$ are perpendicular cycles.  
Let $J$ denote the $n-1$ coordinates that are perpendicular to $R_0$.  Note that if $\Vol_J(z) = 0$,
then $z$ is disjoint from $R_v$ for almost every $v$.  By the perpendicularity assumption, we
know that either $\Vol_J(z) = 0$ or $\Vol_J(w) = 0$.  Hence either $[z \cap R_v] = 0$ for almost
every $v$ or else $[w \cap R_v] = 0$ for almost every $v$.

Therefore, $[z \cap R_v] = 0$ except with probability $C(n, N) \alpha$.   \endproof

\subsection{The deformation operator} \label{defop}

In this section, we review the Federer-Fleming deformation operator.  The deformation operator is defined
in terms of intersection numbers.  Therefore, the intersection number lemma will allow us to prove estimates
about the deformations of $z$.

The Federer-Fleming construction is based on the skeleta of lattices and
their dual skeleta.  Let $\Sigma(s)$ be the cubical lattice at scale $s$ in $\mathbb{R}^N$.  We let
$\Sigma^d(s)$ be the d-skeleton of $\Sigma(s)$.  

Let $\bar \Sigma(s)$ be the dual cubical lattice.  Here dual means that each
vertex of $\bar \Sigma$ is the center of an N-face of $\Sigma$, while each 
vertex of $\Sigma$ is the center of an N-face of $\bar \Sigma$.  Each edge
of $\bar \Sigma$ passes through the center point of a unique (N-1)-face of $\Sigma$ etc.  For
any d-dimensional face $F^d \subset \Sigma^d$, we let $\bar F^{N-d}$ denote
the corresponding (N-d)-dimensional face of the dual skeleton $\bar \Sigma^{N-d}$.  Note that
$F$ and $\bar F$ always have the same center point.

For a vector $v$, we let $\bar \Sigma_v$ denote the translation of $\bar \Sigma$ by
$v$.  If $F$ is a d-face of $\Sigma^d$ and $\bar F$ is the corresponding face of
$\bar \Sigma$, we let $\bar F_v$ denote the translation of $\bar F$ by the vector
$v$.

With these two skeleta set up, we can define the 
Federer-Fleming deformation operator.  For each $v$, the deformation operator
takes any d-chain $T$ transverse to $\bar \Sigma_v$, and outputs $D_v(T)$ a cubical d-chain 
contained in $\Sigma^d$.  The deformation operator is defined as follows:

$$ D_v(T) : = \sum_{F^d \subset \Sigma^d} [\bar F^{N-d}_v \cap T] F . $$

Recall that $[ \bar F_v \cap T] \in \mathbb{Z}_2$ is the topological intersection number of
$\bar F_v$ and $T$.

The deformation $D_v(T)$ depends on the scale $s$.  I think it would clutter the notation too much to write
something like $D_v^s(T)$.  It will always be clear from the context which scale $s$ is being used.

We now recall some standard facts about the deformation operator.  We review the proofs of the standard facts in
Section \ref{fedflemreview}.

1. If $|v| < s/2$, and if $T$ is a cubical d-chain in $\Sigma(s)$, then $D_v(T) = T$.

2. The deformation operator commutes with taking boundaries:

$D_v( \partial T) = \partial D_v(T)$.

(provided that $\bar \Sigma_v$ is transverse to both $\partial T$ and $T$ so that
both sides of the equation are defined.)

3.  If we average over all $|v| < s/2$, then

$$ \Average_v  \Vol_d [ D_v(T) ] \le C(d, N) \Vol_d(T). $$

4. If $T$ is a d-cycle, then we can build a (d+1)-chain $A_v(T)$ in the $C(N) s$ neighborhood of $T$ with $\partial A_v(T) = T - D_v(T)$.  Moreover, if we average over all $|v| < s/2$, then

$$ \Average_v \Vol_{d+1} [ A_v(T)] \le C(d,N) s \Vol_{d}(T). $$

\vskip10pt

The intersection number lemma gives some estimates about the cycle $D_v(z)$.

\begin{lemma} \label{warmup1}  Let $\alpha > 0$ be any number.  
Let $F$ be a face of $\Sigma(s)$.  Suppose that $F$ is not
contained in $T_{\alpha}(y)$.  Then, as we consider all $|v| < s/2$, the
probability that $F$ is contained in $D_v(z)$ is at most $C(n,N) \alpha$.
\end{lemma}

\proof The face $F$ is contained in $D_v(z)$ if and only
if the intersection number $[\bar F_v \cap z]$ is non-zero. 
Since $F$ is not contained in $T_\alpha(y)$, it follows that the ball
around the center of $F$ with radius $N s$ is $\alpha$-thin.  The center of $F$ is
the same as the center of $\bar F$, so the ball around the center of $\bar F$ with radius $N s$ is $\alpha$-thin.   
Now we apply the intersection number
lemma with $\bar F$ playing the role of the rectangle $R_0$.  The intersection number lemma implies that the probability that
$[\bar F_v \cap z] \not= 0$ is at most $C(n,N) \alpha$.  \endproof

It would have been helpful if $D_v(z)$ were completely contained in $T_\alpha(y)$ for some
dimensional constant $\alpha(N) > 0$.  We could then choose $s = s_i$ and define $z_i = D_v(z)$,
and we would know that $z_i \subset T_\alpha(y)$ for each $i$, accomplishing a big chunk of the plan
laid out in the outline.  Unfortunately, Lemma \ref{warmup1} is not strong enough to imply this.

The problem is that just translating the lattice $\bar \Sigma$ does not give
us enough degrees of freedom to find a deformation $D_v(z)$ with all the
properties that we would like.  We will improve the situation by moving each
vertex of $\bar \Sigma(s)$ {\it independently}.  
This will involve not just translating $\bar \Sigma$ but bending
it.

\subsection{Federer-Fleming deformations using bent dual skeleta}

Let $\Phi: \bar \Sigma(s) \rightarrow \mathbb{R}^n$ be a PL or piecewise smooth map.  We call $\Phi$ a 
``bending" of the dual skeleton.
The deformation operator associated to $\Phi$ is a small modification of the standard deformation operator.

$$D_\Phi (T) := \sum_{F^d \subset \Sigma^d(s)} [\Phi (\bar F) \cap T] F. $$

Notice that $D_\Phi(T)$ is a cubical chain in $\Sigma^d(s)$ - we do not bend or 
translate $\Sigma(s)$.  The deformation operator $D_\Phi$ is defined as long as $T$
is transverse to $\Phi(\bar \Sigma)$.

Our next goal is to construct bending functions $\Phi_i: \bar \Sigma(s_i) \rightarrow \mathbb{R}^N$ in such a way that the deformations
$D_{\Phi_i}(z) = z_i$ are contained in the thick region $T_\alpha(y)$.   This will take some work.  We record
here an important property of the deformation operator $D_\Phi$.

\begin{lemma} The deformation operator $D_{\Phi}$ commutes with boundaries.  In other words, if $T$ is any d-chain,
and $\Phi$ is transverse to both $T$ and $\partial T$, then

$$ \partial D_\Phi(T) = D_\Phi ( \partial T) .$$

\end{lemma}

\begin{proof} From the formula for $D(T)$, we see that 

$$\partial D(T) = \sum_{F^d \subset \Sigma^d(s)} [\Phi (\bar F) \cap T] \partial F.$$

Consider a (d-1)-face $G$ in $\Sigma^{d-1}$.  Let $F_1(G), ..., F_{2(N-d+1)}(G)$ be the set of all the d-faces of $\Sigma^d(s)$ that contain $G$ in their boundary.  We can rewrite the formula for $\partial D(T)$ as follows:

$$ \partial D(T) = \sum_{G^{d-1} \subset \Sigma^{d-1}(s)} \left( \sum_{j=1}^{2(N-d+1)} [\Phi( \overline{F_j(G)}) \cap T] \right) G. $$

Now the first key point is that $\sum_{j=1}^{2(N-d+1)} \overline{F_j(G)} = \partial \bar G$.  Therefore,

$$ \partial D(T) = \sum_{G^{d-1} \subset \Sigma^{d-1}(s)} [\Phi( \partial \bar G) \cap T] G. $$

 Since $\Phi$ is transverse to $T$, $\Phi(\bar G) \cap T$ is a
1-chain, and the boundary of $\Phi(\bar G) \cap T$ consists of an even number of points.  By transversality, the boundary of $\Phi(\bar G) \cap T$ is the union of $\Phi( \partial \bar G) \cap T$ and $\Phi(\bar G) \cap \partial T$.
Therefore, $[\Phi (\partial \bar G) \cap T] = [\Phi (\bar G) \cap \partial T]$.  Substituting in, we get

$$\partial D(T) =  \sum_{G^{d-1} \subset \Sigma^{d-1}(s)} [\Phi(\bar G) \cap \partial T] G = D( \partial T). $$

\end{proof}

We have to construct useful bending maps $\Phi_i: \bar \Sigma(s_i) \rightarrow \mathbb{R}^N$.  If a face $F$ is
not in the thick region $T_\alpha(y)$, then we want $[\Phi_i( \bar F) \cap z]$ to vanish.  We will prove this vanishing using the intersection number lemma.  To make this approach work, we need $\Phi_i(\bar F)$ to be a union of axis-parallel rectangles, with some translation freedom.  We set up a framework for this in the next subsection.

\subsection{Local grids and bending maps}

We can think of the cubical lattice at scale $s$ as the union of hyperplanes

$$ \{ x_j = s m \},  j=1, ..., N,  m \in \mathbb{Z} . $$

We define a grid to be a union of coordinate hyperplanes (which may not be evenly spaced).  For example, if
$h_j(m)$ are real numbers with $h_j(m) < h_j(m+1)$, then we can form a grid by taking the union of all hyperplanes of the form $\{ x_j = h_j(m) \}$, for $j = 1,..., N$, and $m \in \mathbb{Z}$.  Any grid can be expressed in this way, for some appropriate numbers $h_j(m)$.  We say that the spacing of the gird is $\le S$ if $h_j(m+1) - h_j(m) \le S$ for every $j, m$.  

For example, we can make a grid by translating the hyperplanes in the cubical lattice at scale $s$.
Given a perturbation function $p(j,m) \in [-1/4, 1/4]$, the corresponding perturbed grid is given by the
union of hyperplanes $ \{ x_j = s (m + p(j,m) ) \}$, where again $j=1, ..., N, 
m \in \mathbb{Z}$.   Since $|p(j,m)| \le 1/4$, the spacing of this perturbed grid is $\le (3/2)s$.

We can also take the union of two grids, just by taking the union of all
of the hyperplanes.  We say that one grid is contained in a second grid if each hyperplane in
the first grid is contained in the second grid.  
 
We can think of a grid as a polyhedron, and talk about its vertices, its edges, its faces, and so on.

Next we define a ``local grid" for the complex $\bar \Sigma(s)$.  A local grid is a function $G$ that
assigns a grid to each face $f$ (of any dimension) in $\bar \Sigma(s)$ in such a way that if $f_1 \subset f_2$, 
then $G(f_1) \subset G(f_2)$.  In particular, if $v_1, ..., v_{2^d}$ are the vertices of a d-dimensional
face $f \subset \bar \Sigma(s)$, then $G(f)$ must contain $\cup_{i=1}^{2^d} G(v_i)$.  We say that a local
grid $G$ has spacing $\le S$ if each grid $G(f)$ has spacing $\le S$.

For any local grid, we can define a bending function $\Phi$ that behaves nicely with respect to the grid.

\begin{lemma} \label{bendingmap} Let $G$ be a local grid for $\bar \Sigma(s)$ with spacing $\le S$.  Then there is a function $\Phi: \bar \Sigma(s) \rightarrow \mathbb{R}^N$ with the following properties.  For each d-dimensional face $f$ of $\bar \Sigma(s)$, $\Phi(f)$ is contained in the d-skeleton of
$G(f)$.  Moreover, as a chain, $\Phi(f)$ is equal to a sum of d-faces of $G(f)$.

Also, for any point $x$, $|\Phi(x) - x| \le 2 N^2 S$.

\end{lemma}

\begin{proof} We define $\Phi$ one skeleton at a time.  First we define $\Phi$ on the vertices of $\bar \Sigma(s)$.  Let $v$ be a vertex of $\bar \Sigma(s)$.  Since the spacing of $G(v)$ is $\le S$, we can choose a point $\Phi(v)$ with $| \Phi(v) - v| \le N s$, by pushing $v$ to the nearest vertex in $G(v)$.  Now we will define $\Phi$ on higher-dimensional skeleta so that for each $x$ in the d-skeleton, $| \Phi(x) - x | \le (d+1) N S$.

Suppose we have defined $\Phi$ on the (d-1)-skeleton of $\bar \Sigma(s)$.  We have to define $\Phi$ on a d-face $f \subset \bar \Sigma(s)$.  We have already defined $\Phi$ on $\partial f$.  For each face $f_i$ of $\partial f$, we know that $\Phi$ maps $f_i$ into the (d-1)-skeleton of $G(f_i)$.  Since $G$ is a local grid, the (d-1)-skeleton of $G(f_i)$ is contained in the (d-1)-skeleton of $G(f)$.  Therefore,
$\Phi$ maps $\partial f$ into the (d-1)-skeleton of $G(f)$.  So we can extend $\Phi$ to $f$, mapping $f$ into the d-skeleton of $G(f)$.  As a chain, $\Phi(f)$ will be equal to a sum of d-faces of $G(f)$.  By induction, we can assume that $|\Phi(x) - x| \le d N S$ for each $x \in \partial f$.  It's then straightforward to arrange that $| \Phi(x) - x | \le d N S + N S$ for each $x \in f$.    \end{proof}

\subsection{Good local grids}

Next we will construct a good local grid $G_i$ at each scale $s_i$.  The good feature of the local grid is that its faces don't intersect $z$ unnecessarily.

\begin{lemma} \label{goodlocalgrids} For each scale $i \ge 1$ and each vertex $v \in \bar \Sigma(s_i)$, we will construct a grid
$G_i(v)$ with spacing $\le (1/100) N^{-2} s_i$.  

For any face $f$ in $\bar \Sigma(s_i)$, we define $G_i(f)$ to be $\cup_{v \in f} G_i(v)$.

We will also need some grids related to two consecutive scales, $s_{i-1}$ and $s_i$.  If $v$ is a vertex of $\bar \Sigma(s_i)$, define $G_{i-1, i}(v)$ to be the union of $G_i(v)$ and $G_{i-1}(w)$, for all the vertices $w \in \bar \Sigma(s_{i-1})$ which lie within $N s_i$ of $v$.  If $f$ is a face of $\bar \Sigma(s_i)$, we define $G_{i-1,i}(f)$ to 
be the union of $G_{i-1,i}(v)$ over all vertices $v \in f$.

There exists a constant $\beta = \beta(N) > 0$ so that the following holds.

If $R$ is an $(N-n+1)$-face of $G_i(f)$ or $G_{i-1, i}(f)$, and if $R$ lies in the $4 N^2 s_i$-neighborhood of $f$, 
then $R$ will be transverse to $z$ and it will obey the following key estimate.  Let $\Ball[R]$ be the ball centered at the center of $R$ and with radius $N s_i$.  

\begin{itemize}

\item If $[z \cap R] \not= 0$, then 
$\Vol_n( y \cap \Ball[R]) \ge \beta s_i^n$.

\end{itemize}

\end{lemma}

Here is the heuristic explanation.  If $\Vol_n(y \cap \Ball[R])$ is very small, then the intersection number lemma implies
that a small random perturbation of $R$ will usually give $[z \cap R] = 0$.  
So by wiggling all the local grids a little bit, we can arrange the key estimate at the end of the lemma.

The proof of the good local grids lemma is probabilistic, and it depends on the following probability lemma.

\begin{lemma} Suppose that $X = \prod_{i =1}^\infty X_i$ is a product of probability spaces.  Suppose that $\Bad \subset X$ is a union $\Bad = \cup \Bad_\alpha$.  Suppose that each set $\Bad_\alpha$ has probability less than $\epsilon$.  Suppose that each set $\Bad_{\alpha}$ depends on $< C_1$ different coordinates
$x_i$ of the point $x \in X$.  Suppose that each variable $x_i$ is relevant for $< C_2$ different
sets $\Bad_{\alpha}$.  If $\epsilon < (1/2) C_2^{- C_1}$, then $\Bad$ is not all of $X$.
\end{lemma}

See the appendix in Section \ref{secproblemma} for a proof of this lemma and also some more discussion.  Using the probability lemma, we now
prove our lemma on the existence of good local grids.

\begin{proof} 
For each $i$ and for each vertex $v$ in $\bar \Sigma(s_i)$, we will choose a perturbation function $p_{i,v}(j,m) \in [-1/4, 1/4]$, where $j = 1, ..., N$ and $m \in \mathbb{Z}$.  Then we define $G_i(v)$ to be the union of the hyperplanes $x_j = (1/200) N^{-2} s_i (m + p_{i,v}(j,m) )$.  The spacing of each $G_i(v)$ is $\le (1/100) N^{-2} s_i$.  

We are going to apply the probability lemma.  The space $X$ is the set of choices of $p_{i,v}(j,m) \in [-1/4, 1/4]$, where $i \ge 0$, $v$ is a vertex in $\bar \Sigma(s_i)$, $j = 1, ..., N$, and $m \in \mathbb{Z}$.  This is a product space over the index set $(i,v,j,m)$, and for each factor we put the uniform probability distribution on $[-1/4, 1/4]$.  

For almost all choices of $p_{i,v}(j,m)$, the $(N-n+1)$ skeleton of each grid $G_i(f)$ or $G_{i-1,i}(f)$ is transverse to $z$.

Now we turn to the key estimate at the end of the lemma.  We enumerate the different ways that this key estimate may fail.

Let us say that our choice of $p_{i,v}(j,m)$ lies in $\Bad_i(f)$ if $f$ is a face of $\bar \Sigma(s_i)$, and $R$ is an (N-n+1)-face of $G_i(f)$, lying in the $4 N^2 s_i$-neighborhood of $f$, and $[z \cap R] = 1$, and $\Vol_n (y \cap \Ball[R]) < \beta s_i^n$.  Define $\Bad_{i-1,i}(f)$ in the same way, using the grid $G_{i-1,i}(f)$.  

We claim that each set $\Bad_i(f)$ or $\Bad_{i-1,i}(f)$ depends on $C_1 \lesssim 1$ parameters $p_{i,v}(j,m)$.  
There are only one or two choices for $i$.  The vertex $v$ needs to belong to $f$, or at least to lie within $N s_i$ of a vertex of $f$, so there are only $\lesssim 1$ choices of vertex $v$.  There are only $N \lesssim 1$ choices of $j$ in any case.  Since the face $R$ needs to lie in the $4 N^2 s_i$-neighborhood of $f$, we only need to consider values of $m$ where the plane $x_j = s_i m_j$ lies within $C(N) s_i$ of the face $f$, and so there are only $\lesssim 1$ choices for $m$.

We also claim that each variable $p_{i,v}(j,m)$ is only relevant for $C_2 \lesssim 1$ bad sets $\Bad_i(f)$ or $\Bad_{i-1,i}(f)$.  In fact, if $p_{i,v}(j,m)$ is relevant for $\Bad_{i'}(f)$ or $\Bad_{i'-1, i'}(f)$, then we must have $i' = i$ or $i+1$, and $f$ must have a vertex lying near $v$.  This leaves $\lesssim 1$ choices for $f$.

Finally, we have to check that the probability of each set $\Bad_i(f)$ or $\Bad_{i-1,i}(f)$ is $\lesssim \beta$.  Then if we choose $\beta$ sufficiently small, the probability lemma will guarantee that there exists a choice of parameters which is not bad, and we will be done.

So let us consider the probability of $\Bad_i(f)$ (or $\Bad_{i-1,i}(f)$).  There are $\lesssim 1$ faces $R$ which could potentially violate the key estimate.  Each of the faces $R$ that we must consider is positioned by the choice of parameters $p_{i,v}(j,m)$ and maybe $p_{i-1,v}(j,m)$ for vertices $v$ near to $f$ and for a certain range of $m$.  Varying the parameters randomly essentially amounts to translating $R$ at random.  The intersection number lemma says that if $y \cap \Ball[R_0]$ has volume $< \beta s_i^n$, then the probability that $[z \cap R] \not=0$ is $\lesssim \beta$.  So the probability that $R$ violates the key estimate is $\lesssim \beta$.

\end{proof}

We now fix the local grids $G_i$ and $G_{i-1,i}$ and the constant $\beta = \beta(N) > 0$ for the rest of the proof.

\subsection{The bending maps $\Phi_i$ and the cycles $z_i$}

Using the good local grid lemma, we can now construct the bending maps $\Phi_i$ and define the cycles $z_i$.

Using Lemma \ref{bendingmap}, we define $\Phi_i$ to be a bending map with respect to the local grid $G_i$, and we fix $\Phi_i$ for the rest of the proof.  We define $z_i$ to be $D_{\Phi_i}(z)$.

The grids $G_i$ have spacing $S_i \le (1/100) N^{-2} s_i$.  By Lemma $\ref{bendingmap}$, the maps $\Phi_i$ obey $| \Phi_i(x) - x| \le 2 N^2 S_i \le (1/50) s_i$.  

\begin{lemma} \label{thickz_i} The cycle $z_i$ lies in $T_\alpha(y)$ for $\alpha \gtrsim 1$.  Moreover, if a face $F \subset \Sigma(s_i)$ belongs to $z_i$, then there is a ball around the center of $F$ with radius $\sim s_i$ and thickness $\gtrsim 1$.  
\end{lemma}

\begin{proof} For each face $F \subset \Sigma(s_i)$, $\Phi_i(\bar F)$ is equal to a sum of (N-n-1)-faces from $G_i(\bar F)$.  We know that $\Phi_i$ displaces points at most $(1/50) s_i$.  Therefore, each of these faces lies within the $(1/50) s_i$-neighborhood of $\bar F$.  

If $F$ is contained in $z_i$, then $[\Phi_i(\bar F) \cap z] = 1$.  Therefore, $[R \cap z] = 1$ for an (N-n-1)-face
$R$ in $G_i(\bar F)$, lying within the $(1/50) s_i$ -neighborhood of $\bar F$.  

By Lemma \ref{goodlocalgrids}, it follows that
$\Vol_n (y \cap \Ball[R]) \ge \beta s_i^n$, where $\Ball[R]$ denotes the ball around the center of $R$ and with radius $N s_i$.  The ball $\Ball[R]$ contains $F$, and so $F$ lies in $T_\alpha(y)$ for some $\alpha \gtrsim \beta \gtrsim 1$.  
 
Also, the ball around the center of $F$ with radius $3 N^2 s_i$ contains $\Ball[R]$, and this ball has thickness $\gtrsim \beta \gtrsim 1$.   \end{proof}

\begin{lemma} \label{s_ibound} There is a constant $C(N)$ so that if $z_i$ is non-zero, then $s_i \le C(N) \Vol_n(y)^{1/n}$.  

\end{lemma}

\begin{proof} If $z_i$ contains a face $F$, then the last lemma says that there is a ball of radius $\sim s_i$ with thickness $\gtrsim 1$.
The total volume of $y$ is at least the volume of $y$ in this ball, which is $\gtrsim s_i^n$.  \end{proof}

At this point, we define $i_{final}$ to be the smallest $i$ so that $s_i > C(N) \Vol_n(y)^{1/n}$, which guarantees that $z_{i_{final}} = 0$.  
We have $s_i \lesssim \Vol_n(y)^{1/n}$ for all $i \le i_{final}$.  

\begin{lemma} \label{radiusbound} Each cycle $z_i$ lies in the $R$-neighborhood of $z$ for $R \lesssim \Vol_n(y)^{1/n}$. 
\end{lemma}

\begin{proof} Suppose that $z_i$ contains a face $F$.  Then $\Phi_i (\bar F)$ must intersect $z$.  But $\Phi_i$ displaces points by
at most $\lesssim s_i$.  Therefore, the face $F$ is contained in the $R$-neighborhood of $z$ for $R \lesssim s_i \lesssim \Vol_n(y)^{1/n}$.
\end{proof}

\begin{lemma} The cycle $z_0$ is equal to $z$.
\end{lemma}

\begin{proof} If $F,G$ are $(n-1)$-faces of $\Sigma(s_0)$, we have to check that $[ \Phi_0(\bar F) \cap G]$ is 1 if $F = G$ and zero otherwise.  
This is clearly true if $\Phi_0$ is the identity.  Now $\Phi_0$ displaces each point by at most $(1/50) s_0$.  The boundary of $\bar F$ is at a distance at least $(1/2) s_0$ from $\Sigma^{n-1}(s_0)$.  So the straightline homotopy from the identity to $\Phi_0$ will never map $\partial (\bar F)$ into any $G$, and the intersection numbers will not change. \end{proof}

To finish the proof of the perpendicular pair inequality, we have to construct $n$-chains $A_i \subset T_\alpha(y)$ with
$\partial A_i = z_{i-1} - z_i$.  
 
\subsection{Homologies between deformations at different scales}
 
Now we have a deformation $D_i = D_{\Phi_i}$ based on the bending map $\Phi_i: \bar \Sigma (s_i)  \rightarrow \mathbb{R}^N$.  
 The deformation operator $D_i$ deforms chains/cycles into cubical chains/cycles inside $\Sigma(s_i)$.  
 In particular, $D_i z = z_i$ is a cubical cycle, and we know that each $z_i$ lies in $T_\alpha(y)$.  Next, we have to construct
 homologies $A_i$ from $z_{i-1}$ to $z_i$ also lying in $T_\alpha(y)$.  The cycles $z_{i-1}$ and $z_i$ are cycles at different scales: $z_{i-1}$ is a cubical cycle in $\Sigma(s_{i-1})$ and $z_i$ is a cubical cycle in $\Sigma(s_i)$.  In this section, we construct a cycle $z_{i-1}^+$ at scale $s_i$ and a homology from
 $z_{i-1}$ to $z_{i-1}^+$.  In the next section, we do the harder work of constructing a homology from
 $z_{i-1}^+$ to $z_i$.  
 
We let $v = (s_{i-1}/2, ..., s_{i-1}/2)$.  We now define $z_{i-1}^+$ using the standard deformation operator for $\Sigma(s)$:
 
 $$z_{i-1}^+ = D_v(z_{i-1}) = \sum_{F \subset \Sigma^{n-1}(s_i)} [\bar F_v \cap z_{i-1} ] F. $$

(One reason for translating by $v$ is as follows.  If $F \in \Sigma^{n-1}(s_i)$ and $\bar F$ is the corresponding face of
$\bar \Sigma^{N-n+1}(s_i)$, then $\bar F$ is not transverse to $\Sigma(s_{i-1})$, and so $\bar F$ is not transverse to $z_{i-1}$.  But $\bar F_v$ is transverse to $\Sigma(s-1)$, and so $z_{i-1}^+$ is defined.  Other reasons of our specific choice of $v$ appear below.)
 
 The cycle $z_{i-1}^+$ lies in the $C(N) s_{i-1}$-neighborhood of $z_{i-1}$.  Also, by the standard properties of the deformation operator, there
 is a homology $A_i'$ from $z_{i-1}$ to $z_{i-1}^+$ in the $C(N) s_{i-1}$-neighborhood of $z_{i-1}$.  (See Section \ref{fedflemreview} for a review of the proof.)  The following lemma shows that $A_i'$ lies in $T_\alpha(y)$.  
 
 \begin{lemma} The $C(N) s_{i-1}$ neighborhood of $z_{i-1}$ lies in $T_\alpha(y)$ for $\alpha \gtrsim 1$.
 \end{lemma}
 
 \begin{proof} From Lemma \ref{thickz_i}, we know that if $F$ is a face of $\Sigma(s_{i-1})$ which is contained in $z_{i-1}$, then there is a
ball around the center of $F$ with radius $\sim s_{i-1}$ and thickness $\gtrsim 1$.  Therefore, the ball of radius $C(N) s_{i-1}$ around any point of $z_i$ has thickness $\gtrsim 1$.   \end{proof}
 
 \begin{lemma}  The chain $A_i'$ lies in the $R$-neighborhood of $z$ for $R \lesssim \Vol_n(y)^{1/n}$.
 \end{lemma}
 
 \begin{proof}  We know that $A'$ lies in the $C(N) s_i$ neighborhood of $z_i$.  By Lemma \ref{s_ibound}, $s_i \lesssim \Vol_n(y)^{1/n}$.
 By Lemma \ref{radiusbound} $z_i$ lies in the $R$-neighborhood of $z$ for $R \lesssim \Vol_n(y)^{1/n}$. \end{proof}
 
The main difficulty is to build a homology $A''_i$ from $z_{i-1}^+$ to $z_i$.  To facilitate this, it helps to describe $z_{i-1}^+$ and $z_i$ in similar ways.  Recall that $z_i = D_{\Phi_i}(z)$.  We will now construct a map $\Phi_{i-1}^+: \bar \Sigma(s_i) \rightarrow \mathbb{R}^N$, and show that
 $z_{i-1}^+$ is $D_{\Phi_{i-1}^+}(z)$.  
 
The map $\Phi_{i-1}^+$ is defined in terms of $\Phi_{i-1}$ by

$$\Phi_{i-1}^+(x) = \Phi_{i-1}(x + v). $$

\begin{lemma} The chain $z_{i-1}^+$ is equal to $D_{\Phi_{i-1}^+}(z)$.
\end{lemma}

\begin{proof} By definition,

$$z_{i-1}^+ = \sum_{F \subset \Sigma^{n-1}(s_i)} [\bar F_v \cap z_{i-1}] F. $$

Now $z_{i-1} = D_{\Phi_{i-1}} z = \sum_{H \subset \Sigma^{n-1}(s_{i-1})} [ \Phi_{i-1} (\bar H) \cap z] H$.

When we plug the definition of $z_{i-1}$ into the formula for $z_{i-1}^+$, we see that the coefficient of $F$ in $z_{i-1}^+$ is

$$\sum_{H \subset \Sigma^{n-1}(s_{i-1})} [\Phi_{i-1} (\bar H) \cap z] [\bar F_v \cap H] .$$

We let $H(F)$ be the set of $(n-1)$-faces $H$ in $\Sigma^{n-1}(s_{i-1})$ so that $[\bar F_v \cap H] = 1$.  Next we check that $\sum_{H \in H(F)} \bar H = \bar F_v$.  (It may be helpful to draw a picture here.)  Recall that $\bar F$ is an (N-n+1)-dimensional face in $\bar \Sigma(s_i)$.  
The translated face $\bar F_v$ lies in the (N-n+1)-skeleton of $\bar \Sigma(s_{i-1})$.  The
face $\bar F_v$ is a sum of $2^{N-n+1}$ faces of $\bar \Sigma^{N-n+1}(s_{i-1})$.  We call these faces $\bar J_1, ..., \bar J_{2^{N-n+1}}$,
where the $J_i$ are $(n-1)$-faces in $\Sigma(s_{i-1})$.  Now if $H$ and $J$ are any two faces in $\Sigma^{n-1}(s_i)$, then the intersection
number $[\bar J \cap H]$ is equal to 1 if $H = J$ and 0 otherwise.  Therefore, $[\bar F_v \cap H] = 1$ if $H$ is one of the faces $J_1, ...,
J_{2^{N-n+1}}$ and 0 otherwise.  In other words, $H(F)$ is exactly $J_1, ..., J_{2^{N-n+1}}$.  Now $\sum_{H \in H(F)} \bar H =
\sum_i \bar J_i = \bar F_v$.  

Using this information, we see that the coefficient of $F$ in $z_{i-1}^+$ is

$$\sum_{H \in H(F)} [\Phi_{i-1} (\bar H) \cap z] = [\Phi_{i-1}(\bar F_v) \cap z] = [\Phi_{i-1}^+ (\bar F) \cap z] . $$

Therefore, $z_{i-1}^+ = \sum_{F \subset \Sigma^{n-1}(s_i)} [\Phi_{i-1}^+ (\bar F) \cap z] F = D_{\Phi_{i-1}^+}(z)$.  \end{proof}

The map $\Phi_{i-1}^+$ has good properties analogous to the map $\Phi_i$.  We state this as a lemma.

\begin{lemma} For each d-dimensional face $f$ of $\bar \Sigma(s_i)$, $\Phi_{i-1}^+(f)$ is contained in the d-skeleton of
$G_{i-1,i}(f)$.  Moreover, as a chain, $\Phi_{i-1}^+(f)$ is equal to a sum of d-faces of $G_{i-1,i}(f)$.

Also, for any point $x$, $|\Phi_{i-1}^+(x) - x| \le N s_i$.

\end{lemma}

\begin{proof} Recall that $v = (s_{i-1}/2, ..., s_{i-1}/2)$.  Now $\Phi_{i-1}^+(f) = \Phi_{i-1}(f+v)$.  As we saw in the proof of the last lemma, $f+v$ is a sum of $2^d$ d-faces in $\bar \Sigma(s_{i-1})$, $f+v = \sum h_j$.  Now $\Phi_{i-1}(h_j)$
lies in the d-skeleton of the grid $G_{i-1}(h_j) \subset G_{i-1,i}(f)$.  Therefore, $\Phi_{i-1}^+(f)$ lies in the d-skeleton of
$G_{i-1,i}(f)$.  As a chain, $G_{i-1}(h_j)$ is a sum of faces of $G_{i-1}(h_j)$ - and each of these faces is a sum of d-faces of $G_{i-1,i}(f)$.  Therefore, $\Phi_{i-1}^+(f)$ is a sum of d-faces of $G_{i-1,i}(f)$.
Finally, $|\Phi_{i-1}(x) - x| \le (1/50) s_{i-1} = (1/100) s_i$.  Now $|\Phi_{i-1}^+(x) - x| = |\Phi_{i-1}(x+v) - x| \le |v| + |\Phi_{i-1}(x+v) - (x+v)| \le (1/100) s_i + |v| = (1/100) s_i + (1/4) N^{1/2} s_i. $
\end{proof}
 
\subsection{Homologies between different deformations at the same scale}
 
 Let $\Phi_1$ and $\Phi_2$ be two ways of bending the dual skeleton $\bar \Sigma(s)$.  Let
 $D_1$ and $D_2$ be the corresponding deformation operators.  If we have a homotopy
 from $\Phi_1$ to $\Phi_2$, we will use it to define a ``homotopy'' between the deformation
 operators $\Phi_i$.
 
 Suppose that $\Phi_{1,2}: \bar \Sigma \times [1,2] \rightarrow \mathbb{R}^N$ is a homotopy from
 $\Phi_1$ to $\Phi_2$.  We define the operator $D_{1,2}$ in terms of $\Phi_{1,2}$ as follows.  Let
 $T$ be a d-dimensional chain in $\mathbb{R}^N$ transverse to all $\Phi$ maps.
 
 $$D_{1,2}(T) := \sum_{G^{d+1} \subset \Sigma^{d+1} } [ \Phi_{1,2}(\bar G \times [1,2]) \cap T ] G. $$
 
 Notice that if $T$ is a d-chain, then $D_{1,2}(T)$ is a (d+1)-chain.
 
 The key formula about $D_{1,2}$ is the following.
 
\begin{lemma} The operator $D_{1,2}$ obeys the following algebraic identity:

$$ \partial D_{1,2} (T) = D_1(T) + D_2(T) + D_{1,2} ( \partial T). $$
 
 In particular, if $z$ is a cycle, then $\partial D_{1,2}(z) = D_1(z) + D_2(z)$.
 
\end{lemma}

\proof  This proof is similar to the proof that a deformation operator $D$ commutes with boundaries.

The left-hand side is

$$ \partial D_{1,2}(T) = \sum_{G \subset \Sigma^{d+1}} [\Phi_{1,2}( \bar G \times [1,2] ) \cap T] \partial G. $$

Now for each d-face $F$ in $\Sigma^{d}(s)$, we let $G_1(F), ..., G_{2(N-d)}(F)$ be the (d+1)-faces
having $F$ in their boundaries.  So the last equation may be rewritten as

$$\sum_{F \subset \Sigma^{d}} \sum_{j=1}^{2N-2d} [\Phi_{1,2} (\overline{G_j(F)} \times [1,2]) \cap T]
F.$$

The first point is that $\sum_j \overline {G_j(F)} = \partial \bar F$.  So our last equation may be
rewritten as

$$\partial D_{1,2}(T) = \sum_{F \subset \Sigma^d} [\Phi_{1,2}(\partial \bar F \times [1,2]) \cap T] F. \eqno{(*)}$$

Next, $\partial \bar F \times [1,2] = \partial (\bar F \times [1,2]) + \bar F \times \{1 \} + \bar F
\times \{ 2 \}$.  Plugging this formula in, we see that

$$[\Phi_{1,2}(\partial \bar F \times [1,2]) \cap T] = [\Phi_1 (\bar F) \cap T] + [\Phi_2 (\bar F)
\cap T] + [\partial \Phi_{1,2}(\bar F \times [1,2]) \cap T]. $$

Next we rearrange the last term in this equation.  By transversality, $\Phi_{1,2}(\bar F \times [1,2]) \cap T$ is a 1-chain with 
an even number of boundary points.  Therefore,
$[\partial \Phi_{1,2} (\bar F \times [1,2]) \cap T] = [\Phi_{1,2} (\bar F \times [1,2]) \cap \partial T]$.  Using this substitution,
we see

$$[\Phi_{1,2}(\partial \bar F \times [1,2]) \cap T] = [\Phi_1 (\bar F) \cap T] + [\Phi_2 (\bar F)
\cap T] + [\Phi_{1,2}(\bar F \times [1,2]) \cap \partial T]. $$

Putting this formula for the intersection number back into equation $(*)$, we see that $\partial D_{1,2} T = D_1(T) + D_2(T) + D_{1,2}(\partial T)$.

\endproof

In our setting, we need to build a homology from $z_{i-1}^+ = D_{\Phi_{i-1}^+} (z)$ to
$z_i = D_{\Phi_i} (z)$.  Here $\Phi_{i-1}^+$ and $\Phi_i$ are both maps
defined on  $\bar \Sigma_{s_i}$.  To get a homology from $z_{i-1}^+$ to $z_i$ we
need a homotopy from $\Phi_{i-1}^+$ to $\Phi_i$.  We will build this homotopy by the same
local grid method that we used to build $\Phi_i$ in the first place.

We know that $\Phi_i$ and $\Phi_{i-1}^+$ each map each d-face $f$ of $\bar \Sigma(s_i)$ into the d-skeleton of $G_{i-1,i}(f)$.  We extend to get a homotopy $\Phi_{i-1,i}: \bar \Sigma(s_i) \times [1,2] \rightarrow \mathbb{R}^N$ which maps each d-face of $\bar \Sigma(s_i) \times [1,2]$ into the d-skeleton of $G_{i-1,i}(f)$.  Since $\Phi_i$ and $\Phi_{i-1}^+$ each have displacement $\le N s_i$, we can arrange that $|\Phi_{i-1,i}(x,t) - x| \le 3 3 N s_i$.  

We define a homology $A_i''$ to be $A_i'' = D_{\Phi_{i-1,i}} (z)$.  Since $z$ is a cycle and $\Phi_{i-1,i}$ is a homotopy from $\Phi_{i-1}^+$ to $\Phi_i$, $\partial A_i'' = D_{\Phi_{i-1}^+}(z) + D_{\Phi_i}(z) = z_{i-1}^+ + z_i$.  

\begin{lemma} The chain $A_i''$ lies in $T_\alpha(y)$ for $\alpha \gtrsim 1$.
\end{lemma}

\begin{proof} Suppose that $F \subset \Sigma^{d+1}(s_i)$ is a (d+1)-face of $A_i''$.  
By definition, we know that $[ \Phi_{i-1,i}(\bar F \times [1,2]) \cap z ] \not= 0$.  But $\Phi_{i-1,i}(\bar F \times [1,2])$ is 
a sum of (N-n-1)-faces of $G_{i-1,i}(\bar F)$ all lying within a $4 N s_i$-neighborhood of $F$.  For one of these faces, $R$, we must have $[R \cap z] \not= 0$.  

Recall that $\Ball[R]$ is the ball around the center of $R$ with radius $N s_i$.   By the good local grid lemma, Lemma \ref{goodlocalgrids}, it follows that $\Vol_n (y \cap \Ball[R]) \ge \beta s_i^n \gtrsim s_i^n$.  Therefore, there is a ball around the center of $F$ with radius $\sim s_i$ and thickness $\gtrsim 1$.  So $F \subset T_\alpha(y)$ for some $\alpha \gtrsim 1$.   \end{proof}

\begin{lemma} The chain $A_i''$ lies in the $R$-neighborhood of $z$ for $R \lesssim \Vol_n(y)^{1/n}$.
\end{lemma}

\begin{proof} If $F$ is a face of $A_i''$, then we see that $\Phi_{i-1,1}(\bar F \times [1,2])$ must intersect $z$.  Since $| \Phi_{i-1,i}(x,t) - x| \lesssim s_i$, we see that $F$ must lie in the $C(N) s_i$-neighborhood of $z$.  But since $i \le i_{final}$, $s_i \lesssim \Vol_n(y)^{1/n}$. 
\end{proof}

Now $A_i := A_i' + A_i''$ is a homology from $z_{i-1}$ to $z_i$, contained in the thick region
$T_\alpha(y)$.  Also, $A_i$ is contained in the $R$-neighborhood of $z$ for $R \lesssim \Vol_n(y)^{1/n}$.  

This completes the proof of the perpendicular pair inequality for cubical cycles and chains $z, w, y$.  In the next
subsection we explain how to reduce the general proposition to the case of cubical cycles and chains.

\subsection{Approximating by cubical cycles and chains} \label{minorapprox}

Let us recall the hypotheses of the perpendicular pair inequality.  We have mod 2 $(n-1)$-cycles $z$ and $w$,
and a mod 2 $n$-chain $y$ in $\mathbb{R}^N$.  We know that $\partial y = z + w$.  We know that $z$ and $w$ are
perpendicular in the sense that for any coordinate $(n-1)$-tuple $J$, either $\Vol_J(z) = 0$ or $\Vol_J(w) = 0$.

We want to approximate these chains and cycles by some cubical chains and cycles $\tilde z$, $\tilde w$, and $\tilde y$.
For some tiny constant $s_0$, we will choose $\tilde z$ and $\tilde w$ as cubical cycles in $\Sigma(s_0)$ and
$\tilde y$ as a cubical chain in $\Sigma(s_0)$.  In order to preserve the structure of the problem, we need to check
that $\tilde z$ and $\tilde w$ are still perpendicular, and that $\Vol(\tilde y) \lesssim \Vol(y)$.  Finally, we need a homology
$A$ from $z$ to $\tilde z$ with volume as small as we like.  Given all these things, we can quickly reduce the perpendicular pair
inequality to the cubical case.  By the cubical case, we can find a chain $\tilde y'$ with $\partial \tilde y' = \tilde z$ and with
$\HC_n(\tilde y') \lesssim \Vol_n(\tilde y) \lesssim \Vol_n(y)$.  Finally, we define $y' = A + \tilde y'$.  We see that $\partial y' = z$.  Also,
$\HC_n(y') \le \HC_n(A) + \HC_n(\tilde y') \lesssim \epsilon + \Vol_n(y)$, where $\epsilon$ is as small as we like.

We do the cubical approximation by the Federer-Fleming deformation operator.  We let $v$ be a vector with $|v| \le s_0/2$, which we can choose later, and we define for each d-chain $T$

$$D_v (T) := \sum_{F^d \subset \Sigma^d(s_0)} [\bar F_v \cap T] F. $$

\noindent We let $\tilde z = D_v(z)$, $\tilde w = D_v(w)$, and $\tilde y = D_v(y)$.  We will use some standard properties of the Federer-Fleming deformation operator, which are reviewed in Section \ref{fedflemreview}.
The deformation operator commutes with taking boundaries, and
so $\partial \tilde y = \tilde z + \tilde w$.   We can choose $v \in B^N(s/2)$ so that $\Vol_n(\tilde y) \lesssim \Vol_n(y)$,
and so that there is a homology $A$ from $z$ to $\tilde z$ with volume $\lesssim s_0 \Vol_{n-1}(z)$.  By taking $s_0$ small enough, we can make this homology as small as we like. 

We still have to check that $\tilde z$ and $\tilde w$ are perpendicular.   To do this, we check that the deformation operator $D_v$ is well-behaved with respect to directed volumes $\Vol_J$ in Euclidean space.

\begin{lemma} If $T$ is any mod 2 d-chain in $\mathbb{R}^N$, and if $J$ is a d-tuple of integers from 1 to $N$, then 

$$\Average_{v \in B^N(s/2)} \Vol_J [ D_v (z) ] \le C(N) \Vol_J z. $$

\end{lemma}

\begin{proof} Let $F$ be a d-face of $\Sigma^d(s)$ in the direction $J$.  We have to consider $[\bar F_v \cap T]$.  Let $\Ball[F]$ denote the ball around the center of $F$ with radius $2N s$.  Note that for all $v \in B^N(s/2)$, $\bar F_v$ is contained in $\Ball[F]$.  Let $\pi_J$ be the orthogonal projection onto the $J$-plane.  The probability that $[\bar F_v \cap T] \not=0$ is bounded by $C(N) s^{-N} \Vol \pi_J( T \cap \Ball[F] ) \le C(N) s^{-N} \Vol_J (T \cap \Ball[F])$.  Therefore, 

$$\Average_{v \in B^N(s/2)} \Vol_J [D_v(T)] \le C(N) \sum_{F \subset \Sigma^d(s), \text{$F$ in direction $J$}} \Vol_J(T \cap \Ball[F]) \le C(N) \Vol_J(T). $$

\end{proof}

In particular, if $\Vol_J(T) = 0$, then $\Vol_J ( D_v(T) ) = 0$ almost surely.  Therefore, for almost every $v$, $\tilde z$ and $\tilde w$ are still perpendicular.  This finishes the reduction of the perpendicular pair inequality to the cubical case.

\section{Thick tubes} \label{twistedtubes}

In this section, we define a tube in $\mathbb{R}^N$ to be an embedding $I: S^1 \times B^{N-1} \rightarrow \mathbb{R}^N$.  Recall that an 
embedding is called $k$-expanding if it increases or preserves the $k$-volume of each $k$-dimensional surface.  In other words, an embedding is
$k$-expanding if its inverse has $k$-dilation $\le 1$.  A tube with 
$k$-thickness equal to $R$ is a $k$-expanding embedding $I$
from $S^1(\delta) \times B^{N-1}(R)$ into $\mathbb{R}^N$, where $\delta > 0$
may be arbitrarily small.  (Here we write $S^1(\delta)$ for the circle of radius $\delta$.)  We will usually denote a tube by the letter $T$.  We will say
that a tube $T$ lies in some set $U \subset \mathbb{R}^N$ if the image of the embedding lies in $U$.

For example, consider a solid torus of revolution in $\mathbb{R}^3$ given by revolving a disk around the z axis.  If we take a disk of radius $1$ 
with center at a distance 2 from the axis, then we get a tube of thickness $\sim 1$ contained in a ball of radius 3 around the origin.  Surprisingly, 
there are tubes of thickness 1 contained in arbitrarily small balls.  Their geometry is quite different from a solid torus of revolution.

\newtheorem*{GZ}{Thick tube example in three dimensions}
\begin{GZ} For every radius $\epsilon > 0$, there is a tube
$T$ with 2-thickness 1 embedded in $B^3(\epsilon) \subset \mathbb{R}^3$.
\end{GZ}

(As far as I know, the first example was constructed by Zel'dovitch - see \cite{Ar}.  We discuss Zel'dovitch's construction in Section \ref{zeltube}.)

\proof  We begin with $S^1(\delta) \times B^2(1)$, where we may choose $\delta$ as small
as we like.  This product isometrically embeds in
$S^1(\delta) \times [0,2] \times [0,2]$.  Now for each $\lambda > 1$, we can make a 2-expanding map from this space into $S^1(\delta \lambda) \times [0, 2 \lambda^{-1}] \times
[0, 2 \lambda]$ by dilating each coordinate by an appropriate factor.   We choose
$\lambda = \delta^{-1/2}$, so the image of the embedding is $S^1(\delta^{1/2}) \times [0, 2 \delta^{1/2}]
\times [0, 2 \delta^{-1/2}]$.  Now the annulus $S^1(\delta^{1/2}) \times [0, 2 \delta^{1/2}]$ has
a 1-expanding embedding into $B^2( 10 \delta^{1/2})$.  Hence our original space has a
2-expanding embedding into the cylinder $B^2(10 \delta^{1/2}) \times [0, 2 \delta^{-1/2}]$.
Note that this cylinder has volume $\sim \delta^{1/2}$.  The cylinder admits a 1-expanding
embedding into $B^3(100 \delta^{1/6})$. (See the appendix in Section \ref{bilipembed} for the details of this embedding.)  \endproof

The situation is different for linked tubes.  This phenomenon was discovered by Gehring in \cite{Ge}.   He proved a result similar
to the following.

\newtheorem*{GLi}{Gehring-type inequality for linked tubes}

\begin{GLi} If $T_1$ and $T_2$ are disjoint tubes in the unit 3-ball, with 2-thickness
$R_1$ and $R_2$, and with linking number $L$, then the following inequality holds:

$$ L R_1^2 R_2^2 \le C. $$

\end{GLi}

\proof  Let $I_1: S^1(\delta) \times B^2(R_1) \rightarrow B^3(1)$ and $I_2: S^1(\delta) \times B^2(R_2) \rightarrow
B^3(1)$ be our 2-expanding embedding maps.  Recall that $T_i$ is the image of $I_i$.  
 Taking the inverses of our embedding maps, we get maps
$\pi_i: T_i \rightarrow B^2(R_i)$ with 2-dilation at most 1.  By the coarea formula,
we can find a fiber of $\pi_1$ with length at most $\Vol(T_1) / \Area(B^2(R_1))$ which
is at most $C R_1^{-2}$.  This fiber bounds a disk of area at most $C R_1^{-4}$ by
the isoperimetric inequality.  Also, the fiber bounds a disk of area at most $C R_1^{-2}$
by the cone inequality.

Now this disk cuts across the tube $T_2$ at least $|L|$ times.  More precisely, if we let
$D$ denote the disk, then the map $\pi_2$ from $D \cap T_2$ to $B^2(R_2)$ has
degree $L$.  Hence we see that $D$ has area at least $L \pi R_2^{2}$.  Using the upper bound
for the area of $D$ from the cone inequality, we see $L \pi R_2^2 \le C R_1^{-2}$, and so $L R_1^2 R_2^2 \le C$.  

This proves the result.  If we use the upper bound for the area of $D$ coming from the traditional isoperimetric
inequality, we see that $L R_2^2 \le C R_1^{-4}$ and so $L R_1^4 R_2^2 \le C$.  This latter inequality is stronger
when $R_1 \gg 1$.  \endproof

To summarize, a tube $T$ with 2-thickness 1 may be squeezed into an arbitrarily small
ball in $\mathbb{R}^3$, but if $T_1$ and $T_2$ are linked tubes with 2-thickness 1, then they
cannot be squeezed into a small ball in $\mathbb{R}^3$.

Now we recall the idea of the twisting number of a tube in $\mathbb{R}^3$.   Let
$p_1$ and $p_2$ be two points in $B^2$.  Then consider the two circles $I (S^1 \times \{p_1\})$
and $I (S^1 \times \{ p_2 \} )$ in $\mathbb{R}^3$.  The twisting number of the tube $T$ is equal
to the linking number of these two circles.  Moreover, let $B_1$ and $B_2$ be disjoint disks
in $B^2$ with centers at $p_1$ and $p_2$.  Then $I$ restricted to $S^1 \times B_1$ defines
a tube $T_1$ and $I$ restricted to $S^1 \times B_2$ defines a tube $T_2$.  We can arrange that
$T_1$ and $T_2$ each have thickness at least one third the thickness of $T$.  Therefore,
we get an inequality for the twisting number of a thick tube:

\newtheorem*{GT}{Inequality for twisted tubes}

\begin{GT} Suppose that $T$ is a tube in the unit 3-ball with 2-thickness $R$ and
twisting number $t$.  Then $R^4 |t| \le C$.
\end{GT}

(For large values of $t$, I don't know whether this inequality is sharp.)

To summarize, in three dimensions, a tube with 2-thickness 1 may be squeezed into an arbitrarily small ball $B^3(\epsilon)$,
but a tube with a non-zero twisting number cannot.  
In the early 1990's, Freedman and He \cite{FH} extended Gehring's work, proving estimates for general knots and links.  For example,
they proved that a 3-dimensional tube with 2-thickness 1 contained
in a small ball must be unknotted.

Next we discuss the situation in dimension $N \ge 4$.  We begin with the generalization of the thick tube example.

\newtheorem*{GZ2}{Thick tube example in higher dimensions}
\begin{GZ2} If $N \ge 3$ and $k > N/2$, then there is a tube $T$ with $k$-thickness 1 embedded in
$B^N(\epsilon) \subset \mathbb{R}^N$ for every $\epsilon > 0$.
\end{GZ2}

\proof  We begin with $S^1(\delta) \times B^{N-1}(1)$, where we may choose $\delta$ as small
as we like.  This product isometrically embeds in
$S^1(\delta) \times [0,2]^{N-1}$.  Now for each $\lambda > 1$, we can make a $k$-expanding map from this space into $S^1(\delta \lambda^{\frac{1}{k-1}}) \times [0, 2 \lambda^{-1}] \times
[0, 2 \lambda^{\frac{1}{k-1}}]^{N-2}$ by dilating each coordinate by an appropriate factor.   We choose $\lambda$ so that
$\delta \lambda^{\frac{1}{k-1}} = \lambda^{-1}$.  By making $\delta$ small, we can make $\lambda$ as large as we like.  We are now working with $S^1(\lambda^{-1}) \times [0, 2 \lambda^{-1}]
\times [0, 2 \lambda^{\frac{1}{k-1}}]^{N-2}$.  Now the annulus $S^1(\lambda^{-1}) \times [0, 2 \lambda^{-1}]$ has
a 1-expanding embedding into $B^2( 10 \lambda^{-1})$.  Hence our original space has a
$k$-expanding embedding into the cylinder $B^2(10 \lambda^{-1}) \times [0, 2 \lambda^{\frac{1}{k-1}}]^{N-2}$.
Note that this cylinder has volume $\sim \lambda^{-2 + (N-2)/(k-1)}$.  Since $k > N/2$, the exponent is negative, and
we can make the volume as small as we like.  Therefore, this cylinder admits a 1-expanding embedding into an
arbitrarily small ball.  (See the appendix in Section \ref{bilipembed} for the details of this last embedding.)  \endproof

I don't know whether the range of $k$ in this result is sharp.  There is more discussion in the open problems section, section \ref{openproblems}.

In dimension $N \ge 4$, any two circles are unlinked, and so there is no linking number between tubes.  It turns out that it is still possible to define the twisting number of a tube in $\mathbb{Z}_2$.

A tube is given by an embedding $I: S^1 \times B^{N-1} \rightarrow \mathbb{R}^N$.  The `core circle'
of the tube is $I ( S^1 \times \{ 0 \} )$.  The tube structure defines a trivialization of the normal
bundle of this circle -- in other words a framing of this circle.  The twisting number comes from this framing - it measures how the trivialization `twists' as we move around the circle.

As a warmup, suppose that the core circle is just the standard circle $S^1 \subset \mathbb{R}^N$, defined by the equations $x_1^2 + x_2^2 = 1$, $x_3 = ... = x_N = 0$, and suppose that $I: S^1 \rightarrow S^1$ is the identity map.  This core circle has a standard framing, given by $\{ r, e_3, ..., e_N \}$, where $r$ is the outward radial vector $(x_1, x_2)$, and $e_3, ..., e_N$ are the coordinate vectors.  Now we can compare the framing coming from out embedding $I$ with this standard framing.  At each point in $S^1$, the two framings each give a basis of $N S^1$.  Therefore, we can get our framing by applying an element of $\GL_{N-1}(\mathbb{R})$ to the standard framing at each point $x \in S^1$.  So our framing induces a continuous map from $S^1$ to $\GL_{N-1}(\mathbb{R})$.  If $I$ is orientation-preserving, we see that the image of our map lies in the orientation-preserving maps $\GL_{N-1}^+(\mathbb{R})$, which is homotopic to $\SO(N-1)$.  Therefore, our framing induces an element of $\pi_1 (\SO(N-1))$.  If $N \ge 4$, then $\pi_1(\SO(N-1)) = \mathbb{Z}_2$, and this homotopy class is the twisting number of the tube.  (If $I$ is orientation reversing, our map goes from $S^1$ to $\GL_{N-1}^-(\mathbb{R})$, which is also homotopic to $\SO(N-1)$, and the twisting number is defined in the same way.)

However, this definition was just a warmup, and we still have to define the twisting number for a general embedding $I: S^1 \times B^{N-1} \rightarrow \mathbb{R}^N$.  The main problem is to define a standard framing for an arbitrary embedded circle in $\mathbb{R}^N$.  Because $N \ge 4$, any two circles are isotopic.  Therefore, there is an orientation-preserving diffeomorphism $\Psi$ of $\mathbb{R}^N$ so that $\Psi \circ I$ is the identity map from $S^1$ to the standard $S^1$.  The map $\Psi$ gives an isomorphism from the normal bundle of the core circle to the normal bundle of the standard $S^1$.  Using $\Psi$, we can pull back the standard framing of $S^1$ to give a standard framing of the core circle.  Comparing this standard framing with the framing induced by the tube, we get a map $S^1 \rightarrow \GL_{N-1}(\mathbb{R})$, and the homotopy class of the map gives a twisting number in $\mathbb{Z}_2$.  A priori, this definition may depend on the choice of $\Psi$, but it turns out that it does not.  This was first established by Pontryagin in his study of framed cobordism and homotopy groups of spheres.

There is a nice description of (some of) this work in Milnor's book \cite{M}.  Pontryagin defined a notion of framed cobordism, which is an equivalence relation on closed framed $k$-manifolds within a given ambient manifold.  Pontryagin proved that for $N \ge 4$, there are exactly two equivalence classes of framed 1-manifolds in $\mathbb{R}^N$.  He proved that two framed circles are framed cobordant if and only if they have the same twisting number.

Pontryagin went on to make an important connection between framed cobordisms and homotopy groups of spheres.
Let $T$ be a twisted tube in the unit $N$-ball.  The tube $T$ is defined by an embedding
$I$ from $S^1 \times B^{N-1}$ into the unit $N$-ball, with image $T \subset B^N$.  Using the inverse of $I$, we
construct a map $\pi: T \rightarrow B^{N-1}$, sending the boundary of $T$ to the boundary of $B^{N-1}$.
We let $h$ be a degree 1 map from $(B^{N-1}, \partial B^{N-1})$ to $S^{N-1}$, sending the boundary $\partial B^{N-1}$
to the basepoint of $S^{N-1}$.  Let us consider the unit N-ball, $B^N$, as the upper hemisphere of $S^N$.  Now we construct a map
$F$ from $S^N$ to $S^{N-1}$ as follows.  If $x \in T$, then we define $F(x) = h( \pi(x) )$.  If $x$ is not in $T$,
then we define $F(x)$ to be the basepoint of $S^{N-1}$.  This definition gives a continuous map because
if $x \in \partial T$, then $\pi(x) \in \partial B^{N-1}$ and $h ( \pi(x))$ is the basepoint of $S^{N-1}$.  Pontryagin showed
that the homotopy type of the map $F$ depends only on the framed cobordism class of the core circle.  In particular, he proved the following theorem.

\newtheorem*{PT}{Pontryagin's theorem}

\begin{PT} If $N \ge 4$, then the twisting number of the tube $T$ in $\mathbb{Z}_2$ is non-zero if
and only if the map $F$ is non-contractible. \end{PT}

Since $F$ is non-contractible if and only if $\SH(F) = 1$, we see that the twisting number of $T$ is equal to the Steenrod-Hopf invariant of $F$.

Combining our main theorem on $k$-dilation with Pontryagin's theorem about twisting numbers of tubes, we get the following estimate for twisted tubes.

\begin{prop} \label{kthicktube} If $T$ is a tube in the unit ball in $\mathbb{R}^N$ with non-zero twisting number, and if $k \le (N+1)/2$,
then the $k$-thickness of $T$ is $\le C(N)$.
\end{prop}

\begin{proof} We write $A \lesssim B$ for $A \le C(N) B$.  If the tube $T$ has $k$-thickness $R$, then the map $F$ has $k$-dilation $\lesssim R^{-k}$.
By our main theorem, if $k \le (N+1)/2$, then the $k$-dilation of every non-trivial map $S^N$ to $S^{N-1}$ is $\gtrsim 1$.
If $k \le (N+1)/2$, then we see that $1 \lesssim R^{-k}$.  Therefore the tube $T$ has $k$-thickness $R \lesssim 1$. \end{proof}

Proposition \ref{kthicktube} is most interesting for odd dimensions $N \ge 5$.  In this case, the thick tube example in higher dimensions shows that
we may embed a tube in the unit N-ball with arbitrarily large $(N+1)/2$-thickness, but Proposition \ref{kthicktube} shows that if the $(N+1)/2$-thickness
is too large, then the twisting number must be zero.

\section{Quantitative general position arguments} \label{sectionquantemb}

In the next section, we will prove the h-principle for $k$-dilation stated in the introduction.  In the construction, we need to construct
embeddings with some geometric control.  We will construct the embeddings using the h-principle for immersions and general position arguments,
and we need quantitative versions of these arguments to control the geometry of the embeddings.

Our setup will be the following.  We have $P$ a polyhedron embedded in a manifold $(M^m, g)$, and we are considering maps to a manifold $(N^m, h)$ of the 
same dimension $m$.  We start with a map $I_0: M \rightarrow N$ and we want to  
perturb the map to an embedding from some neighborhood $U \supset P$ to $N$.  We recall some basic results about this situation using immersion theory and 
general position arguments, and then we give quantitative versions.  We begin by outlining all the results, and then we come back to the proofs.

We start by recalling the h-principle for immersions.  This result was first proven by Smale and Hirsch, see \cite{EM} for references.

\newtheorem*{hprin}{The h-principle for immersions}

\begin{hprin} (a special case) Suppose that $P \subset M^m$ is a polyhedron of dimension $p \le m-1$,

and $N^m$ is a manifold of the same dimension $m$,

and $I_0: M \rightarrow N$ is a smooth map

and $T_0: TM \rightarrow TN$ is a fiberwise isomorphism covering $I_0$.  

Then there is an open neighborhood $U$ containing $P$, and an immersion $I_1: U \rightarrow N$ so that $dI_1$ is homotopic to $T_0$ through
fiberwise isomorphisms $TU \rightarrow TN$.

\end{hprin}

We want to make this result more quantitative by estimating the size of $U$ and the local bilipschitz constant of $I_1$.  The standard argument
actually gives quantitative estimates, and we will get the following.

\begin{prop} \label{quantimm1} For any $s, \mu > 0$ the following holds.  Suppose that $P \subset M^m$ is a polyhedron of
dimension $p \le m-1$, and suppose that the following hypotheses hold:

\begin{enumerate}

\item $P \subset (M^m, g)$, and the pair $(P, M)$ has bounded geometry at scale $s$,

\item $(N^m, h)$ is a Riemannian manifold with bounded geometry
at the larger scale $10^m s$,

\item $I_0: M \rightarrow N$ has Lipschitz constant at most 1,

\item $T_0: TM \rightarrow TN$ is a fiberwise isomorphism covering $I_0$, with fiberwise bilipschitz constant $\le \mu$, 

\item the scale invariant quantity $s \| \nabla T_0 \|_{C^0} \le \mu$,

\end{enumerate}

Then there are constants $L$ and $W$ that depend only on $m, \mu,$ and the bounds
on the geometry of $M, N,$ and $P$, so that the following holds.

Let $U_W$ denote the $Ws$-neighborhood of $P \subset M$.  

Then $I_0$ can be homotoped to a smooth immersion $I_1: U_W \rightarrow N$ which is locally $L$-bilipschitz.

Also, $dI_1$ is homotopic to $T_0$ in the category of fiberwise isomorphisms $TU_W \rightarrow TN$, and the distance
from $I_0(x)$ to $I_1(x)$ is $\le L s$ for each $x \in U_w$.  

\end{prop}

At the end of this subsection, we will give a precise formulation of bounded geometry.

Since this statement is rather long, we take a moment to explain why we need all the hypotheses.  Basically, we are perturbing
$I_0$ at a scale $\sim s$ in order to make $I_1$ a bilipschitz immersion
on a neighborhood of size $\sim s$.  To do this, we need to know that the geometry of the domain and range is controlled at this scale.  Also, if
$I_0$ did not have a controlled Lipschitz constant, we could not expect to make the Lipschitz constant $\lesssim 1$ by a small perturbation.
The map $T_0$ is supposed to be a model for $dI_1$, so we need to know that it is bilipschitz on each fiber.  The subtlest hypothesis is
hypothesis (5), which says
that $s \| \nabla T_0 \|_{C^0} \lesssim \mu$.  This condition prevents $T_0$ from spinning around too much.  To see that this is necessary,
consider the following example.  Suppose that $N$ is the Euclidean plane $\mathbb{R}^2$, $M \subset \mathbb{R}^2$ is the annulus defined by
$1/2 < |x| < 2$,  $P$ is the unit circle, and $I_0$ is the function zero.  Suppose that for each point $x \in M$, $T_0$
is an isometry from $\mathbb{R}^2$ to $\mathbb{R}^2$, so $T_0$ gives a map from $S^1$ to $\SO(2)$.  If $T_0$
has a high degree, wrapping $S^1$ many times around $\SO(2)$, then it's impossible to find an immersion $I_1$ which obeys the conclusions of the
proposition.

Next we would like to know if we can make $I_1$ an embedding.  This is a very difficult question in general, but there is one
simple result coming from a general position argument.  
If the dimension of $P$ is $\le (m-1)/2$, then a generic perturbation of $I_1$ is an embedding from $P$ into $N$.  Since
$I_1$ is an immersion, a generic perturbation of $I_1$ is an embedding from a small neighborhood of $P$ into $N$.  We will give
a quantitative version of this general position argument leading to the following embedding estimate.

\begin{prop} \label{quantemb1}  For any $s, \mu > 0$ the following holds.  Suppose that $P \subset M^m$ is a polyhedron of
dimension $p \le (m-1)/2$, and suppose that the following hypotheses hold:

\begin{enumerate}

\item $P \subset (M^m, g)$, and the pair $(P, M)$ has bounded geometry at scale $s$,

\item $(N^m, h)$ is a Riemannian manifold with bounded geometry
at the larger scale $10^m s$,

\item $I_0: M \rightarrow N$ has Lipschitz constant at most 1,

\item $T_0: TM \rightarrow TN$ is a fiberwise isomorphism covering $I_0$, with fiberwise bilipschitz constant $\le \mu$, 

\item the scale invariant quantity $s \| \nabla T_0 \|_{C^0} \le \mu$,

\item $I_0$ maps at most $\mu$ vertices of $P$ into any $s$-ball in $N$.

\end{enumerate}

Then there are constants $L$ and $W$ that depend only on $m, \mu,$ and the bounds
on the geometry of $M, N,$ and $P$, so that the following holds.

Let $U_W$ denote the $Ws$-neighborhood of $P \subset M$.  

Then $I_0$ can be homotoped to a smooth embedding $I: U_W \rightarrow N$ which is locally $L$-bilipschitz.

Also, $dI$ is homotopic to $T_0$ in the category of fiberwise isomorphisms $TU_W \rightarrow TN$,  and the distance
from $I_0(x)$ to $I(x)$ is $\le L s$ for each $x \in U_w$.  

\end{prop}

The hypotheses in this proposition are mostly the same as those in Proposition \ref{quantimm1}.  We assume that the dimension of
$P$ is $\le (m-1)/2$ instead of $m-1$.  Hypotheses (1)-(5) are exactly the same.  We had to add one
new hypothesis: that $I_0$ maps $\le \mu$ vertices of $P$ into any $s$-ball of $N$.  Since $I$ is supposed to be a bilipschitz embedding
from $U_W$, the images of the balls centered at vertices of $P$ with radius $W s$ will contain disjoint balls of radius $\sim s$.  
Therefore, the map $I$ cannot cram too many vertices of $P$ into a ball of radius $s$.  And since $I$ is only a small
perturbation of $I_0$, we need $I_0$ to obey this condition as well. 

To prove Proposition \ref{quantemb1}, we need to revisit the general position argument and give quantitative estimates.  This proof
is the main work involved in this section.

Proposition \ref{quantemb1} gives an embedding with quantitative control, but it doesn't tell us what isotopy class
the embedding will be in.  If $p = \Dim P = (m-1)/2$, then there will typically be infinitely many different isotopy classes
of embeddings from $P$ into $N$, and we cannot get uniform geometric control for all the isotopy classes.  But 
under the stronger condition that the dimension of $P$ is $ < (m-1)/2$, then we can construct a geometrically controlled
embedding in any isotopy class.

\newtheorem*{qel}{Quantitative embedding lemma}

\begin{qel}    For any $s, \mu > 0$ the following holds.  Suppose that $P \subset M^m$ is a polyhedron of
dimension $p < (m-1)/2$, and suppose that the following hypotheses hold:

\begin{enumerate}

\item $P \subset (M^m, g)$, and the pair $(P, M)$ has bounded geometry at scale $s$,

\item $(N^m, h)$ is a Riemannian manifold with bounded geometry
at the larger scale $10^m s$,

\item $I_0: M \rightarrow N$ has Lipschitz constant at most 1,

\item $T_0: TM \rightarrow TN$ is a fiberwise isomorphism covering $I_0$, with fiberwise bilipschitz constant $\le \mu$, 

\item the scale invariant quantity $s \| \nabla T_0 \|_{C^0} \le \mu$,

\item $I_0$ maps at most $\mu$ vertices of $P$ into any $s$-ball in $N$.

\item $I': M \rightarrow N$ is an embedding and $dI'$ is homotopic to $T_0$ in the category of fiberwise isomorphisms from $TM$ to $TN$.

\end{enumerate}

Then there are constants $L$ and $W$ that depend only on $m, \mu,$ and the bounds
on the geometry of $M, N,$ and $P$, so that the following holds.

Let $U_W$ denote the $Ws$-neighborhood of $P \subset M$.  

Then $I_0$ can be homotoped to a smooth embedding $I: U_W \rightarrow N$ which is locally $L$-bilipschitz and isotopic to $I'$,
 and the distance from $I_0(x)$ to $I(x)$ is $\le L s$ for each $x \in U_w$.  

\end{qel}

The hypotheses here are almost the same as in Proposition \ref{quantemb1}.  We assume that $\Dim P < (m-1)/2$ instead of
$\Dim P \le (m-1)/2$.  Hypotheses (1) - (6) are exactly the same.  Hypothesis (7) says that the data $(I_0, T_0)$ is homotopic
to $(I', dI')$ for some embedding $M \rightarrow N$.  In this case, we get an embedding $I$ with controlled geometry isotopic
to $I'$.  

The quantitative embedding lemma is the result that we will use to prove the h-principle for $k$-dilation. It 
follows easily from Proposition \ref{quantemb1}, so we give the proof here.  The point is that the embedding $I$
we constructed in Proposition \ref{quantemb1} is automatically isotopic to $I'$.

\begin{proof} By Proposition \ref{quantemb1}, we can construct an L-bilipschitz embedding $I: U_W \rightarrow N$ so that
$dI$ is homotopic to $dI'$ in the category of fiberwise isomorphisms from $TU_W$ to $TN$.  By the h-principle for immersions,
$I$ is regular homotopic to $I'$.  We will prove that $I$ is also isotopic to $I'$.  The point is that there is only one isotopy
class of embeddingss $U_W \rightarrow N$ regular homotopic to $I'$.

Let $h_t$ be a regular homotopy from $h_0 = I'$ to $h_1 = I$.

We can assume that $h_t$ is in general position.  Because $p < (m-1)/2$, $h_t: P \rightarrow N$
will be an embedding for every $t$.  Now we can choose some
tiny $\epsilon > 0$, so that $h_t: U_\epsilon \rightarrow N$ is an embedding for all $t$.  

Notice that $U_W$ deformation retracts into $U_\epsilon$ for any $0 < \epsilon < W$.  In more details, let $\Psi_t: U_W \rightarrow U_W$ be an isotopy with $\Psi_0$ equal to the identity and $\Psi_1(U_W) \subset U_\epsilon$.  Now we define $H_t: U_W \rightarrow N$ with $t \in [0,3]$ as follows.  For times $t \in [0,1]$, we define $H_t = h_0 \circ \Psi_t$.  For times $t \in [1,2]$, we define $H_t = h_{t-1} \circ \Psi_1$.  For times $t \in [2,3]$, we define $H_t = h_1 \circ \Psi_{3-t}$.  The maps $H_t$ give an isotopy from $I' = h_0$ to $I = h_1$. \end{proof}

In the next two subsections we will give the proofs of the propositions.  To finish this subsection, we give a precise formulation of the
phrase bounded local geometry.  We say that a Riemannian manifold has bounded local geometry at scale $s$ if each ball of radius
$s$ is diffeomorphic to a Euclidean ball of radius $s$ with bilipschitz norm $\lesssim 1$ and $C^2$ norm $\lesssim s^{-1}$.  (The definition
is scale invariant.)  We say that $P \subset M$ has bounded
geometry at scale $s$, if for each $k$-simplex we can choose the above coordinates so that the $k$-simplex
is mapped to a standard equilateral $k$-simplex in the Euclidean ball, and if in addition, each edge of $P$ has length $\le s$, 
the distance between any two disjoint closed faces of $P$ is $\gtrsim s$, and the dihedral angles of $P$ are $\gtrsim 1$.

\subsection{Constructing immersions with geometric control}

In this section, we prove Proposition \ref{quantimm1}.
This result essentially follows from the standard proof of the immersion theory by keeping track of constants.

\begin{proof}  By scaling, we may assume that the scale $s$ is equal to 1.

We write $A \lesssim B$ to mean that $A \le C B$ for a constant $C$ that only 
depends on $m, \mu,$ and the bounded local geometry of $M, N, P$.  

This construction is based on the relative h-principle for immersions.  We will use the following version of the h-principle.

Let $\Lin(\mathbb{R}^m, \mathbb{R}^m)$ denote the linear maps from Euclidean space $\mathbb{R}^m$ to itself.  Let $\Bil(L) \subset \Lin(\mathbb{R}^m, \mathbb{R}^m)$ denote the linear isomorphisms with bilipschitz constant $\le L$.  

\newtheorem*{quantimm}{Relative h-principle for immersions with quantitative control}

\begin{quantimm} Suppose that $\Delta^d \subset \mathbb{R}^m$ is a unit equilateral d-simplex in Euclidean space, with $d < m$.

Let $N_W \Delta^d$ be the $W$-neighborhood of $\Delta^d$ in $\mathbb{R}^m$.  

Suppose we have smooth maps 

$$I_0: N_{W_0} \Delta^d \rightarrow B^m(1), $$

$$T_0: N_{W_0} \Delta^d \rightarrow \Lin(\mathbb{R}^m, \mathbb{R}^m) . $$

Suppose that $T_0$ and $d I_0$ agree on $N_{W_0} \partial \Delta$.

Suppose that $T_0$ and $I_0$ have $C^1$ norms $\le A_0$.

Suppose that the image of $T_0$ lies in $\Bil(L_0) \subset \Lin(\mathbb{R}^m, \mathbb{R}^m)$.  

Then there are smooth homotopies $I: N_{W_0} \Delta \times [0,1] \rightarrow B^m(1)$ and $T: N_{W_0} \Delta \times [0,1] \rightarrow \Lin(\mathbb{R}^m, \mathbb{R}^m)$ with the following properties.

For all $t$, $I_t = I_0$ and $T_t = T_0$ on $N_{W_0/2} \partial \Delta$ and outside of $N_{2 W_1} \Delta$.

$T_1$ and $d I_1$ agree on $N_{W_1} \Delta$.

The maps $I$ and $T$ have $C^1$ norms $\le A_1$.

The image of $T_t$ lies in $\Bil(L_1)$ for all $t$.

In this theorem, the constants $W_0, L_0, A_0$ may be arbitrary, and the constants $W_1, A_1, L_1$ depend on them, as well as on $d, m$.  

\end{quantimm}

This result is essentially the standard relative h-principle for immersions.  The proof may be found in \cite{EM}, pages 21-35 and 66-68.  The only non-standard ingredient is that the constants $W_1, A_1, L_1$ only depend on the ingredients $d, m, W_0, A_0, L_0$.  This can be observed by following the proof in \cite{EM} and keeping track of the constants at each step.

Without writing out the entire proof of the h-principle for immersions, I want to try to give some explanation of why the constants are controlled.  
The proof of the h-principle for immersions in \cite{EM} is based on a fundamentally 2-dimensional construction which is then repeated several times. 
We begin with $I_0, T_0$ defined on $N_{W_0} \Delta^1 \subset \mathbb{R}^2$.  The functions may depend on other variables
also, but we can suppress the dependence and think of the other variables as just parameters.  We say that a pair $(I,T)$ is holonomic if $
dI = T$.  We choose a set of evenly spaced points along $\Delta^1$
with some spacing $\epsilon$ - a crucial small number that we will choose later.  Next, we define a rectangular block centered at each of these points 
with width (along $\Delta^1$) of $4 \epsilon$ and height $2 W_0$.  On each block, say block $J$, we perturb $(I_0, T_0)$ to a holonomic pair $(I_J,
T_J)$.  To be explicit, let us take $I_J$ be an affine function so that $(I_J, T_J) = (I_J, dI_J)$ agrees with $(I_0, T_0)$ at the center of block $J$.
The blocks overlap, so we can also define some functions that interpolate between the
$(I_J, T_J)$ on the overlaps.  We define holonomic pairs $(I_{J, J+1}, T_{J, J+1})$ which agree with $(I_J, T_J)$ in the middle part of the intersection 
(say $N_{W_0/4} \Delta$)
and agree with $(I_{J+1}, T_{J+1})$ on the outside part of the intersection (say outside of $N_{3 W_0/4} \Delta$).  We can do this by taking weighted
averages of $I_J$ and $I_{J+1}$, say $I_{J, J+1} = \rho I_J + (1 - \rho) I_{J+1}$, where $\rho$ is 1 near the middle of the intersection and on the
outside part of the intersection.  Gluing together the different
$I_J$ and $I_{J, J+1}$ we get a a function $I$ which is holonomic on $N_{W_1} \Delta$ except on some vertical slits
with spacing $\epsilon/2$  -- see the pictures on pages 27-28 of \cite{EM}.  (On the slits, $I$ is not even defined.)
At each point, the linear transformation $T_J$ is $L_0$-bilipschitz because $T_J$ is constant and it agrees with $T_0$ at one point.
Now $T_{J, J+1}$ is the derivative of a weighted average of $I_j$ and $I_{j+1}$.  If $I_j$ and $I_{j+1}$ are very close together in $C^1$, then
the weighted average will also have controlled bilipschitz constant.  Hence the map $I$ is an immersion with controlled bilipschitz constant (on
the complement of the slits).  
Finally, we precompose $I$ with a map $\phi: N_{W_0}\Delta^1 \rightarrow N_{W_0} \Delta^1$ whose image avoids the slits.  The image is a thin neighborhood
of a rapidly oscillating curve.  The resulting map is an immersion.

To get quantitative estimates, we just need to bound explicitly the characters that enter the story in terms of $W_0, A_0, L_0$.  We write
$A \lesssim B$ for $A \le C(W_0, A_0, L_0) B$.  We let $x_J$ be the center of block $J$.  Then $T_J = T_0(x_J)$, and so $|T_{J+1} - T_J| \le
A_0 \epsilon$.  Also $|I_0(x_{J+1}) - I_0(x_J)| \le A_0 \epsilon$, and so $|I_J(x) - I_{J+1}(x)| \le A_0 \epsilon + 4 L_0 \epsilon$, for each
$x$ in the overlap of block $J$ and block $J+1$.  In summary, the $C^1$ distance from $(I_J, T_J)$ to $(I_{J+1}, T_{J+1})$ is $\lesssim \epsilon$.
To check that $T_{J, J+1}$ has controlled bilipschitz constant, we compute:

$$ T_{J, J+1} = d I_{J, J+1} = d \left( \rho I_J + (1- \rho) I_{J+1} \right) = T_{J+1} + \rho (T_J -  T_{J+1}) + d \rho (I_J - I_{J+1}). $$

We have $|T_{J} - T_{J+1}| \lesssim \epsilon$, $|I_J - I_{J+1}| \lesssim \epsilon$ and $|d \rho| \le 100 W_0^{-1} \lesssim 1$.  As long as we choose
$\epsilon$ very small compared to $(1/ L_0)$ and $W_0$ and $(1/A_0)$, we see that $T_{J, J+1}$ is still $2 L_0$-bilipschitz.   This is the key step where we
choose the size of $\epsilon$ - and we see that $\epsilon$ only depends on $W_0, A_0,$ and $L_0$.  Once we have controlled
 $I_{J, J+1}$, we see that the bilipschitz constant and $C^1$ norm of $I, T$ is controlled.  Once we have picked $\epsilon$, then we know 
how closely the slits are located, and we can bound the size of arbitrarily many derivatives of $\phi$.   Given bounds on the size of 
the derivatives of $\phi$, we can then bound the bilipschitz constant and norms of $I \circ \phi$.  

To construct the map $I_1$ we must repeat this two-dimensional argument $d$ times, but each time the quantitative analysis goes like in the last paragraphs.
This finishes our explanation of why the constants $W_1, A_1, L_1$ depend only on $d, m, W_0, A_0, L_0$.  

The condition that $(I_t, T_t) = (I_0, T_0)$ for all $x$ outside of $N_{2 W_1} \Delta$ doesn't appear in \cite{EM}, but it's trivial to add.  Suppose that
$(I'_t, T'_t)$ obey all the other conditions of the theorem.  Let $\rho: N_{W_0} \Delta \rightarrow [0,1]$ be equal to 1 on $N_{W_1} \Delta$ and equal 
to zero outside of $N_{2 W_1} \Delta$.  Then set $(I_t, T_t) = (I_{\rho(x) t}', T_{\rho(x) t}')$.  This finishes our discussion of the relative h-principle
for immersions with quantitative control stated above.  Now we apply it to prove Proposition \ref{quantimm1}.

Let $I_0$ and $T_0$ be as in Proposition $\ref{quantimm1}$.  
We homotope our map $I_0$ and our initial data $T_0$ to a bilipschitz immersion by applying this result one skeleton at a time.  
We construct a sequence of (homotopic) maps, $(I_0, T_0)$, $(\bar I_0, \bar T_0)$, $(\bar I_1, \bar T_1)$, etc. with the following properties:

\begin{itemize}

\item The maps $(\bar I_j, \bar T_j)$ are all defined on a $W_0$-neighborhood of $P$.

\item All the maps $\bar I_j$ agree with $I_0$ at the vertices of $P$.

\item The $C^1$ norms of $\bar I_j$ and $\bar T_j$ are bounded by $A_j$.

\item The bilipschitz constant of $\bar T_j$ is $\le L_j$.

\item The maps $\bar T_j$ are all homotopic to $T_0$ in the category of fiberwise isomorphisms $T U_{W_j} \rightarrow TN$.

\item On a $W_j$ neighborhood of the j-skeleton of $P$, we have $d \bar I_j = \bar T_j$, and so on this neighborhood $\bar I_j$ is
an immersion with local bilipschitz constant $\le L_j$.  

\item The map $\bar I_j$ sends the $W_j$-neighborhood of each j-face into a ball of radius $\le 10^j$ in $(N, h)$.

\item The constants $A_j$ and $L_j$ are $\lesssim 1$ and $W_j \gtrsim 1$.  

\end{itemize}

Constructing $(\bar I_0, \bar T_0)$ is elementary.  

Then each homotopy from $(\bar I_{j-1}, \bar T_{j-1})$ to $(\bar I_j, \bar T_j)$ is constructed 
by using the quantitative relative h-principle on each simplex.  Since $\bar I_{j-1}$
maps a $W_{j-1}$-neighborhood of each (j-1)-simplex to a ball of radius $10^{j-1}$, it follows that it maps each $W_{j-1}$-neighborhood of each
j-simplex to a ball of radius $10^j$.  Since $N$ has bounded local geometry at scale $10^m$, we can pick a $C^2$-controlled change of coordinates from this
ball to the unit $m$-ball.  Using these coordinates and the relative h-principle for immersions with quantitative control, we homotope $\bar I_{j-1}$
to $\bar I_j$ around the given j-simplex.  This homotopy is constant except on $N_{2 W_j} \Delta \setminus N_{W_{j-1}/2} \partial \Delta$.  Because
of our control on the angles and geometry of $P$, we can choose $W_j$ so that these active regions don't overlap.

Each $\bar T_j$ is homotopic to $\bar T_{j-1}$.  The constants $A_0, L_0, W_0$ depend on
$\mu$, $m$, and the bounded local geometry of $M, P, N$.  The constants $A_j, L_j, W_j$ depend on $A_{j-1}, L_{j-1}, W_{j-1}$ 
and the bounded local geometry of $M,
N, P$.    By induction, we have all $A_j, L_j \lesssim 1$ and all $W_j \gtrsim 1$. 
In this way, we arrive at an immersion $I_1: U_{W} \rightarrow (N,h)$ which is locally $L$-bilipschitz, where $W \gtrsim 1$ and $L \lesssim 1$.  
We also see that $d I_1$ is homotopic to $T_0$ in the category of fiberwise isomorphisms.  The distance from $\bar I_{j-1}(x)$ to $\bar I_j(x)$ is
$\le 10^{j+1}$, and so the distance from $I_0(x)$ to $I_1(x)$ is $\le 10^{m+1}$.  By choosing $L$ sufficiently large, this is also less than $L$.

\end{proof}

\subsection{Constructing embeddings with geometric control}  In this subsection, we give the proof of Proposition \ref{quantemb1}. 

\begin{proof}

By scaling, we may assume that the scale $s$ is equal to 1.

We write $A \lesssim B$ to mean that $A \le C B$ for a constant $C$ that only 
depends on $m, \mu,$ and the bounded local geometry of $M, N, P$.  

By Proposition \ref{quantimm1}, we can find an immersion $I_1: U_{w_1} \rightarrow N$ which is locally $L_1$-bilipschitz,
where $w_1 \gtrsim 1$ and $L_1 \lesssim 1$.  We also know that $dI_1$ is homotopic to $T_0$ and that the distance from $I_0(x)$ to $I_1(x)$ is always
$\lesssim 1$.  Our goal is to modify $I_1$ to make it an embedding.  A general position argument
shows that a generic perturbation of $I_1$ is an embedding from $P$ to $N$, and hence from some tiny neighborhood of $P$ to $N$.  
We will make this argument more quantitative.  Here is an outline of what we will do.  (Recall that $U_w$ denotes
the $w$-neighborhood of $P$ in $M$.)

Step 1. We slightly deform $I_1$ to another immersion $I_2$, by flowing
on a vector field.  The map $I_2: U_{w_1} \rightarrow N$ still has
controlled local bilipschitz constant.  

The map $I_2$ depends on many parameters (which are used to specify the vector field).  We will prove that for some
values of these parameters, the map $I_2$ obeys the conclusion.

Step 2. The restriction of $I_2$ to a ball $B(p, r)$ is actually an embedding as long
as $p \in U_{w_1/2}$ and $r$ is sufficiently small.  This step holds for all the choices of the parameters.

Step 3. For some smaller scale $w_2 \ll w_1$, the mapping $I_2: U_{w_2} 
\rightarrow N$ is actually an embedding for some choices of the parameters
used to define $I_2$.  This last step is a quantitative version of the general position
arguments.  We check that only a bad coincidence would force $I_2$ to be non-injective, and we check that
there are some choices of the parameters which avoid all the bad coincidences.

The map $I_2$ obeys the conclusion: it embeds $U_{w_2}$ into
$N$ with controlled bilipschitz constant.  We take $W = w_2$.

\vskip10pt

{\bf Step 1: Putting $I_1$ in ``general position''}

\vskip5pt

We now perturb $I_1$ by precomposing it with the flow from a 
vector field.  

Let us pick an open cover of $\bar U_{w_1}$ using balls of radius
$w_1'$.  Here $w_1'$ is a constant that we will choose below.
We will have $w_1' \le w_1/100$ and $w_1' \gtrsim 1$.
Since $(M^m, g)$ is locally bounded at scale $1$, we can
arrange that the cover has bounded multiplicity.  We call the balls
in the cover $B_j$.  We let $\Psi_j$ be smooth non-negative functions, supported
on $B_j$, so that $\sum_j \Psi_j$ is equal to 1 on $U_{w_1}$ (and $\le 1$ everywhere).  
For $1 \le l \le m$, we let $V_{j, l}$ be vector fields defined on 
$B_j$ which are essentially orthonormal and essentially constant.

For any numbers $a_{j,l} \in [-1, 1]$, we can build the vector field
$V = \sum_{j,l} a_{j,l} \Psi_j V_{j,l}$.  Note that $| \nabla \Psi | \lesssim 1$, and so
$| \nabla V | \lesssim 1$.   We define the map $\Phi: (M, g)
\rightarrow (M,g)$ to be the time $t_{flow}$ flow of the vector field
$V$.  Here $t_{flow} \gtrsim 1$ is a small time which we will choose below.

Now we define our perturbed map $I_2$ to be $I_1 \circ \Phi$.

By choosing the flow time $t_{flow} \gtrsim 1$ sufficiently small compared to $w_1'$, 
we can arrange that the map $\Phi$ is bilipschitz with bilipschitz constant $\le 2$, 
and that it moves each point a distance $\le w_1' $.  Therefore, $\Phi$ maps $U_{w_1/2}$ into
$U_{w_1}$.  If we restrict the map $I_2$ to $U_{w_1/2}$, we
get an immersion with local bilipschitz constant $L \le 2 L_1 \lesssim 1$.

The map $I_2$ has all of the properties that we want, except that we don't know whether it's an embedding.
The embedding $\Phi: U_{w_1/2} \rightarrow U_{w_1}$ is isotopic to the identity, so $dI_2$ is homotopic to $T_0$.  It's
easy to check that the distance from $I_2(x)$ to $I_0(x)$ is $\lesssim 1$.  In steps 2 and 3, we will check that
for some values of the parameters $a_{j,l}$, the map $I_2$ restricted to $U_{w_2}$ is injective for some $w_2 \gtrsim 1$.

\vskip10pt

{\bf Step 2. Injectivity on small balls}

\vskip5pt

A bilipschitz immersion isn't always injective, but if we restrict a bilipschitz immersion to a small centrally located
ball, then the restriction is automatically injective.  We begin with a lemma about bilipschitz immersions in
Euclidean space.

\begin{lemma} \label{injlemma} Suppose that $ I $ is a locally L-bilipschitz immersion
from $B^m(1)$ into $\mathbb{R}^m$.  Then the restriction of 
$I$ to the ball of radius $r$ around 0 is an embedding for $r = (1/10)
L^{-2}$.
\end{lemma}

\proof The immersion $I$ obeys a version of the homotopy lifting
property as long as the lifts don't touch $\partial B^m$.

Suppose that $K$ is a compact polyhedron and $g_0: K \rightarrow \mathbb{R}^m$ is a continuous map
which happens to have a lift: in other words, there is a map $\tilde g_0: K \rightarrow B^m(1)$ so
that $g_0 = I \circ \tilde g_0$.  

Now let $g_t: K \rightarrow \mathbb{R}^m$ be a homotopy of $g_0$, defined for $t \in [0,1]$.

The homotopy $g_t$ always lifts to a unique homotopy $\tilde g_t: K \rightarrow B^m(1)$ for
$t$ in a small interval around 0.  Then there are two possibilities.

{\bf Case A.} The homotopy $g_t$ lifts to a homotopy $\tilde g_t$ for all $t \in
[0,1]$.

{\bf Case B.} The homotopy $g_t$ lifts to a homotopy $\tilde g_t$ for $t$ in a
maximal interval $[0, T)$, and  
for every $\epsilon > 0$, the image $\tilde g_t(K)$ touches the
$\epsilon$-neighbhorhood of $\partial B^m(1)$ for some $t < T$.

Now, suppose that $I$ is not an embedding on $B(r)$.  In
that case, there are two distinct points $x,y \in B(r)$ with $I(x)
= I(y)$.  By translating $I$, we may suppose $I(x) = I(y) = 0$.

Now let $\tilde g_0: [0,1] \rightarrow B^m(1)$ parametrize the segment
from $x$ to $y$, with $\tilde g_0(0)=x$ and $\tilde g_0(1) = y$.  This segment
has length at most $2r $. Then we let $g_0 = I \circ \tilde g_0: [0,1] \rightarrow
\mathbb{R}^m$.  The function $g_0$ parametrizes a curve in
$\mathbb{R}^m$ with $g_0(0) = g_0(1) = 0$.  The length of the
curve is at most $2 r L$.

Next we homotope $g_0$ to zero by rescaling.  We define $g_t(s) =
(1-t) g_0(s)$ for $t \in [0,1]$.  We see that $g_t(0) = g_t(1)
= 0$ for all $t$ and that $g_1(s) = 0$ for all $s$.  The length of
the curve parametrized by $g_t$ decreases monotonically, and so it is always 
$ \le 2 r L$.

Now we consider the lifts $\tilde g_t$ of $g_t$.  These lifts exist on some
interval $[0, T)$ or $[0,1]$.   For every $t$ where the lift $\tilde g_t$ is
defined, $\tilde g_t(0) = x$
and $\tilde g_t(1) = y$.  Moreover, each lift has length at most $2
r L^2$.  Because $r = (1/10) L^{-2}$, each lift
has length at most $(1/5)$.  Also, $r \le 1/10$, and
so $x$ and $y$ lie in $B(1/10)$.  Now $\tilde g_t$ parametrizes a curve from $x$
to $y$ of length at most $(1/5)$.  This curve must lie entirely in
$B(1/2)$.

Because of this bound, Case B above is excluded.   Therefore, we
can define lifts $\tilde g_t$ for all $t \in [0,1]$.  But $\tilde g_1$ is a
lift of the constant curve $g_1$, and so $\tilde g_1$ is a constant. 
However, $\tilde g_1(0) = x$ and $\tilde g_1(1) = y$.  This contradiction shows
that $I$ is an embedding on $B(r)$ as claimed.  \endproof

At this point we may choose the constant $w_1'$.  In Step 1, we needed to know
that $w_1' \le w_1/100$.  We define $w_1'$ to be the much smaller number
$w_1' = 10^{-6} L^{-3} w_1$, where $L$ is the local bilipschitz constant of $I_2$.  
We know that $w_1 \gtrsim 1$ and $L \lesssim 1$, and so $w_1' \gtrsim 1$.

\begin{lemma}
Suppose that $p \in U_{w_1/4}$ and $r = 100 w_1' $.  Then the
restriction of $I_2$ to $B(p,r)$ is an embedding.
\end{lemma}

\begin{proof} Let $R = 100 L^2 r$.  Because
$r = 100 w_1' \le 10^{-4} L^{-3} w_1$, we see that $R \le L^{-1} w_1/100$.  In particular $B(p, R)$ is contained in $U_{w_1/2}$.
Therefore, we know that $I_2: B(p, R) \rightarrow N$ is a locally $L$-bilipschitz immersion.  The image $I_2 ( B(p,R))$
must lie in a ball in $N$ of radius $\le L R \le w_1/100$.  Because of the bounded local geometry of $M$ and $N$, we know
that $B(p,R)$ and this target ball in $N$ are each 2-bilipschitz to balls in Euclidean space.  In particular, we can use geodesic coordinates centered at $p$ on the ball $B(p, R)$.  In these coordinates, $B(p, R)$ is mapped to $B^m(R)$ and $B(p, r)$ is mapped to $B^m(r) \subset B^m(R)$.  The resulting map from $B^m(R)$ to $\mathbb{R}^m$ is $4 L$ bilipschitz.  Since $r  <  (1/10) (4 L)^{-2} R$, Lemma \ref{injlemma} implies that this map restricted to $B^m(r)$ is an embedding.  Therefore, the map $I_2: B(p,r) \rightarrow N$ is an embedding.  \end{proof}

\vskip10pt

{\bf Step 3. General position estimates}

\vskip5pt

Now we restrict $I_2$ to $U_{w_2}$ for $w_2 = \epsilon w_1'$.  The number
$\epsilon > 0$ is a small constant that we will choose later.  (Eventually, we will choose $\epsilon \gtrsim 1$, but
until we choose $\epsilon$, we write lemmas that hold for every $\epsilon > 0$.)

Recall that in Step 1, we defined a cover of $U_{w_1}$ by balls $B_j$
of radius $w_1'$.  Now we choose a cover of $U_{w_2}$ with
balls $B'_k$ of radius $w_2 = \epsilon w_1'$.  We can choose a covering with bounded multiplicity.  Because
$P$ has dimension $p$, we see that for each $j$, $U_{w_2} \cap B_j$ is covered by
$\lesssim \epsilon^{-p}$ balls $B'_k$.

By Step 2, we know that $I_2$ restricted to $B(p,r)$ is injective as long as
$p \in U_{w_1/4}$ and $r \le 100 w_1'$.  In particular, $I_2$ is injective on each
ball $B_j$ that intersects $U_{w_2}$.  

Recall that the map $I_2$ depends on the parameters $a_{j,l}$.  We will prove that $I_2: U_{w_2} \rightarrow N$ is an embedding
for some choice of the parameters $a_{j,l}$ and for some $\epsilon \gtrsim 1$.  
Let us define $\Bad(j_1, j_2)$ to be the set of parameters so that there exists $x_1 \not= x_2$ and $I_2(x_1) = 
I_2(x_2)$, where $x_1 \in U_{w_2} \cap B_{j_1}$ and $x_2 \in U_{w_2} \cap B_{j_2}$.
If $a_{j,l}$ are parameters so that $I_2$ is not injective, then the parameters must lie
in one of the sets $\Bad(j_1, j_2)$.

We can assign a probability measure to the set of parameter choices, by choosing each
parameter $a_{j,l}$ uniformly at random in $[-1, 1]$.  With this probability measure, we will
prove that each set $\Bad(j_1, j_2)$ is small.  This is the key step of the proof, where we use
the condition that $p \le (m-1)/2$.

\begin{lemma} The measure of $\Bad(j_1, j_2)$ is $\lesssim \epsilon$.  

\end{lemma}

\begin{proof} Let us define $\Bad(k_1, k_2)$ to be the set of parameters so that
there exists $x_1 \not= x_2$ and $I_2(x_1) = I_2(x_2)$, where $x_1 \in B'_{k_1}$ and $x_2 \in B'_{k_2}$.
We are going to prove that the probability of each $\Bad(k_1, k_2)$ is $\lesssim \epsilon^m$.  We give a rough intuition why
this is true.  The images $I_2(B'_{k_1})$ and $I_2(B'_{k_2})$ are approximately balls of radius $\sim \epsilon$.  
When we change the parameters, we randomly move
these balls in $N$ a distance $\sim 1$.  So the probability that they intersect is $\lesssim \epsilon^m$.  

Now $\Bad(j_1, j_2)$ can be covered by $\cup \Bad(k_1, k_2)$, taking the union over all $B'_{k_1}$ that intersect $B_{j_1}$
and all $B'_{k_2}$ that intersect $B_{j_2}$.  Because $P$ has dimension $p$, the number of balls $B_{k_1}$ that intersect
$B_{j_1}$ is $\lesssim \epsilon^{-p}$.  Therefore, the probability of $\Bad(j_1, j_2)$ is $\lesssim \epsilon^{m - 2p}$.  
Because $p \le (m-1)/2$, this probability is $\lesssim \epsilon$. 

It remains to carefully prove that for each $k_1, k_2$, the probability of $\Bad(k_1, k_2)$ is $\lesssim \epsilon^m$.  
If the distance from $B'_{k_1}$ to $B'_{k_2}$ is $< 40 w_1'$, then by Step 2, $I_2$ is always injective
on a ball containing $B'_{k_1}$ and $B'_{k_2}$.  In this case $\Bad(k_1, k_2)$ is empty and it has probability zero.  

If the distance from $B'_{k_1}$ to $B'_{k_2}$ is $\ge 40 w_1'$, then we can find a ball $B_{j_0}$ so
that $\Psi_{j_0} \gtrsim 1$ on $B'_{k_1}$ and $B_{j_0}$ is far from $B'_{k_2}$.  We fix $j_0$, and we 
consider changing the parameters $a_{j_0,l}$ while holding fixed all the other parameters (the $a_{j,l}$ with $j \not= j_0$). 

Changing the parameters $a_{j_0,l}$ affects the vector field $V$ only on the ball $B_{j_0}$.  Since the flow $\Phi$ moves
each point at most $w_1'$ (for any choice of parameters), it follows that changing $a_{j_0,l}$ affects $\Phi(x)$ only if $x$
lies in the $w_1'$-neighborhood of $B_{j_0}$.  Therefore, changing $a_{j_0,l}$ affects $I_2(x)$ only if $x$ lies in the
$w_1'$-neighborhood of $B_{j_0}$.  Since the distance from $B_{k_2}'$ to $B_{j_0}$ is at least $35 w_1'$, 
we see that changing the parameters $a_{j_0, l}$ does not change $I_2$ on $B_{k_2}'$.  

Since $I_2$ is $L$-Lipschitz, the image $I_2(B_{k_2}')$ is contained in a ball of radius $\le L w_2 = L \epsilon w_1' \lesssim \epsilon$.  
Let $2 B_{j_0}$ be the ball with the same center as $B_{j_0}$ and twice the radius.  Since $B_{k_1}'$ intersects $B_{j_0}$, $B_{k_1}'$ is
totally contained in $2 B_{j_0}$.  The map $I_1: 10 B_{j_0} \rightarrow N$ is a locally $L$-bilipschitz embedding.  
Therefore, $I_1^{-1} (I_2 (B_{k_2}') \cap 2 B_{j_0}$ is contained
in a ball of radius $\lesssim L \epsilon \lesssim \epsilon$ - call this ball $B(\epsilon)$.  Notice that $B(\epsilon)$ doesn't depend
on the parameters $a_{j_0, l}$.  

If $I_2(B_{k_1}')$ intersects $I_2(B_{k_2}')$, then
$\Phi( B_{k_1}')$ must intersect $B(\epsilon)$.  Since $\Phi$ is 2-bilipschitz, $\Phi$ must map the center of the ball $B_{k_1}'$ into 
the double of $B(\epsilon)$.  Let $x_0$ denote the center of $B_{k_1}'$.  
Recall that we randomly choose the parameters $a_{j_0,l}$ for $1 \le l \le m$, and fix $a_{j,l}$ for
all $j \not= j_0$.  We have to prove that the probability that $\Phi(x_0)$ lies in a fixed ball of radius $\lesssim \epsilon$ is
$\lesssim \epsilon^m$.

We pick coordinates for $10 B_{j_0}$ so that the vector fields $V_{j_0, l}$ are just the coordinate vector fields $\partial_l$.  Then
we can write $V$ in these coordinates as $V_0 + \Psi \vec{a}$, where $\Psi = \Psi_{j_0}$ and $\vec{a}$ is the vector with components
$a_{j_0, l}$.  Then we define $\Phi^t_{\vec{a}}(x)$ to be the result of the time $t$ flow of the vector field $V = V_0 + \Psi \vec{a}$ 
with initial condition $x$.  By the fundamental theorem of calculus in this coordinate chart, we see that

$$ \Phi_{\vec{a}}^t (x_0) = x_0 + \int_0^t \left[ V_0(\Phi_{\vec{a}}^s(x_0)) + \Psi(\Phi_{\vec{a}}^s(x_0)) \vec{a} \right] ds. \eqno{(*)}$$

Subtracting and taking norms, we get

$$ \sup_{0 \le s \le t} | \Phi_{\vec{a}}^s (x_0) - \Phi_{\vec{b}}^s(x_0) | \le t ( | \nabla V_0| + | \nabla \Psi | )  \sup_{0 \le s \le t} | \Phi_{\vec{a}}^s (x_0) 
- \Phi_{\vec{b}}^s(x_0) | + t |\vec{a} - \vec{b} |. $$

Therefore, there exists a time $t_0 \gtrsim 1$ so that for all $t \le t_0$ and all $\vec{a}, \vec{b}$ with components $\le 1$, we have 

$$ | \Phi_{\vec{a}}^t (x_0) - \Phi_{\vec{b}}^t(x_0) | \le 2 t |\vec{a} - \vec{b} |. $$

Plugging this back into formula (*), we see that there is a smaller time $t_1 \gtrsim 1$ and a constant $c_1 \gtrsim 1$ 
so that for all $t \ge t_1$ and all $\vec{a}, \vec{b}$ with components $\le 1$ we have

$$ | \Phi_{\vec{a}}^t (x_0) - \Phi_{\vec{b}}^t(x_0) | \ge c_1 t |\vec{a} - \vec{b}|. $$

We choose the flow time $t_{flow}$ in the definition of $\Phi$ so that both these bounds hold.  Now the choice of all possible $\vec{a}$ so that
$\Phi_{\vec{a}}^t(x_0)$ lies in the bad target $B(2 \epsilon)$ is contained in a ball of radius $\lesssim \epsilon$ and probability
$\lesssim \epsilon^m$.  So the probability that $\Phi( B'_{k_1})$ intersects $B(\epsilon)$ is $\lesssim \epsilon^m$.  Therefore,
the probability of $\Bad(k_1, k_2)$ is $\lesssim \epsilon^m$.   \end{proof}

This lemma is useful, but there is still some ways to go in our proof.  We have shown that each set $\Bad(j_1, j_2)$ is small,
but we have no control over the number of sets $\Bad(j_1, j_2)$.

Our situation can be described as follows.  Suppose that $X = \prod_{i \in I} X_i$ is a (countable or finite) product of probability spaces.  Suppose that $\Bad \subset X$ is a ``bad" set, consisting of a union $\Bad = \cup_\alpha \Bad_\alpha$.  We would
like to find a not-bad element of $X$ i.e. an element $x \in X$ which is not in $\Bad$.  We know that the measure (probability) of each $\Bad_{\alpha}$ is less than $\epsilon$ a small number.  But,
we have no control over the number of sets $\Bad_{\alpha}$.   We can still find an element outside of $\Bad$ provided that the sets $\Bad_\alpha$ are ``localized" in the following sense.

\begin{lemma} Suppose that $\Bad$ is the union of sets $\Bad_\alpha$ each with probability less than $\epsilon$.  
Suppose that each set $\Bad_{\alpha}$ depends on $< C_1$ different coordinates
$x_i$ of the point $x$.  Suppose that each variable is relevant for $< C_2$ different bad
sets $\Bad_{\alpha}$.  If $\epsilon < (1/2) C_2^{- C_1}$, then $\Bad$ is not all of $X$.
\end{lemma}

We give a proof of this probability lemma in Section \ref{secproblemma}.

In order to apply this lemma, we must estimate the constants $C_1$ and $C_2$ in our situation.

\begin{lemma} \label{C_1est} Suppose that $\Bad(j_1, j_2)$ depends on the value of a parameter $a_{j,l}$.  Then the distance
from $B_j$ to $(B_{j_1} \cup B_{j_2})$ is $\le w_1'$.

Therefore, each set $\Bad(j_1, j_2)$ depends on $\lesssim 1$ parameters $a_{j,l}$.
\end{lemma}

\begin{proof} Recall that $I_2 = I_1 \circ \Phi$, and that $\Phi$ moves each point a distance $\le w_1'$.    The parameter $a_{j,l}$ only affects $V$ on $B_j$.  If the distance from $x$ to $B_j$ is $> w_1'$, then $\Phi(x)$ will not depend on the parameter $a_{j,l}$.  Therefore, $I_2(x)$ will not depend on the parameter $a_{j,l}$.  So if the parameter $a_{j,l}$ affects $\Bad(j_1, j_2)$, then there must be a point $x$ in either $B_{j_1}$ or $B_{j_2}$ which lies a distance $\le w_1'$ from $B_j$.  \end{proof}

Next, we have to estimate the number of different sets $\Bad(j_1, j_2)$ which are influenced by a single parameter $a_{j,l}$.  As
a first step we prove the following lemma.

\begin{lemma} If $\Bad(j_1, j_2)$ is non-empty for some choice of the parameters $a_{j,l}$, then 
the distance from $I_0(B_{j_1})$ to $I_0(B_{j_2})$ is $\lesssim 1$.
\end{lemma}

\begin{proof} We saw above that $\Dist( I_0(x), I_2(x) ) \lesssim 1$.  If $\Bad(j_1, j_2)$ is non-empty (for some parameters $a_{j,l}$), then we can 
find $x_1 \in B_{j_1}$ and $x_2 \in B_{j_2}$ with $I_2(x_1) = I_2(x_2)$.  Then we conclude that the distance from $I_0(x_1)$ to 
$I_0(x_2)$ is $\lesssim 1$. \end{proof}

\begin{lemma} Each parameter $a_{j,l}$ influnces $\lesssim 1$ bad sets $\Bad(j_1, j_2)$.
\end{lemma}

\begin{proof}  Fix a parameter $a_{j,l}$.  Suppose that $\Bad(j_1, j_2)$ depends on $a_{j,l}$.  By Lemma \ref{C_1est}, we see that $B_j$ lies fairly close
to either $B_{j_1}$ or $B_{j_2}$.  After changing the labels, we can assume that the distance from $B_j$ to $B_{j_1}$ is $\lesssim w_1'$.  This leaves only $\lesssim 1$ choices for $j_1$.  Let us fix a choice of $j_1$, and consider how many choices we have for $j_2$ so that $\Bad(j_1, j_2)$ depends on $a_{j,l}$. 

If $I_0(B_{j_1})$ and $I_0(B_{j_2})$ are far apart, then the last lemma tells us that $\Bad(j_1, j_2)$ is empty.  (And if $\Bad(j_1, j_2)$ is empty, 
it does not depend on $a_{j,l}$.)  So we only need to consider $j_2$ so that the distance from $I_0(B_{j_1})$ to $I_0(B_{j_2})$ is $\lesssim 1$.   In other
words, we just need to count the number of $j_2$ so that $I_0$ maps $B_{j_2}$ into a certain ball of radius $\lesssim 1$.
We will show that the number of such $j_2$ is $\lesssim \mu \lesssim 1$.  We know that $I_0$ maps at most $\mu \lesssim 1$ vertices of $P$ into any unit ball of $N$.  Since $N$ has bounded geometry at scale 1, it follows that $I_0$ maps $\lesssim 1$ vertices of $P$ into any ball of radius $\lesssim 1$.
Then by the bounded geometry of $P$, and since $I_0$ is Lipschitz, it follows that $I_0$ maps $\lesssim 1$ balls $B_{j_2}$ into any ball of radius
$\lesssim 1$.  \end{proof}

Now we can finish the proof of the embedding proposition.

If we choose $\epsilon$ small enough, then the probability lemma guarantees that we can find a choice of parameters $a_{j,l}$ which is not in any bad set $\Bad(j_1, j_2)$.  Therefore, $I_2: U_{w_2} \rightarrow N$ is an embedding.  Plugging in our inequalities, we see that we can choose $\epsilon \gtrsim 1$, and so $w_2 \gtrsim 1$ also.  The number $w_2$ is the $W$ from the statement of the proposition.  The embedding $I_2: U_{w_2} \rightarrow N$ is locally $L$-bilipschitz for $L \lesssim 1$.  

\end{proof}

\section{An h-principle for $k$-dilation}

In this section, we prove the h-principle for $k$-dilation stated in the introduction.

\begin{hp} Suppose that $F_0$ is a map from $S^m$ to $S^n$ with $m > n$, and that $k > (m+1)/2$.
Then for any $\epsilon > 0$, we can homotope $F_0$ to a map $F$ with $k$-dilation less than
$\epsilon$.
\end{hp}

\subsection{Zeldovitch's construction of a thick tube} \label{zeltube}

Our proof of the h-principle is based on Zel'dovitch's construction of thick tubes in $B^3(1)$.  As motivation, and in
order to describe the main ideas, we outline Zel'dovitch's construction here.  Zel'dovitch was an astrophysicist who was studying the
motion of magnetized fluid in neutron stars, and his physical problem led him to the following construction.  His work is
described more in the paper \cite{Ar}.

We won't use the results from this section anywhere, so we just sketch the main ideas.  Zel'dovitch's construction gives an alternate
proof of the thick tube example from Section \ref{twistedtubes}.

\begin{GZ} (Zel'dovitch) For any radius $R$, there is some $\delta = \delta(R) > 0$ and a 2-expanding embedding from $S^1(\delta)
\times B^2(R)$ into the unit 3-ball.
\end{GZ}

Let $\{ Q_i \}$ be a collection of small disjoint squares in $B^2(R)$ with side length $\delta$, filling 
most of the area of $B^2(R)$.  We want to build a 2-expanding embedding from $I: S^1(\delta) \times B^2(R)$ into $B^3(1)$.  
Let's try to first construct $I$ on $S^1(\delta) \times Q_i$.  Notice that $S^1(\delta) \times Q_i$ is basically
$S^1(\delta) \times B^2(\delta)$, and there is a 1-expanding embedding from this tube into a ball of radius $\sim \delta$.  
We have $\sim R^2 \delta^{-2}$ such tubes in $S^1(\delta) \times B^2(R)$, and there are $\sim \delta^{-3}$ such balls in
$B^3(1)$.  We choose $\delta$ small enough so that the number of balls is larger than the number of tubes, and then
we embed each tube in a ball.  In this way we can construct a 2-expanding embedding $I$ from $\cup_i S^1(\delta) \times Q_i$
into $B^3(1)$.  We can even arrange that $I$ increases all areas by a factor of 10.

So far we have defined $I$ on $\cup_i S^1(\delta) \times Q_i$.  We have to extend it to an embedding on the whole domain.  Why is
it possible to do this?  If we let $I_0$ be a standard embedding from $S^1(\delta) \times B^2(R)$ into the 3-ball (unknotted and
with twisting number zero), then the images of the small tubes will be unlinked, and we can isotope $I_0$ to our embedding $I:
\cup_i S^1(\delta) \times Q_i \rightarrow B^3(1)$.  Therefore, our map $I$ extends to some embedding from $S^1(\delta) \times B^2(R)
\rightarrow B^3(1)$.

This embedding $I$ is 2-expanding on $\cup_i S^1(\delta) \times Q_i$, but it has terrible properties on the complement of this
region.  We call the complement of $\cup_i S^1(\delta) \times Q_i$ the interstitial
region.  Since $I$ is an embedding, there is some number $\beta(I) > 0$ so that each surface of area $A$ in the domain is mapped to
a surface of area $\ge \beta(I) A$, but we have no estimate for $\beta(I)$.  
We can fix this problem by squeezing the interstital region in the following way.

For any $\epsilon > 0$, there is a diffeomorphism $\Psi_\epsilon: S^1(\delta) \times B^2(R) \rightarrow S^1(\epsilon \delta) \times B^2(R)$ 
with Lipschitz constant $\le 2$ and with 2-dilation $\lesssim \epsilon$ on the interstitial region.  This diffeomorphism is just a product
of a map in each factor.  The map on the circle is just a rescaling.  In the map $B^2(R) \rightarrow B^2(R)$, 
the squares $Q_i$ grow a bit, and the interstitial region between them shrinks to a very thin neighborhood of a graph.

If we choose $\epsilon$ small enough, then $I \circ \Psi_\epsilon^{-1}: S^1(\epsilon \delta) \times B^2(R) \rightarrow B^3(1)$ will be a 2-expanding 
embedding.  If we take a small surface in $\cup_i S^1(\delta) \times Q_i$, then $\Psi_{\epsilon}^{-1}$ doesn't compress areas by more than
a factor of $2^2=4$, and $I$ expands areas by a factor of 10.  If we take a small surface in the interstitial region with area $A$, then $I
\circ \Psi_\epsilon^{-1}$ maps the surface to an image with area at least $\beta(I) \epsilon^{-1} A$.  As long as we choose $\epsilon$ small
enough, we are done.

Next we ask if we can construct an embedding with non-zero twisting number.  Here the answer is no.  The problem is that in 
an embedding with non-zero twisting number the images of the smaller tubes $S^1(\delta) \times Q_i$ would have to be linked
with each other, and so they couldn't lie in different balls.

But in dimension $m \ge 4$, the smaller tubes would be unlinked and the construction above would go through, giving an (m-1)-expanding
embedding.  If $m \ge 4$, then for every $R > 1$, there is some $\delta(R)$, and an $(m-1)$-expanding embedding from $S^1(\delta)
\times B^{m-1}(R)$ into $B^m(1)$ with twisting number 1.  Now using this thick tube and the 
Pontryagin-Thom collapse, we get homotopically non-trivial maps from $S^m$ to $S^{m-1}$ with arbitrarily small $(m-1)$-dilation.

In the proof of the h-principle, we will generalize this method.  Instead of $S^1 \times B^{m-1}$, we will work more generally with
$Y \times B^n$ for an $(m-n)$-dimesional manifold $Y$.  Instead of a union of cubes $\cup_i Q_i$ - which is a neighborhood of a 0-dimensional
polyhedron -- we will work with neighborhoods of higher dimensional polyhedra in $B^n$.

\subsection{An outline of the proof} \label{outlinehprinc}

We will construct a degree 1 map $\Psi: S^n \rightarrow S^n$, and a degree 1 map
$G: S^m \rightarrow S^m$, and the map $F$ will be $\Psi \circ F_0 \circ G$.  Since degree 1 maps of
spheres are homotopic to the identity, we see that $G$ and $\Psi$ are each homotopic to the identity,
and so $F$ is homotopic to $F_0$.  By choosing $G$ and $\Psi$ judiciously, we will arrange that 
$\Dil_k(F) < \epsilon$.

Our construction depends on a small parameter $\delta > 0$.  The maps $\Psi$ and $G$ depend
on $\delta$, and so the final map $F$ depends on $\delta$.  We will show that as $\delta \rightarrow 0$,
$\Dil_k(F) \rightarrow 0$ also.  We have to keep track of how the dilations and other geometric quantities depend
on $\delta$.  We write $A \lesssim B$ if $A \le C(F_0) B$, where $C(F_0)$ is a constant that may depend on
the map $F_0$ but does not depend on $\delta$.

We can assume that $F_0$ is smooth.  We let $y_0 \in S^n$ be a regular value of $F_0$.  We let $Y = F_0^{-1}(y_0) \subset S^m$.  So $Y$
is a submanifold of dimension $m-n$.  Now we let $B_r(y_0)$ be the ball around $y_0$ with a small
radius $r \gtrsim 1$.  By choosing $r$ appropriately small, we can be sure that 
$F_0^{-1} (B_r(y_0))$ is diffeomorphic to $Y \times B_r(y_0)$.  
We can choose a map $\pi_Y: F_0^{-1}(B_r(y_0)) \rightarrow Y$ so that $\pi_Y \times F_0: F_0^{-1}(B_r(y_0)) \rightarrow
Y \times B_r(y_0)$ is a diffeomorphism.  

Since $r$ is small, $B_r(y_0)$ is nearly Euclidean, and we choose an identification with $B^n_r \subset \mathbb{R}^n$.
We let $Q^{n-k} \subset B^n_r$ be the (n-k)-skeleton of the 
cubical grid with side length $\delta$ intersected with $B^n_r$.  (More precisely, $Q$ is the union of all the faces of cubical grid of dimension $\le n-k$ which lie entirely in $B^n_r$.)  We let $V_W$ be the $W \delta$-neighborhood of $Q^{n-k} \subset B_r(y_0)$, where
$W > 0$ is a small constant depending on $F_0$ which we will choose later.  The constant $W$ will be independent of $\delta$ and so $W \gtrsim 1$.  Using our identification of $B^n_r$ with $B_r(y_0)$, we can think of $V_W$ as an open subset of $B_r(y_0) \subset S^n$. 

Now we can describe the map $\Psi: S^n \rightarrow S^n$.

\begin{lemma} \label{squeeze}  We will construct a degree 1 map $\Psi$ from $S^n$ to $S^n$ with the following properties.
On the set $V_W \subset S^n$, the 1-dilation of $\Psi$ is $\lesssim 1$.  
On the complement of $V_W$, the $k$-dilation of $\Psi$ is identically zero.  This happens because $\Psi$
maps the complement of $V_W$ into a $(k-1)$-dimensional subset of $S^n$.
\end{lemma}

We chose the dimension of $Q$ to be $n-k$ in order to make this lemma work.  The complement of $V_W$ is a neighborhood of a polyhedron of dimension $k-1$.  The map $\Psi$ retracts the complement of $V_W$ onto this polyhedron, while $V_W$ gets thicker in order to fill the vacated region.  

We define $U_W := F_0^{-1}(V_W) \subset S^m$.  On the complement of $U_W$, Lemma \ref{squeeze} implies that the $k$-dilation of $\Psi \circ F_0$ is zero.  So we only have to worry about what happens on $U_W$.  We use the map $G$ in order to deal with this region.  To explain our strategy, we need to define
an auxiliary metric on $U_W$.

We let $h_0$ be the unit sphere metric on $S^n$, and hence on $V_W \subset B_r(y_0) \subset S^n$.  We let $g_0$ be the unit sphere metric on $S^m$ and $g_Y$ be the restriction of $g_0$ to $Y$.  The set $U_W$ is diffeomorphic to $Y \times V_W$.  We let $g_1$ be the metric $\delta^2 g_Y + h_0$.  More precisely,
the map $\pi_Y \times F_0$ is a diffeomorphism from $U_W$ to $Y \times V_W$, and we define $g_1$ so that $\pi_Y \times F_0$ is an isometry from $(U_W, g_1)$ to $(Y, \delta^2 g_Y) \times (V_W, h_0)$.  In particular, the map $F_0: (U_W, g_1) \rightarrow (V_W, h_0)$ has 1-dilation equal to 1.

Now we are ready to describe the degree 1 map $G: S^m \rightarrow S^m$.

\begin{lemma} \label{precomp}  If $k > (m+1)/2$, then there exists $W \gtrsim 1$ and a degree 1 map $G$ from $S^m$ to $S^m$ with the following property.  
If we view $G$ as a map from $(G^{-1}(U_W), g_0)$ to $(U_W, g_1)$, then
it has 1-dilation bounded by $\lesssim \delta^a$ for some exponent $a > 0$.
\end{lemma}

With these two lemmas, we prove the h-principle.  We define $F$ to be the composition $\Psi \circ F_0 \circ G$.  Since $\Psi$ and $G$ are degree 1 maps of spheres, they are each homotopic to the identity and so $F$ is homotopic to $F_0$.  Now we estimate the $k$-dilation of the map $F$.  

The $k$-dilation of $F$ is the supremum of $|\Lambda^k dF_x|$.  We consider two cases, depending on whether $x$ lies in $G^{-1}(U_W)$.  If $x$ lies in $G^{-1}(U_W)$, then we proceed as follows.  We view $F$ as a composition of maps

$$(G^{-1}(U_W), g_0) \rightarrow (U_W, g_1) \rightarrow (V_W, h_0) \rightarrow (S^n, h_0).$$

\noindent (The first map is $G$, the second map is $F_0$, and the last map is $\Psi$.)  
For the first map, the derivative $dG_x$ has norm $\lesssim \delta^a$ by Lemma \ref{precomp}.  For the second map, the norm
of the derivative is $\le 1$ by the definition of $g_1$.  For the third map, the norm of the derivative is $\lesssim 1$ by Lemma \ref{squeeze}.  Therefore, the derivative $dF_x$ has norm $\lesssim \delta^a$.  By making $\delta$ small, we can arrange that
$| dF_x|$ and $| \Lambda^k dF_x|$ are as small as we like.
 
Next we consider the case that $x$ does not lie in $G^{-1}(U_W)$.  In this case we have no control over
$dG_x$.  But we know that $G(x)$ does not lie in $U_W$ and so $F_0(G(x))$ does not lie in $V_W$.  Therefore, $\Lambda^k d \Psi_{F_0(G(x))}$ is zero.  And so $\Lambda^k dF_x$ is
zero also.  In summary, in each of the two cases, $|\Lambda^k dF_x| < \epsilon$, and so $\Dil_k(F) < \epsilon$.  

Now we discuss the construction of the map $G$.  The map $G$ is in fact a diffeomorphism, and we will construct its inverse $G^{-1}$.  The main task is to define $G^{-1}$ on the set $U_W$.  We construct it using the following lemma.

\begin{lemma} \label{emb1}  If $k > (m+1)/2$, then there is a constant $W \gtrsim 1$, and an embedding $I: (U_W, g_1) \rightarrow (S^m, g_0)$ which is isotopic to the inclusion $U_W \subset S^m$, and which increases all lengths by a factor $\gtrsim \delta^{-a}$.  Here $a = \frac{m-n}{m} > 0$.
\end{lemma}

Given this lemma, the construction of $G$ is straightforward.  Since $I$ is isotopic to the inclusion map $U_W \subset S^m$, we can extend $I$ to a diffeomorphism from $S^m$ to $S^m$.  This diffeomorphism is $G^{-1}$.  Because $I$ is very expanding on $U_W$, it follows that $G$ is very contracting on $G^{-1}(U_W)$.  So Lemma \ref{emb1} implies Lemma \ref{precomp}.

Lemma \ref{emb1} is the main step in the proof of the h-principle.  It is proven by a quantitative general position argument.  The set $U_W$ is a small neighborhood of a polyhedron $P$ of dimension $p = m-k$.  The condition $k > (m+1)/2$ implies that $p < (m-1)/2$.  By a standard general position argument, any two embeddings from $P$ into $S^m$ are isotopic.  The quantitative embedding lemma from Section \ref{sectionquantemb} allows us to construct embeddings
from neighborhoods of $P$ with geometric control of the embedding.  With a little work, we will see that Lemma \ref{emb1} follows from the quantitative
embedding lemma.

Before turning to the proofs, we describe the construction in a simple example, to try to help the reader visualize the
situation.

Let us suppose that $k =n $ and that $m = n+1$.  Since $k=n$ the dimension of $Q$ is zero.  So $Q$ is simply a grid of points 
with spacing $\delta$, with a total of $\sim \delta^{-n}$ points.  The set $V_W$ is simply a union of balls, centered at the points of $Q$ and with radius $W \delta$.

Since $m = n+1$, the set $Y = F_0^{-1}(y_0)$ is a compact 1-dimensional manifold in $S^m$.  So $Y$ is a union of circles.  For simplicity, let's consider the case that $Y$ is a single circle.  The circle $Y$ is independent of $\delta$, so the length of $Y$ is $\sim 1$.  Now the set $U_W = F_0^{-1}(V_W)$ consists of $\sim \delta^{-n}$ cylinders.  Each cylinder is diffeomorphic to $S^1 \times B^n$.  Geometrically, each cylinder is roughly a product of a circle of length $\sim 1$ and an $n$-ball of radius $\sim \delta$.  

So far, we have discussed the geometry of $U_W$ in the metric $g_0$.  In the metric $g_1$, each cylinder is a product of a circle with length $\sim \delta$ by a ball with radius $\sim \delta$.  The metric $g_1$ is much smaller than $g_0$ in the direction along the fibers of $F_0$, and it approximately agrees with $g_0$ perpendicular to the fibers of $F_0$.

We can construct an embedding from $(U_W, g_1)$ into $(S^m, g_0)$ as follows.  First find $\delta^{-n}$ disjoint balls inside of $(S^m, g_0)$.  Each
ball will have radius $\sim \delta^{\frac{n}{n+1}}$.  Embed each tube of $(U_W, g_1)$ into one of these balls.  In the metric $g_1$, the tube has length $\sim \delta$ and thickness $\sim \delta$.   But in the target ball, the image will be a tube of length $\sim \delta^{\frac{n}{n+1}}$ and thickness $\sim \delta^{\frac{n}{n+1}}$.  Therefore, the embedding can expand all lengths by a factor $\sim \delta^{- \frac{1}{n+1}}$ as desired.

So far this discussion works for any $n \ge 2$.  We still need to make sure that our embedding is isotopic to the inclusion $U_W \subset S^m$.  Here we will see a difference between the case $n=2$ and the case $n \ge 3$.  In the case $n=2$, if the map $F_0: S^3 \rightarrow S^2$ has non-zero Hopf invariant, then the tubes of $U_W$ are all linked with each other.  In the embedding from the last paragraph, the tubes are mapped into disjoint balls,
and so their images are unlinked.   Therefore, the embedding is not isotopic to the identity.  (In this case, $k = (m+1)/2$, and the hypotheses of Lemma \ref{emb1} are not satisfied.)  But when $n = 3$, this obstruction disappears because tubes in $S^4$ cannot be linked.  We still need to arrange that on each tube individually, the embedding we construct is isotopic to the inclusion, but this turns out to be reasonably straightforward.  The key point is that 1-dimensional curves in $S^3$ can be linked, but 1-dimensional curves in $S^4$ cannot be.

At this point, we can say more about how general position arguments are relevant to our problem.  Recall that the set $V_W$ is a small neighborhood of a polyhedron $Q$ of dimension $n-k$.  The set $U_W := F_0^{-1}(V_W)$ is a small neighborhood of a polyhedron $P = Y \times Q$ of dimension $(m-n) + (n-k) = m-k$.  Let $p = m-k$ be the dimension of $P$.  The hypothesis of our construction is that $k > (m+1)/2$, which is equivalent to $p < (m-1)/2$.  So our 
set $U_W$ is a small neighborhood of a polyhedron of dimension $p < (m-1)/2$.  The quantitative general position arguments from the last
section show that, because of the dimension condition $p < (m-1)/2$, the set $U_W \subset S^m$ may be isotoped rather freely.  

For our particular application, we want to perform a certain type of isotopy.  Recall that $F_0: U_W \rightarrow V_W$ is a submersion with fibers diffeomorphic to $Y$.  We want to isotope $U_W$ so that the cross-section grows while allowing the fibers to shrink.  In the example above, $U_W$
is a union of tubes of length $\sim 1$ and radius $\sim \delta$.  The fiber is a circle of length $\sim 1$ and the cross-section is an $(m-1)$-ball of radius $\sim \delta$.  We described how to isotope these tubes so that the cross-section grows from radius $\sim \delta$ to radius $\sim \delta^{\frac{n-1}{n}}$.  The isotopy also shrinks the length of the fibers from $\sim 1$ to $\sim \delta^{\frac{n-1}{n}}$.  Lemma \ref{emb1} generalizes this construction to all dimensions, as long as $k > (m+1)/2$.

\subsection{The squeezing map} \label{squeezesec}

In this section, we prove Lemma \ref{squeeze}.  First we recall the statement.

\vskip10pt

{\it We will construct a degree 1 map $\Psi$ from $S^n$ to $S^n$ with the following properties.
On the set $V_W \subset S^n$, the 1-dilation of $\Psi$ is $\lesssim 1$.  
On the complement of $V_W$, the $k$-dilation of $\Psi$ is identically zero.  This happens because $\Psi$
maps the complement of $V_W$ into a $(k-1)$-dimensional subset of $S^n$.}

\vskip10pt

\begin{proof}  We begin by constructing a squeezing map in the setting of the cubical lattice in $\mathbb{R}^n$.  Let $\Sigma$ be
the unit cubical lattice in $\mathbb{R}^n$.  Let $\bar \Sigma$ denote the dual polyhedral structure to
$\Sigma$.  So there is a vertex of $\bar \Sigma$ in the center of each $n$-cube of $\Sigma$, and there is an edge of $\bar \Sigma$ perpendicular to each $(n-1)$-face of $\Sigma$, etc.  So $\bar \Sigma$ is also a unit cubical lattice, shifted by $(1/2, ..., 1/2)$ relative to $\Sigma$.  We let $\Sigma^d$ denote the d-skeleton of $\Sigma$ and $\bar \Sigma^d$ denote the d-skeleton of $\bar \Sigma$.  The key point of our construction is that the complement of $\Sigma^{n-k}$ retracts to $\bar \Sigma^{k-1}$.  The next lemma gives a more quantitative version of this fact.

\begin{lemma} \label{periodsqueeze} For any $W > 0$, there is a $\mathbb{Z}^n$-periodic map $R$ from $\mathbb{R}^n$ to $\mathbb{R}^n$ with the following properties:

\begin{itemize}

\item $R$ maps the complement of the $W$-neighborhood of $\Sigma^{n-k}$ to $\bar \Sigma^{k-1}$,

\item For any point $y \in \mathbb{R}^n$, $|R(y) - y | \le C(W)$,

\item $\Dil_1 R \le C(W)$.

\end{itemize}

\end{lemma}

\begin{proof} We will work on the torus $T^n = \mathbb{R}^n / \mathbb{Z}^n$.  Because $\Sigma$ and $\bar \Sigma$ are periodic, they descend to polyhedral structures on the torus, which we call $\Sigma_{per}$ and $\bar \Sigma_{per}$.  Now $T^n \setminus \Sigma_{per}^{n-k}$ deformation retracts to $\bar \Sigma_{per}^{k-1}$.  By the homotopy extension property, we can homotope the identity map to a smooth map $R_{per}$ which retracts the complement of the $W$-neighborhood of $\Sigma_{per}^{n-k}$ to $\bar \Sigma_{per}^{k-1}$.  Then we lift $R_{per}$ to a $\mathbb{Z}^n$-periodic map $R$ from $\mathbb{R}^n$ to itself.

The first property follows because $R_{per}$ maps $T^n \setminus N_W(\Sigma_{per}^{n-k})$ to $\bar \Sigma_{per}^{k-1}$.  The second property follows because $R_{per}$ is homotopic to the identity.  The last property follows because $R_{per}$ is a smooth map on a compact manifold.  \end{proof}

The map $R$ ``squeezes'' $\mathbb{R}^n \setminus N_W(\Sigma^{n-k})$ into $\bar \Sigma^{k-1}$, and it expands $N_W(\Sigma^{n-k})$ to fill in the space.  We are going to do the same thing at a small scale $\delta$ on the sphere $S^n$.  First we switch scales.

We let $\Sigma_\delta$ be the cubical lattice of side length $\delta$ in $\mathbb{R}^n$, we let $\bar \Sigma_\delta$ be the dual structure, and so on.  By just rescaling the lemma above we get the following.

\begin{lemma} \label{periodsqueezedelta} Let $W > 0$ be fixed independent of $\delta$.  Then for each $\delta > 0$, there is a $\delta \mathbb{Z}^n$-periodic map $R_\delta$ from $\mathbb{R}^n$ to $\mathbb{R}^n$ with the following properties:

\begin{itemize}

\item $R_\delta$ maps the complement of the $W \delta$-neighborhood of $\Sigma^{n-k}_\delta$ to $\bar \Sigma^{k-1}_\delta$,

\item For any point $y \in \mathbb{R}^n$, $|R_\delta(y) - y | \lesssim \delta $,

\item $\Dil_1 R_\delta \lesssim 1 $.

\end{itemize}

\end{lemma}

Finally, we adapt this squeezing map from $\mathbb{R}^n$ to the unit sphere $S^n$.  Recall that $B_r(y_0) \subset S^n$ is close to Euclidean because we can assume that $r \le 1/10$.  We use the exponential map to identify $B_r(y_0)$ with $B^n_r \subset \mathbb{R}^n$.  Recall that $Q^{n-k} \subset B^n_r$ is just the union of all faces of $\Sigma^{n-k}$ inside of $B^n_r$.  Recall that $V_W \subset B^n_r$ is just the $W \delta$-neighborhood of $Q$.  Since $B^n_r$ is identified with $B_r(y_0)$, we can also think of $V_W$ as a subset of $B_r(y_0) \subset S^n$.

We are now ready to construct the map $\Psi$.  Let $B' = B_{r/4}(y_0)$, which we identify with $B^n_{r/4} \subset \mathbb{R}^n$.  Let $\Phi: S^n \rightarrow S^n$ be a degree 1 map which collapses $S^n \setminus B'$ to a point $q$.  We define $\Psi$ as follows.

\begin{itemize}

\item If $y \in 3 B'$, then $\Psi(y) = \Phi \circ R_\delta (y)$.

\item If $y \notin 3 B'$, then $\Psi(y) = q$.  

\end{itemize}

First we check that these definitions match up to give a globally defined smooth map.  We should say a bit more about the definition in the first case.  If $\delta > 0$ is small enough, then $R_\delta$ maps $B_{3r/4}^n$ into $B_r^n$; so by a slight abuse of notation we can think of it as a map from $3 B'$ into $B_r(y_0) \subset S^n$.  The definitions match because if $y$ lies in $3 B' \setminus 2 B'$, then $R_\delta(y)$  maps $y$ to the complement of $B'$, and so $\Phi \circ R_\delta(y) = q$.  

We check that $\Psi$ has degree 1.  We let $R_{t,\delta}(y) = (1-t) R_\delta(y) + t y$ be a straight-line homotopy from $R_\delta$ to the identity.  We define $\Psi_t$ by replacing $R_\delta$ with $R_{t,\delta}$ in the definition above.  Each $R_{t, \delta}$ obeys the displacement bound $|R_{t, \delta}(y) - y| \lesssim \delta$, and so the argument above shows that $\Psi_t$ is a continous family of maps.  Therefore, $\Psi$ is homotopic to $\Phi$ and has degree 1.

Next we check the geometric properties of $\Psi$.  Since $R_\delta$ and $\Phi$ each have 1-dilation $\lesssim 1$, we see that $\Psi$ has 1-dilation $\lesssim 1$.  Suppose that $y$ is in the complement of $V_W$.  We have to check that $\Lambda^k d\Psi_y = 0$.  If $y$ does not lie in $2 B'$, then $\Psi$ collapses a neighborhood of $y$ to a point, and so $d\Psi_y = 0$.  Suppose $y$ lies in $2 B'$ but not in $V_W$.  After identifying $B_r(y_0)$ with $B^n_r$, $y$ lies in $B_{r/2}^n$ but not in $N_{W \delta} \Sigma_\delta^{n-k}$.  Therefore, $R_\delta$ maps a neighborhood of $y$ to
the $(k-1)$-dimensional polyhedron $\bar \Sigma_\delta^{k-1}$.  Therefore, $\Lambda^k d R_{\delta,y} = 0$, and so $\Lambda^k d \Psi_y = 0$ as well. \end{proof}

\subsection{Quantitative embedding}

Now we show that Lemma \ref{emb1} follows from the quantitative embedding lemma in the last section.  First we recall the statement
of Lemma \ref{emb1}.

\vskip10pt

{\it If $k > (m+1)/2$, then there is a constant $W \gtrsim 1$, and an embedding $I: (U_W, g_1) \rightarrow (S^m, g_0)$ which is isotopic to the inclusion $U_W \subset S^m$, and which increases all lengths by a factor $\gtrsim \delta^{-a}$.  Here $a = \frac{m-n}{m} > 0$. }

\vskip10pt

Now we give the proof of Lemma \ref{emb1}.

\begin{proof}  In order to apply the quantitative embedding lemma, we need to define, $M$, $P$, $N$, $I_0$, etc.

We let $M$ be $B_r(y_0) \times Y$.  Recall that $h_0$ is the unit sphere metric on $B_r(y_0) \subset S^n$, and that $g_Y$ is the metric on
$Y \subset S^m$ induced from the unit sphere metric of $S^m$.  Recall that $g_1$ is defined to be $h_0 + \delta^2 g_Y$.  The volume of $(M, g_1)$ is $\sim \delta^{m-n}$.  We rescale the metric to have volume $\sim 1$: we let $g = \delta^{-2 \frac{m-n}{n}} g_1 = \delta^{-2a} g_1$.  We have now defined $(M,g)$.

We recall that $Q \subset B_r(y_0)$ is the (n-k)-skeleton of a cubical lattice with spacing $\delta$ (with respect to the metric $h_0$).  We recall that $P = Q \times Y \subset B_r(y_0) \times Y = M$.  After picking a triangulation of $Y$, we can view $P$ as a polyhedron embedded in $M$.  The pair $P \subset (M, g_1)$ has bounded local geometry at scale $\delta$.  (In fact, the local geometry of the $P \subset (M, g_1)$ at scale $\delta$ is essentially independent of $\delta$.  In other words, if we take any $\delta > 0$, and then look at a $\delta$-neighborhood in $(M, g_1)$ and rescale it to size 1, the result will be essentially independent of $\delta$.)  After rescaling, we see that $P \subset (M, g)$ has
bounded geometry at scale $s = \delta^{-a} \delta = \delta^{1-a}$.  

Because $k > (m+1)/2$, we recall that $p = \Dim P < (m-1)/2$.  (To check this, note that $Q$ was has dimension $n-k$, $Y$ has dimension $m-n$, and so $p = (n-k) + (m-n) = m - k$.)

The target $N$ is just $(S^m, g_0)$, which has bounded geometry at scale 1.

Recall that $F_0 \times \pi_Y: F_0^{-1}(B_r(y_0)) \rightarrow B_r(y_0) \times Y$ is a diffeomorphism.  The inverse of this diffeomorphism is our embedding $I' : M \rightarrow N$.  (We have $M = B_r(y_0) \times Y \rightarrow F_0^{-1}(B_r(y_0)) \subset S^m = N$.)

It remains to construct a 1-Lipschitz map $I_0: (M, g) \rightarrow (S^m, g_0)$ which maps at most $\mu \lesssim 1$ vertices
of $P$ into any ball of radius $s$ in $(S^m, g_0)$.  Geometrically, the $Y$-factor in $(M,g)$ is very small, and so $(M,g)$ looks 
essentially like an $n$-dimensional disk of radius $\delta^{-a} = \delta^{-1} s$.  As our mapping, we fold up this disk inside of $(S^m, g_0) = N$.  

Here are the details.  We begin with the projection $\pi_B: M = B_r(y_0) \times Y \rightarrow B_r(y_0)$.  Recall that the metric $g$ is
a product metric $g = \delta^{-2a} h_0 + \delta^{2-2a} g_Y$.  So $\pi_B: (M,g) \rightarrow (B_r(y_0), \delta^{-2a} h_0)$ has Lipschitz constant 1.
Now $(B_r(y_0), \delta^{-2a} h_0)$ is bilipschitz equivalent to a Euclidean ball of radius $\delta^{-a} = \delta^{-1} s$ (and the bilipschitz constant is $\lesssim 1$).  So up to a controlled bilipschitz error, we can identify $(B_r(y_0), \delta^{-2a} h_0)$ with the Euclidean ball $B^n(\delta^{-1}s)$.  

Next we want to fold up this $n$-ball inside of $(S^m, g_0)$.  To fold it up in a useful way, we consider the product $B^n (\delta^{-1}s) \times [-s, s]^{m-n}$, which is an $m$-dimensional convex set with volume $\sim 1$.  It admits an embedding $i_0$ into a hemisphere of $(S^m, g_0)$ with bilipschitz constant
$\lesssim 1$.  (For the details of this embedding, see the appendix in Section \ref{bilipembed}.)  By a slight rescaling, we can
arrange that $i_0: (B_r(y_0), \delta^{-2a} h_0) \rightarrow (S^m, g_0)$ has Lipschitz constant $\le 1$.  Now we define $I_0 = i_0 \circ \pi_B$.  The map $I_0: M \rightarrow N$ has Lipschitz constant $\le 1$.

We have to check that $I_0$ maps $\mu \lesssim 1$ vertices of $P$ into each ball of radius $s$ in $N = (S^m, g_0)$.
If $B(s)$ is a ball of radius $s$ in $(S^m, g_0)$, then $i_0^{-1}(B(s))$ is contained in  $\lesssim 1$
balls of radius $\lesssim s$ in $(B_r(y_0), \delta^{-2a} h_0)$.  
Since $(Y, \delta^{2-2a} g_Y)$ has diameter $\lesssim s$, the preimage $\pi_B^{-1}$ of a ball of radius $\lesssim s$ in $(B_r(y_0), \delta^{-2a} h_0)$ is contained in a ball of radius $\lesssim s$ in $(M, g)$.  Therefore, $I_0^{-1}( B(s) )$ is contained in $\lesssim 1$ balls of radius $\lesssim s$ in $(M, g)$.  Since $P \subset (M, g)$ has uniformly bounded local geometry at scale $s$, we see that $I_0^{-1}(B(s))$ contains $\mu \lesssim 1$ vertices of $P$.

Finally, we have to define $T_0$.  Without loss of generality, we can assume that $F_0^{-1}(B_r(y_0)) = I'(M)$ is contained in a hemisphere of $(S^m, g_0)$, and that $I_0: M \rightarrow (S^m, g_0)$ is contained in the same hemisphere.  We let $N_{hemi} \subset N$ be this hemisphere.  The tangent
bundle of the hemisphere $N_{hemi}$ is trivial.  We pick a particular trivialization by doing parallel transport on the geodesics to the pole, and we get
a trivialization map $\Triv^N: T N_{hemi} \rightarrow \mathbb{R}^m$, which is an isometry on each tangent space.  Also $| \nabla \Triv^N | \lesssim 1$.  
We would like to also find a trivialization of $TM$.  One option is to take $\Triv^N \circ dI': TM \rightarrow \mathbb{R}^m$.  This is a trivialization, but
the map on each tangent space is far from an isometry, because the map $I': (M,g) \rightarrow (N,h)$ is far from an isometry.  Recall that $TM = TY \oplus TB_r(y_0)$.  Let $S: TM \rightarrow TM$ be the map that multiplies each vector in the $Y$-direction by $\delta^{1-a}$ and each vector in the $B_r(y_0)$-direction by $\delta^{-a}$.  Then $dI' \circ S: TM \rightarrow TN$ has bilipschitz constant $\lesssim 1$ (at each point of $M$).  We define $\Triv^M = \Triv^N \circ dI' \circ S$.  It has bilipschitz constant $\lesssim 1$ on each tangent space.  With these two trivializations, we can define a fiberwise isomorphism $T_0: TM \rightarrow TN$ covering $I_0$ as the unique map that commutes with the two trivializations.  In other words, if $x \in M$ and $v \in T_xM$, then 

$$ T_0 (x,v) = \left( I_0(x), (\Triv^N_{I_0(x)})^{-1} \left( \Triv^M_x v \right) \right). $$

Because both trivializations are bilipschitz, the map $T_0$ has fiberwise bilipschitz constant $\lesssim 1$. 

We have to show that our $T_0$ is homotopic to $dI'$ in the category of fiberwise isomorphisms.  Because $I_0$ and $I'$ both map $M$ to a contractible
hemisphere, they are homotopic.  So we may homotope $T_0$ to

$$ \left( I'(x), (\Triv^N_{I'(x)})^{-1} \left( \Triv^M_x v \right) \right) = $$

$$ \left( I'(x), dI' \circ S (x,v) \right). $$

To finish, we may homotope $S$ to the identity, and thus homotope this last map to $dI'$.  

Finally, we have to check that $T_0$ does not oscillate too rapidly -- in particular that $s | \nabla T_0 | \lesssim 1$.

Roughly speaking, if $dist_g(x_1, x_2) = d$, then the distance in $(S^m, g_0)$ from $I'(x_1)$ to $I'(x_2)$ is $\lesssim s^{-1} d$, and so the change between
$T_0(x_1)$ and $T_0(x_2)$ is $\lesssim s^{-1} d$ also.  Therefore, $s | \nabla T_0 | \lesssim 1$.  

Here are more details.  We have $T_0$ a bundle map from $(TM, g)$ to $(TN, h = g_0)$.  Each bundle is equipped with a metric and a connection, and so
we can define $\nabla T_0$ and $| \nabla T_0 |$.  

We have $T_0(x,v) =  \left( I_0(x), (\Triv^N_{I_0(x)})^{-1} \left( \Triv^M_x v \right) \right). $   The first component $I_0(x)$ is harmless, because $I_0: 
(M,g) \rightarrow (N,h)$ has $| \nabla I_0 | \lesssim 1$.  We concentrate on the second component.  The trivialization $\Triv^N$ is also harmless since
we have $| \nabla \Triv^N | \lesssim 1$.  So it suffices to check that $| \nabla \Triv^M | \lesssim s^{-1}$.  We recall that

$$ \Triv^M := \Triv^N \circ dI' \circ S . $$

To analyze this composition, we have to think about the domain and range of each map in the right way, as follows:

$$ S: (TM, g) \rightarrow (TM, g_Y \oplus h_0),$$

$$ dI': (TM, g_Y \oplus h_0) \rightarrow (TN, h), $$

$$ \Triv^N: (TN, h) \rightarrow \mathbb{R}^m.$$

Each of the spaces in the list above is a bundle equipped with a metric and a corresponding connection.  (If we like, $\mathbb{R}^m$ is a trivial bundle 
over a point.)  Using the relevant metrics, we can define the operator norm of $S, dI',$ and $\Triv^N$, and they are all $\lesssim 1$.  
Using the relevant metrics and connections, we can define $\nabla S$, $\nabla dI'$, and $ \nabla \Triv^N$.  They are each quite nice.
The first of them, $\nabla S$, is identically zero.  This is because the splitting $TM = TY \oplus TB_r(y_0)$ is parallel with respect to both $g_Y \oplus
h_0$ and $g = \delta^{2-2a} g_Y \oplus \delta^{-2a} h_0$.  The other two obey $| \nabla dI' | \lesssim 1$ and $| \nabla \Triv^N | \lesssim 1$, since after
all neither of these tensors depends on $\delta$.  Just to be clear, when we say $| \nabla dI' | \lesssim 1$, this means that if $v$ is a tangent
vector in $TM$, then $| \nabla_v dI' | \lesssim |v|_{g_Y \oplus h_0}$.  Finally we expand $\nabla_v \Triv^M$ using the Leibniz rule:

$$ \nabla_v \Triv^M = \nabla_{dI'(v)} \Triv^N \circ dI' \circ S + \Triv^N \circ \nabla_v dI' \circ S + \Triv^N \circ dI' \circ \nabla_v S. $$

The last term is just 0.  Since the operators are all bounded and $| \nabla \Triv^N |$ and $|\nabla dI'|$ are bounded, we see that the norm
of this expression is bounded by

$$ |dI'(v)|_{h} + |v|_{g_Y \oplus h_0} \lesssim s^{-1} |v|_{g}.$$

In summary, we have shown that $ | \nabla_v \Triv^M | \lesssim s^{-1} |v|_g$ which is equivalent to $|\nabla \Triv^M | \lesssim s^{-1}$.  

We have now checked all the hypotheses of the quantitative embedding lemma, and we may apply it to finish the proof of Lemma \ref{emb1}.  
The quantitative embedding lemma 
gives us an $L$-bilipschitz embedding $I: (U_W, g) \rightarrow (S^m, g_0)$ isotopic to $I'$ with $W \gtrsim 1$, $L \lesssim 1$.  
In the context of the quantitative embedding lemma, $U_W$ is defined to be the $Ws$-neighborhood of $P \subset (M,g)$.  
But this definition agrees with the previous definition of $U_W$ as $V_W \times Y$.  The embedding $I': U_W \rightarrow S^m$ is just
the inclusion map, and so $I$ is isotopic to the inclusion.  Finally, if we use the metric $g_1 = \delta^{2a} g$ on $U_W$, then
we see that the embedding $I$ from $(U_W, g_1)$ into $(S^m, g_0)$ expands all lengths by a factor $\gtrsim \delta^{-a}$ as desired. \end{proof}

This finishes the proof of Lemma \ref{emb1}, and hence the proof of the h-principle for k-dilation.

\section{Some previous lower bounds for $k$-dilation} \label{prevlower}

For context, we recall here several approaches to get lower bounds on $k$-dilation of maps from the unit $m$-sphere to the
unit $n$-sphere.  There are very few known techniques for this problem.

The most basic estimates for $k$-dilation have to do with $k$-dimensional homology.
Suppose $F: (M, g) \rightarrow (N,h)$.  If $\Sigma^k$ is a $k$-dimensional surface
in $M$, then by definition $\Vol_k (F (\Sigma)) \le \Dil_k(F) \Vol_k(\Sigma)$.  This estimate
passes to homology.  For example, we can assign a volume to a class $h \in H_k(M; \mathbb{Z})$
as the smallest $k$-dimensional volume (or mass) of a cycle $z$ in the class $h$.  Then
$\Vol_k( F_* (h) ) \le \Dil_k(F) \Vol_k (h)$.  This argument implies that a degree $D$ map from the unit
$n$-sphere to itself has $n$-dilation at least $|D|$, which is sharp.

If $m > n$, then maps from $S^m$ to $S^n$ are homologically trivial, and this simple method doesn't give
any information.  The next homotopy classes that were studied were classes with non-zero Hopf invariant.  These
homotopy classes are well-understood by the following theorem.

\begin{ho} (\cite{GMS}, pages 358-359) Let F be a map from $S^{4n-1}$ to
$S^{2n}$.  Then the norm of the Hopf
invariant of $F$ is bounded by $C(n) \Dil_{2n}(F)^2$.  Since the Hopf
invariant is an integer, any map with non-zero Hopf invariant has
$2n$-dilation at least $C(n)^{-1/2}$.
\end{ho}

\proof Let $\omega$ be a $2n$-form on $S^{2n}$ with $\int \omega =
1$.  The pullback $F^*(\omega)$ is a closed $2n$-form on
$S^{4n-1}$.  Since $H^{2n}(S^{4n-1}) = 0$, this form is exact. 
We let $P F^*(\omega)$ denote any primitive of $F^*(\omega)$. 
Then the Hopf invariant of F is equal to $\int_{S^{4n-1}} P
F^*(\omega) \wedge F^*(\omega)$.

We take $\omega$ to be a multiple of the volume form, so $|\omega|
< C$ at every point of $S^{2n}$.  The norm of $F^*(\omega)$ is
bounded by $C \Dil_{2n}(F)$ pointwise.  Therefore, the $L^2$ norm of
$F^*(\omega)$ is bounded by $C \Dil_{2n}(F)$.  

Using Hodge theory, we can
choose $P F^*(\omega)$ to be perpendicular to all of the exact
$(2n-1)$-forms.  For this choice, the $L^2$ norm of $P F^*(\omega)$
is bounded by $\lambda^{-1/2} \| F^* \omega \|_2$, where $\lambda$ is the
smallest eigenvalue of the Laplacian on exact (2n)-forms.  The
eigenvalue $\lambda$ is greater than zero and depends only on n. 
Finally, the norm of the Hopf invariant is bounded by
$|F^*(\omega)|_{L^2} |P F^*(\omega)|_{L^2}$, which is bounded by
$C(n) \Dil_{2n}(F)^2$. \endproof

The 2-dilation for $k=2$ has stronger properties than for $k > 2$, and
there are several special techniques for dealing with it.
The most important result in this direction is the theorem of
Tsui and Wang mentioned in the introduction.

\begin{rtwo} (Tsui and Wang, \cite{TW}) Let $F$ be a $C^1$
map from $S^m$ to $S^n$, where $m \ge 2$.  If the 2-dilation of $F$
is less than 1, then $F$ is nullhomotopic.
\end{rtwo}

The proof by Tsui and Wang uses the mean curvature flow to deform
the graph of the map $F$ as a submanifold of $S^m \times S^n$. 
They prove that the mean curvature flow converges to the graph of
a constant function and that at each time $t$ the flowed
submanifold is the graph of a map $F_t$.  Therefore, $F_t$
provides a homotopy from $F$ to a constant map.

In \cite{GCC} (page 179), Gromov proved a slightly weaker theorem in the same
spirit.  For each $m$ and $n$, there exists a number
$\epsilon(m,n) > 0$, so that any $C^1$ map from $S^m$ to $S^n$
with 2-dilation less than $\epsilon(m,n)$ is null-homotopic.  The proof is based on the
uniformization theorem and the borderline Sobolev inequality.
Here is a sketch of the proof.  We view the map from $S^m$ to $S^n$ as a family of
maps from $S^2$ to $S^n$, parametrized by $B^{m-2}$, where the maps at the 
boundary of $B^{m-2}$ are constant maps.  Let's call the family $F_a: S^2 \rightarrow
S^n$, where $a \in B^{m-2}$.  If the 2-dilation of the original map is less than
$\epsilon$, then each image $F_a(S^2)$ has area less than $4 \pi
\epsilon$.  We change coordinates on each copy of $S^2$ using the uniformization theorem,
so each map $F_a$ becomes conformal.  With care, this can be done in a way that
is continuous in $a$.  After the change of coordinates, we get a (homotopic) map $\tilde F_a:
S^2 \rightarrow S^n$, where each map $\tilde F_a$ has Dirichlet energy less than $C \epsilon$.
By the borderline Sobolev inequality, each map $\tilde F_a$ has BMO norm $\lesssim \epsilon$.
To define the BMO norm, we think of $\tilde F_a$ as a map from $S^2$ to $\mathbb{R}^{n+1}$.
Let $B$ be any ball in $S^2$.  Let $M_a(B)$ be the mean value of $\tilde F_a$ on $B$.  The \
bordeline Sobolev inequality says that the mean value of $|\tilde F_a - M_a(B)|$ on the ball $B$
is $\lesssim \epsilon$.  Now $\tilde F_a$ maps $S^2$ into $S^n \subset \mathbb{R}^{n+1}$.  There is
no reason that $M_a(B)$ must lie in $S^n$.  But the last inequality implies that $M_a(B)$ is rather
close to $\tilde F_a(x)$ for many points $x \in B$.  In particular, it implies that $M_a(B)$ is $\lesssim \epsilon$
from $S^n$.  Now we can homotope $\tilde F_a$ to a constant map by averaging over balls of radius $r$ and
sending $r$ from 0 to $\pi/2$.  This homotopy depends continuously on $a$.  The homotopy does not lie
in $S^n$, but the BMO inequality tells us that it lies in the $C \epsilon$-neighborhood of $S^n$.  As long
as $C \epsilon < 1/2$, we can modify the homotopy to lie entirely in $S^n$.  In summary, we get a homotopy
from our original map to a new map that factors through $B^{m-2}$ and so is contractible.

The following result of Joe Coffey (unpublished) is also relevant to 2-dilation.

\begin{prop} The space of non-surjective pointed maps from $S^2$ to $S^2$ has vanishing homotopy groups.
\end{prop}

Sketch.  Let $\Map_{NS}(S^2, S^2)$ be the space of non-surjective maps from $S^2$ to $S^2$, taking the base point
of the domain to the basepoint of the range.
Let $X \subset S^2$ be a finite subset, and let $\Map_X \subset \Map_{NS}(S^2, S^2)$ be the set of maps from
$S^2$ to $S^2$ which do not contain $X$ in their image.  The first key point is that $\Map_X$ is
contractible.  This happens because the universal cover of $S^2 \setminus X$ is contractible.
The space $\Map_{NS}(S^2, S^2)$ is the union $\cup_{p \in S^2} \Map_p$.  Each of the sets $\Map_p$ is contractible.  
Any finite interesction $\cap_{i=1}^I \Map_{p_i}$ is the space $\Map_X$ for $X = \cup p_i$, and so it is contractible.
Now by a standard argument with nerves, any finite union $\cup_{i=1}^I \Map_{p_i}$ is contractible.  The sets $\Map_{p}$
are also open.  Therefore, if $f: S^p \rightarrow \Map_{NS}(S^2, S^2)$ is a continuous map, the image $f(S^p)$ lies in a
 finite union of sets $\Map_p$.  Therefore, $f$ is contractible.  By this argument, 
 all the homotopy groups of $\Map_{NS}(S^2, S^2)$ vanish.  This finishes the sketch of the proof.

As a corollary, we see that every map from the unit $m$-sphere to the unit 2-sphere
with 2-dilation $< 1$ is contractible, recovering a special case of the theorem of Tsui and Wang.

These three techniques give strong results about 2-dilation, but they haven't yet led to any results about $k$-dilation for
$k \ge 3$.  Can the mean curvature flow shed any light on $k$-dilation for $k \ge 3$?  The Riemann mapping theorem seems
inherently two-dimensional.  Coffey's proof uses the fact that the complement of a finite (non-empty) set in $S^2$ is aspherical.  This fact is special to two dimensions.  But studying the space of non-surjective maps may yield some results in any dimension.
Let $\Map_{NS}(S^n, S^n)$ denote the space of non-surjective basepoint-preserving 
maps from $S^n$ to $S^n$.  Clearly $\Map_{NS}(S^n, S^n) \subset \Map(S^n, S^n)$.  This inclusion induces a map of homotopy groups

$$ \pi_q(\Map_{NS}(S^n, S^n)) \rightarrow \pi_q (\Map(S^n, S^n)) = \pi_{n+q}(S^n). $$

If $a \in \pi_{n+q}(S^n)$ can be realized by maps with $n$-dilation $ < 1$, then $a$ will lie in the image of 
$\pi_q(\Map_{NS}(S^n, S^n)$.  
Other than Coffey's theorem, I don't know of any example where this image has been calculated.  For example, it would be interesting to know whether the image of $\pi_3( \Map_{NS}(S^4, S^4))$ in $\pi_7(S^4)$ contains classes with non-zero Hopf invariant.

There is another basic fact about $k$-dilation which is relevant to our discussion.
This result has to do with the $C^0$ limits of maps with bounded $k$-dilation.  It appears as an exercise in \cite{GPDR} (page 23, exercise B3).  

\begin{prop} \label{c^0closed} Suppose that $F_i:  B^m \rightarrow \mathbb{R}^n$ is a sequence of $C^1$ maps
from the unit $m$-ball to $\mathbb{R}^n$, where each map has $k$-dilation $\le 1$.  Suppose that $F_i$ converges in $C^0$ 
to a limit $F$.  If $F$ is $C^1$ then $\Dil_k(F) \le 1$.  In fact, if $F$ is just differentiable at one point $x$, then $\Dil_k ( dF_x) \le 1$.
\end{prop}

We begin with the following special case.

\begin{lemma} Let $F_i: \bar B^k \rightarrow \mathbb{R}^k$ be a sequence of maps from the closed $k$-ball to
$\mathbb{R}^k$ that converges in $C^0$ to a linear map $L$.  If $\Dil_k(F_i) \le 1$, then $\Dil_k(L) \le 1$.
\end{lemma}

\begin{proof} 
The key point is the following.  If $| F_i - L |_{C^0} \le \epsilon$, then $F_i (B^k)$ must cover $L(B^k)$ except for the
$\epsilon$-neighborhood of $L(S^{k-1})$.  This fact is well-known but not completely elementary.  It follows from Brouwer's theory of degrees and winding numbers.  If $y$ is contained in $L(B^k)$, then the winding number of $L: S^{k-1} \rightarrow \mathbb{R}^k \setminus \{ y \}$ is equal to 1.  
We make a straight-line homotopy from $L$ to $F_i$.  If $y$ is not in the $\epsilon$-neighborhood of $L(S^{k-1})$, then this homotopy never hits $y$, and so $F_i: S^{k-1} \rightarrow \mathbb{R}^k \setminus \{ y \}$ has winding number 1.  Therefore, $F_i(B^k)$ must cover $y$.

Taking $\epsilon \rightarrow 0$, we see that $\Vol_k L(B^k) \le \limsup_i \Vol_k F_i(B^k) \le \Vol_k (B^k)$, and so $\Dil_k(L) \le 1$.  \end{proof}

By scaling the domain and the range, the unit $k$-ball may easily be replaced by a ball of any radius $r > 0$.

Now we turn to the general proposition, which follows from the lemma and some rescaling.

\begin{proof}  We do a proof by contradiction.  Suppose that $x_0$ is a point where $F$ is differentiable,
and $d F_{x_0}$ has $k$-dilation $> 1$.  By translating, we can assume without loss of generality that $x_0 = 0$
and $F(x_0) = 0$.
Next, let $P^k \subset \mathbb{R}^m$ be a $k$-plane through $0$, so that the $k$-dilation of
$dF_0$ restricted to $P$ is $ > 1$.  Let $Q^k = dF_0 (P^k)$, a $k$-plane in $\mathbb{R}^n$.
By taking a subsequence of the $F_i$, we can arrange that $|F_i(x) - F(x)| < 3^{-i}$ for all $x$ in the unit ball and all $i \ge 1$.
 Now we consider a sequence of maps $G_i$ from $P$ to $Q$ made by scaling down, applying $F_i$, projecting onto $Q$, and scaling
back up.  If $\pi_Q$ denotes the orthogonal projection from $\mathbb{R}^n$ to $Q$, we can write $G_i$ by the formula

$$ G_i(x) = 2^i \pi_Q \left(   F_i( x /2^i) )\right). $$

\noindent The maps $G_i$ all have $k$-dilation $\le 1$.   

Next, let $H_i(x)$ be defined by using $F$ in place of $F_i$ in the definition of $G_i$:

$$ H_i(x) = 2^i \pi_Q \left(   F( x /2^i) )\right). $$

\noindent Since $F$ is differentiable at 0, the maps $H_i$ converge locally in $C^0$ to the linear map $dF_0$. 

If $|x| \le 1$, then $| G_i(x) - H_i(x) | \le 2^{i} | F_i(x 2^{-i}) - F(x 2^{-i})| \le (3/2)^{-i}$.  Therefore, the maps
$G_i$ converge to $dF_0$ in $C^0$ on the unit ball.  Since $\Dil_k(G_i) \le 1$, the lemma implies that
 $dF_0$ has $k$-dilation $\le 1$.   \end{proof}

As far as I know, all the proofs of Proposition \ref{c^0closed} need degree theory or a similar contribution from topology.
For perspective on the role of topology, consider the following counterexample.  
Suppose that $F_i: \mathbb{R}^2 \rightarrow \mathbb{R}^2$ is a sequence of $C^1$ maps (or even smooth maps). 
Suppose that at each point, each derivative $dF_i$ has singular values $s_1 \le s_2$ obeying the inequality $s_1^{1/2} s_2 \le 1$.
Roughly speaking, this means that $F_i$ is allowed to stretch space in one direction by a factor $\Lambda > 1$ as long as it shrinks
in the perpendicular direction by a factor $\Lambda^2$.  It turns out that $C^0$ limits of the $F_i$ do not have to obey the same condition on singular values.  For example, let $B > 1$ be any number, and let $L$ be the linear map $L(x_1, x_2) = (x_1/ B, B x_2)$.  The map $L$ has singular values $s_1 = B^{-1}$ and $s_2 = B$.  Therefore, $s_1(L)^{1/2} s_2(L) = B^{1/2}$, which can be arbitrarily large.  Nevertheless, for any $B > 1$, there exists a sequence of maps $F_i: \mathbb{R}^2 \rightarrow \mathbb{R}^2$ which obey the condition $s_1^{1/2} s_2 \le 1$ pointwise and converge to $L$ in $C^0$.  This construction is described in Appendix A of \cite{GUT}.

\section{Open problems} \label{openproblems}

The main question we considered in the paper was the following. Fix a homotopy class $a \in \pi_m(S^n)$ and an integer $k$.  Can the class $a$ be realized by a sequence of maps $F_j: S^m \rightarrow S^n$ with $\Dil_k(F_j) \rightarrow 0$?  This question was previously understood for maps of non-zero degree or non-zero Hopf invariant.  Our main theorem answers this question for the non-zero homotopy class in $\pi_m(S^{m-1})$ when $m \ge 4$.  There are a few other cases where the answer is known, especially in low dimensions, but the question is open for most homotopy classes.

To give some perspective, we record here what we know about some low-dimensional homotopy groups of spheres.  We use the lists of homotopy groups of spheres and the suspension maps between them
given in \cite{T} on pages 39-42.  By the 
theorem of Tsui and Wang, no homotopy class can be realized with arbitrarily small 2-dilation.  For maps $a \in \pi_m(S^2)$, this
theorem completely answers our question.  In the next few paragraphs, we consider a few homotopy groups of $S^3$, $S^4$, and $S^5$.

We start with the homotopy groups of $S^3$.  No homotopy class can be realized with arbitrarily small 2-dilation.  
The group $\pi_4(S^3)$ is isomorphic to $\mathbb{Z}_2$, and the non-trivial element is the suspension of the Hopf fibration.  By the suspension construction, it can be realized with arbitrarily small 3-dilation.  The group $\pi_5(S^3)$ is also isomorphic to $\mathbb{Z}_2$, and the non-trivial element is the suspension of an element from $\pi_4(S^2)$.  By the suspension
construction, it can be realized with arbitrarily small 3-dilation.  The group $\pi_6(S^3)$ is isomorphic to $\mathbb{Z}_{12}$.  One non-trivial element of $\pi_6(S^3)$ is a suspension of an element from $\pi_5(S^2)$.  This one element can be realized by maps with arbitrarily small 3-dilation.  For the other (non-zero) elements of $\pi_6(S^3)$, it is an open question whether they can be realized with arbitrarily small 3-dilation.

We next consider the homotopy groups of $S^4$.  The group $\pi_5(S^4)$ is isomorphic to $\mathbb{Z}_2$.  Our main theorem says that the non-trivial class can be realized by maps with arbitrarily small 4-dilation but not arbitrarily small 3-dilation.  The group $\pi_6(S^4)$ is also isomorphic to $\mathbb{Z}_2$.  The non-trivial element is the double suspension of a class from $\pi_4(S^2)$.  By the suspension construction, it can be realized by maps with arbitrarily small 4-dilation.  I don't know whether it can be realized with arbitrarily small 3-dilation.  The group $\pi_7(S^4)$ is isomorphic to $\mathbb{Z} \oplus \mathbb{Z}_{12}$.  The elements with non-trivial Hopf invariant cannot be realized with arbitrarily small 4-dilation.  The other elements are suspensions of elements in $\pi_6(S^3)$.  By the suspension construction, they can be realized with arbitrarily small 4-dilation.  One element is a double suspension of an element in $\pi_5(S^2)$.  It can be realized with arbitrarily small 3-dilation.  I don't know whether the other torsion elements can be realized by maps with arbitrarily small 3-dilation.

Finally, we consider some homotopy groups of $S^5$.  The group $\pi_6(S^5)$ is isomorphic to $\mathbb{Z}_2$.  Our main theorem says that the non-trivial class can be realized by maps with arbitrarily small 4-dilation but not arbitrarily small 3-dilation.  The group
$\pi_7(S^5)$ is isomorphic to $\mathbb{Z}_2$, and the non-trivial element is the triple suspension of an element in $\pi_4(S^2)$.  By the suspension construction it can be realized by maps with arbitrarily small 4-dilation.  I don't know whether it can be realized by maps with arbitrarily small 3-dilation.  The group $\pi_8(S^5)$ is isomorphic to $\mathbb{Z}_{24}$.  By the h-principle, every class can be realized by maps with arbitrarily small 5-dilation.  The classes with non-zero Steenrod-Hopf invariant cannot be realized with arbitrarily small 4-dilation.  These classes correspond to the odd numbers in $\mathbb{Z}_{24}$.  The other classes are all double suspensions of classes in $\pi_6(S^3)$, and one class (the class corresponding to the number 12) is the triple suspension of a class in $\pi_5(S^2)$.  The suspension construction implies that the triple suspension can be realized by maps with arbitrarily small 4-dilation.  I don't know whether any of the double suspensions can be realized with arbitrarily small 4-dilation, or whether any non-trivial map can be realized with arbitrarily small 3-dilation.

We also briefly consider the homotopy groups of small codimension.  For $m \ge 4$, $\pi_m(S^{m-1}) = \mathbb{Z}_2$.  Our main theorem says that the non-trivial class can be realized by maps with arbitrarily small $k$-dilation if and only if $k > (m+1)/2$.  Next we consider the group $\pi_m(S^{m-2})$.   The homotopy group $\pi_m(S^{m-2})$ is equal to $\mathbb{Z}_2$ for all $m \ge 4$, and the suspension is an isomorphism.  The suspension construction implies that the non-trivial class can be realized by maps with arbitrarily small $k$-dilation for all $k > m/2$.  (This improves slightly on the h-principle, which imples that the non-trivial class can be realized by maps with arbitrarily small $k$-dilation for all $k > (m+1)/2$.)  By the Tsui-Wang theorem, we know that none of these classes can be realized with arbitrarily small 2-dilation.  But we don't know a lower bound on the 3-dilation for any of these classes.  The group $\pi_m(S^{m-3})$ is isomorphic to $\mathbb{Z}_{24}$ for all $m \ge 8$.  By the h-principle, any of these elements can be realized with arbitrarily small $k$-dilation for $k > (m+1)/2$.  Half of the elements have non-zero Steenrod-Hopf invariant.  These elements can be realized with arbitrarily small $k$-dilation only if $k > (m+1)/2$, so we understand them well.  The elements with zero Steenrod-Hopf invariant are all suspensions from $\pi_6(S^3)$.  By the suspension construction, they can all be represented by maps with arbitrarily small $k$-dilation for $k > m/2$.  One of the elements is a suspension from $\pi_5(S^2)$.  By the suspension construction, it can be represented by maps with arbitrarily small $k$-dilation for $k > (2/5)m$.

We can use the $k$-dilation to define a filtration on the homotopy groups of spheres.  We say that $a \in V_k \pi_m (S^n)$ if the class $a$ can be realized by maps with arbitrarily small $k$-dilation.  It is relatively easy to check that
$V_k \pi_m(S^n)$ is a subgroup of $\pi_m (S^n)$, and that $0 = V_1 \pi_m(S^n) \subset  V_2 \pi_m(S^n) \subset ...
\subset V_n \pi_m (S^n) \subset \pi_m(S^n)$.  We will give the proof below.

The definition of $V_k \pi_m(S^n)$ does not depend on the choice of a metric on $S^m$ or $S^n$.  We have been working with the unit sphere metrics in this paper.  Suppose that we choose other metrics $g$ on $S^m$ and $h$ on $S^n$.  Let $\Dil_k^{(g,h)}(F)$ be the $k$-dilation of the map $F$ from $(S^m, g)$ to $(S^n, h)$.  Let $g_0$ and $h_0$ be the unit sphere metrics.  Suppose that $g$ is $L_1$ bilipschitz to $g_0$ and $h$ is $L_2$-bilipschitz to $h_0$.  Then

$$ L_1^{-k} L_2^{-k} \le \frac{\Dil_k^{(g,h)}(F)}{\Dil_k^{(g_0,h_0)}(F)} \le L_1^k L_2^k .$$

\noindent Therefore, if $F_i: S^m \rightarrow S^n$ is a sequence of maps, then $\Dil_k^{(g,h)}(F_i) \rightarrow 0$ if and only if $\Dil_k^{(g_0, h_0)}(F_i) \rightarrow 0$.  In particular, the definition of $V_k \pi_m(S^n)$ is independent of the choice of metric.

We can define a similar filtration on the homotopy groups of any finite simplicial complex.  Let $X$ be a finite simplicial complex.  Equip each simplex with the standard metric.  Then we say that $a \in \pi_m(X)$ belongs to $V_k \pi_m(X)$ if there are maps $F_i: S^m \rightarrow X$ in the homotopy class $a$ with $\Dil_k(F_i) \rightarrow 0$.  The $V_k \pi_m(X)$ form a filtration of $\pi_m(X)$: they are each subgroups of $\pi_m(X)$, with $0 = V_1 \pi_m(X) \subset V_2 \pi_m(X) \subset ... \subset V_m \pi_m(X) \subset \pi_m(X)$.  

\begin{lemma} The set $V_k \pi_m(X)$ is a subgroup of $\pi_m(X)$.  
\end{lemma}

\begin{proof}
If $a$ lies in $V_k \pi_m(X)$, then let $f_i$ be a sequence of maps
from $S^m$ to $X$ in the homotopy class $a$ with $k$-dilation
tending to zero.  Let $I$ be a reflection, mapping $S^m$ to
itself with degree -1, and taking the basepoint of $S^m$ to
itself.  Then the maps $f_i \circ I$ have $k$-dilations tending to
zero and lie in the homotopy class $-a$.  Therefore $-a$ lies in
$V_k \pi_m(X)$.  

Next, suppose that $a$ and $b$ lie in $V_k
\pi_m(X)$.  Again, let $f_i$ be a sequence of (pointed) maps in
the class $a$ with $k$-dilation tending to zero, and let $g_i$ be a
sequence of (pointed) maps in the homotopy class $b$ with
$k$-dilation tending to zero.  Let $I$ be a map from $S^m$ to $S^m
\vee S^m$ with degree (1,1).  Let $h_i$ be the map from $S^m \vee
S^m$ to $X$ whose restriction to the first copy of $S^m$ is equal
to $f_i$ and whose restriction to the second copy of $S^m$ is
equal to $g_i$.  Then the sequence $h_i \circ I$ has $k$-dilation
tending to zero.  Each map in the sequence lies in the homotopy
class $a+b$.  So $a+b$ lies in $V_k \pi_m(X)$. \end{proof}

\begin{lemma} For any finite simplicial complex $X$, the subgroups $V_k \pi_m(X)$ are nested, with $V_k \pi_m(X) \subset
V_{k+1} \pi_m(X)$. 
\end{lemma}

\begin{proof} For any map $F$, $\Dil_{k+1} F^{\frac{1}{k+1}} \le \Dil_k F^{1/k}$ by Proposition
\ref{dilkdill}.   In
particular, if $f_i$ is a sequence of maps with $k$-dilation
tending to zero, then the (k+1)-dilation of $f_i$ also tends to
zero.  Therefore, $V_k \pi_m(X) \subset V_{k+1} \pi_m(X)$. 
\end{proof}

These two lemmas show that $V_k \pi_m(X)$ form a filtration of $\pi_m(X)$.  
The filtration $V_k$ also behaves naturally under mappings.  

\begin{lemma} If $\Psi: X \rightarrow Y$ is a continuous pointed
mapping between finite simplicial complexes, then $\Psi_*: \pi_m(X)
\rightarrow \pi_m(Y)$ takes $V_k \pi_m (X)$ into $V_k \pi_m(Y)$.
\end{lemma}

\begin{proof}  First homotope $\Psi$ to a PL map with some finite
Lipshitz constant $L$.  Let $a$ be a class in $V_k \pi_m(X)$,
realized by mappings $f_i: S^m \rightarrow X$ with $k$-dilation
tending to zero.  The map $\Psi \circ f_i$ from $S^m$ to $Y$ has
$k$-dilation less than $L^k \Dil_k(f_i) \rightarrow 0$.  Each
map $\Psi \circ f_i$ lies in the homotopy class $\Psi_*(a)$. 
Therefore, $\Psi_*(a)$ lies in $V_k \pi_m(Y)$.
\end{proof}

In particular, if $\Psi$ is a homotopy equivalence, then $\Psi_*$ maps $V_k \pi_m(X)$ bijectively to $V_k \pi_m(Y)$.  In other words, the filtration we have defined is homotopy invariant.  Very little is known about $V_k \pi_m(X)$ for spaces $X$ besides $S^n$.

In the rest of this section, we mention a few other open problems.

\subsection{Are there minimizers in the $k$-dilation problem?}

One might try to study the $k$-dilation of mappings using the calculus of variations.  I'm not sure how much can
be achieved in this direction.  Let's start by framing some questions.
Suppose that we pick a homotopy class $a \in \pi_m(S^n)$, and we try to minimize the $k$-dilation of maps
$F: S^m \rightarrow S^n$ in the homotopy class $a$.  Will the infimum be achieved by a $C^1$ map?  
If not, will the infimum be achieved in some weaker space of maps?  The results in this paper
give a little bit of information about these questions, which we summarize here.

Our information about these questions comes from the following estimate.

\begin{prop} If $F: S^4 \rightarrow S^3$ has non-trivial Steenrod-Hopf invariant, then
$\Dil_2(F) \Dil_3(F) > c > 0$.
\end{prop}

\proof This argument is based on the construction of the cycle $Z(F)$ in Section \ref{cyclez(f)}.
Our cycle $Z(F)$ is a 6-cycle in $\Gamma_1 S^3$.  By Proposition \ref{dirvolz(f)}, the directed volume 
$\Vol_{(a,b,c)}(Z(F))$ is bounded by $C \Dil_b(F) \Dil_c(F)$.  Now $a + b + c \le 6$, and we know $a \le 1$, and $b,c \le 3$.  
Hence the vector $(b,c)$ is $(2,3)$, $(3,2)$,
or $(3,3)$.  So the total volume of $Z(F)$ is bounded by $C \Dil_2(F) \Dil_3(F) + C \Dil_3(F) ^2$.  If $F$ has non-trivial Steenrod-Hopf invariant, then
Proposition \ref{homz(f)=sh(f)} implies
 $Z(F)$ is homologically non-trivial.  In this case, the total volume of $Z(F)$ cannot be too small.  Therefore $\Dil_2(F) \Dil_3(F)$ is bounded below. \endproof

If $F: S^4 \rightarrow S^3$ has non-trivial Steenrod-Hopf invariant and yet $\Dil_3(F) < \epsilon$, then 
we see $\Dil_2(F) \gtrsim \epsilon^{-1}$ and so $\Dil_1(F) \gtrsim \epsilon^{-1/2}$.
(It's unclear how sharp these estimates are.  The mappings constructed in Proposition \ref{suspmeth} have 3-dilation $\epsilon$, 2-dilation
$\sim \epsilon^{-2}$, and 1-dilation $\sim \epsilon^{-1}$.)

As a corollary, we see that there is no homotopically non-trivial $C^1$ map $F$ from $S^4 \rightarrow S^3$ with $\Lambda^3 dF = 0$.  

So let's consider the problem of minimizing $\Dil_3(F)$ among all homotopically
non-trivial maps $S^4 \rightarrow S^3$.  It follows from the h-principle or from Proposition \ref{suspmeth} that the infimum is equal to zero.  
Since there is no homotopically non-trivial $C^1$ map from $S^4$ to $S^3$ with 3-dilation zero, we see that
the infimum is not achieved by a $C^1$ map.  

Now we could try to consider less regular maps.  It's easy to define the $k$-dilation of a piecewise $C^1$-map.  
With a little extra work, we could probably define the $k$-dilation for Lipschitz maps, for example by using Rademacher's theorem.  However, I believe that the last proposition could be extended to Lipschitz maps.  It would imply that $\Lip(F)^2 \Dil_3(F) > c > 0$ for any homotopically non-trivial map $F: S^4 \rightarrow S^3$.  This would imply that the infimum of $\Dil_3(F)$ is not achieved by any Lipschitz map.  

For maps which are not even Lipschitz, I am not sure how to define the $k$-dilation.

Can we extend the 3-dilation to some appropriate ``weak space of maps" where the infimum is achieved?

\subsection{The rank of the derivative and the topology of the mapping}

When I first began to work on this subject, I expected
that every homotopically non-trivial map from $S^m$ to $S^n$ must have $n$-dilation at least
$c(m,n) > 0$.  My incorrect intuition about the problem came partly from Sard's theorem.
According to Sard's theorem, every $C^\infty$ map $F$ from $S^m$ to $S^n$ has a full measure
set of regular values.  If the $n$-dilation of $F$ is zero, then every point in the domain is a critical point.
Hence every point in the image of $F$ is
a critical value.  According to Sard's theorem, if $F$ is $C^\infty$ with zero $n$-dilation, then the image of $F$ has measure zero.  
So we see that every $C^\infty$
map from $S^m$ to $S^n$ with $n$-dilation zero is contractible.  At first, I expected that this result
should extend to maps with sufficiently tiny $n$-dilation - but it does not.

In \cite{W}, Whitney discovered that Sard's theorem is false for $C^1$ maps.  
In the mid 1970's, Hirsch raised the question if there could be a surjective $C^1$ map from
$B^3$ to $B^2$ with 2-dilation zero.  In \cite{K}, Kaufman produced such a map.
Kaufman's technique can easily be generalized to construct surjective $C^1$ maps from
$S^3$ to $S^2$ with zero 2-dilation.  The maps constructed this way are contractible.
In fact, we have seen that
every map from $S^3$ to $S^2$ with zero two-dilation is contractible. 

We say that a $C^1$ map $F$ from one manifold to another has rank $< k$ if
the rank of $dF_x$ is less than $k$ for each $x$ in the domain.  A map has
rank less than $k$ if and only if it has $k$-dilation equal to zero.  The rank of a
map is a differential topological invariant.
We have very little knowledge about how the rank of a map is related to its homotopy type.

\newtheorem*{DTQ}{Rank of the derivative and topology of mappings} 

\begin{DTQ}  Let $F: S^m
\rightarrow S^n$ be a $C^1$ map with rank $< k$.  What can we conclude about the
homotopy type of $F$?
\end{DTQ}

We don't know any homotopically non-trivial $C^1$ map from $S^m$ to $S^n$ with rank $< n$.  Does one exist?  

A related question is whether there are homotopically non-trivial $C^1$ maps $F_i: S^m \rightarrow S^n$ with $\Dil_n(F_i) \rightarrow 0$
and uniformly bounded 1-dilation.

Added in proof. Recently, Wenger and Young addressed this question in \cite{WY}.  They proved the following result (page 2 of \cite{WY}).

\begin{theorem} (Wenger, Young) If $n+1 \le m < 2n-1$, then any Lipschitz map $f: S^m \rightarrow S^n$ can be extended to a
Lipschitz map $B^{m+1} \rightarrow B^{n+1}$ whose derivative has rank $\le n$ almost everywhere.
\end{theorem}

\begin{Cor} If $n+1 \le m < 2n-1$, and $a \in \pi_m(S^n)$, then the suspension of $a$ can be realized by a Lipschitz map
$S^{m+1} \rightarrow S^{n+1}$ whose derivative has rank $\le n$ almost everywhere.
\end{Cor}

\subsection{On thick tubes}

In Section \ref{twistedtubes}, we constructed $k$-expanding embeddings $I: S^1(\delta) \times B^{m-1}(1) \rightarrow B^m(\epsilon)$
for every $\epsilon > 0$ and for all $k > m/2$.  In other words, we constructed tubes with $k$-thickness 1 in arbitrarily small
balls $B^m(\epsilon)$ for all $k > m/2$.  

We don't know whether this is sharp.  It's straightforward to check that a tube with 1-thickness 1 does not embed in a small ball.  
We don't know if there are tubes with 2-thickness 1 in arbitrarily small balls $B^m(\epsilon)$ for every
dimension $m$.

We can generalize the question to embeddings from $S^p \times B^{m-p}$ into $B^m$.  The generalization of the thick tube construction in 
Section \ref{twistedtubes} gives the following lemma.

\begin{lemma} \label{thicktubesgen} If $k > \frac{m}{p+1}$, then for any $\epsilon > 0$, we can choose $\delta > 0$, and construct
a $k$-expanding embedding from $S^p(\delta) \times B^{m-p}(1)$ into $B^m(\epsilon)$.
\end{lemma}

We proved the lemma for the case $p=1$ in Section \ref{twistedtubes}.  Now we give a straightforward generalization to other values of $p$.

\begin{proof} The domain $S^p(\delta) \times B^{m-p}(1)$ is contained in $S^p(\delta) \times [-1,1]^{m-p}$.  Now
we apply a $k$-expanding diffeomorphism that shrinks one direction of the cube by a factor $\lambda > 1$, and grows
all the other directions by a factor $\lambda^{\frac{1}{k-1}}$.  This map is a $k$-expanding diffeomorphism to
$S^p(\delta \lambda^{\frac{1}{k-1}}) \times [- \lambda^{-1}, \lambda^{-1}] \times [- \lambda^{\frac{1}{k-1}}, \lambda^{\frac{1}{k-1}}]^{m-p-1}$.
Now we choose $\delta$ so that $\delta \lambda^{\frac{1}{k-1}} = \lambda^{-1}$.  In other words, $\delta = \lambda^{- \frac{k}{k-1}}$.
So $S^p(\delta \lambda^{\frac{1}{k-1}}) \times [- \lambda^{-1}, \lambda^{-1}] = S^p(\lambda^{-1}) \times [- \lambda^{-1}, \lambda^{-1}]$,
which admits a 1-expanding embedding into a ball $B^{p+1}(C \lambda^{-1})$.  In summary, we have constructed a $k$-expanding
embedding from our domain into $B^{p+1}(C \lambda^{-1}) \times B^{m-p-1}(C \lambda^{\frac{1}{k-1}})$.  The volume of this product of balls
is $\sim \lambda^{-(p+1) + \frac{ m-p-1}{k-1}}$.  The condition $k > \frac{m}{p+1}$ makes the exponent negative.  By taking $\lambda$ large,
we can make the volume as small as we want.  This product of balls then admits a 1-expanding embedding into an arbitrarily small
ball (see Section \ref{bilipembed} for details).  
 \end{proof}

Is it possible to find a $k$-expanding embedding $S^p(\delta) \times B^{m-p}(1)$ into a small ball for any $k \le \frac{m}{p+1}$?

\subsection{On $k$-dilation and Uryson width}

Let $F$ be a map from $S^m$ to $S^n$ with $k$-dilation $W$.  Let
$g_0$ denote the unit sphere metric on $S^m$, and let $h_0$ denote the unit sphere
metric on $S^n$.  Let $g$ denote the ``pullback metric" $F^* (h_0)$.   We use quotes because
the symmetric tensor $g$ may not be positive definite.  In fact it won't be positive definite in 
the interesting case that $m > n$, but $g$ is always a positive semi-definite symmetric 2-tensor.
We can let $\tilde g = g + \epsilon g_0$ for some tiny $\epsilon > 0$, so that $\tilde g$ is an honest metric
on $S^m$.  The $k$-dilation of $F$ is closely related to $\Lambda^k g$ - the $k^{th}$ exterior power of the metric
$g$.  In particular, $\Lambda^k g 
\le \Dil_k(F)^2 \Lambda^k g_0$.  If $\epsilon$ is small enough, then $\Lambda^k \tilde g
\le (1.01) \Dil_k(F)^2 \Lambda^k g_0$.  This setup motivates the following question.

What can we say about metrics $g$ on $S^m$ obeying $\Lambda^k g \le \Lambda^k g_0$?

The Uryson widths are fundamental metric invariants of Riemannian manifolds.  
Recall that the Ursyon q-width of a metric space $X$, denoted $UW_q(X)$, is
defined as follows.  We say that $UW_q(X) \le W$ if there is a continuous map
from $X$ to a q-dimensional polyhedron $P^q$ whose fibers each have diameter
less than $W$.  In other words, if $x_1, x_2 \in X$ are any two points of $X$ mapped
to the same point of $P$, then the distance $dist_X(x_1, x_2)$ should be $\le W$.
Intuitively, the Uryson q-width of $X$ measures ``how far $X$ is from looking like
a q-dimensional polyhedron".

The main facts about Uryson width are contained in the paper \cite{GWR}.
One fundamental fact is that the Uryson $(n-1)$-width of the unit $n$-cube $[0,1]^n$ is positive.  In fact, the Lebesgue
covering lemma implies that $UW_{n-1}( [0,1]^n ) = 1$.  More generally, the Uryson q-width of a Riemannian
manifold of dimension $> q$ is always positive.  For this section, we recall one other fact about Uryson width.

\begin{prop} \label{uwcontr} (\cite{GWR}) For each dimension $n$, there is a constant $\beta(n) > 0$ so that the following holds.  
If $X$ is a metric space and $UW_{n-1}(X) < \beta(n)$, and $F: X \rightarrow S^n$ has Lipschitz constant 1, then
$F$ is contractible.
\end{prop}

\newtheorem*{uwq}{Ursyon width question}
\begin{uwq} If $g$ is a metric on $S^m$ obeying $\Lambda^k g \le \Lambda^k g_0$, then
what can we say about the Uryson widths of $(S^m, g)$?
\end{uwq}

In a recent preprint \cite{GUUW}, I proved that the codimension 1 Uryson width is controlled by
the volume:

\newtheorem*{uwi}{Uryson width inequality}
\begin{uwi} If $(M^m, g)$ is a closed $m$-dimensional Riemannian manifold, then
$UW_{m-1}(M,g)$ is at most $C(m) \Vol(M,g)^{1/m}$.
\end{uwi}

(This inequality is a technical improvement on the filling radius inequality from \cite{GFRM}.)

If $\Lambda^m g \le \Lambda^m g_0$, then one knows that the total volume of $g$ is
at most $\Vol(S^m, g_0)$.  By the Uryson width inequality, we see that $UW_{m-1}(S^m, g)$
is bounded by a dimensional constant $C(m)$.

As $k$ decreases, the condition $\Lambda^k g \le \Lambda^k g_0$ becomes stronger,
and for sufficiently small $k$, it may control Uryson q-widths for some $q < m-1$.

There are some examples built using the construction of ``thick tubes" discussed in Lemma \ref{thicktubesgen}.

\newtheorem*{gzm}{Thick tube metrics}
\begin{gzm} Let $c$ be an integer $2 \le c \le m-1$.  If $k > m/c$, then there are metrics $g$ on $S^m$ with
$\Lambda^k g \le \Lambda^k g_0$ and $UW_{m-c}(S^m, g)$ arbitrarily large.
\end{gzm}

\begin{proof} Let $p = c-1$.  According to Lemma \ref{thicktubesgen}, we can construct a $k$-expanding embedding $I_1$ from
$S^p(\delta) \times B^{m-p}(R)$ into the upper hemisphere of $(S^m, g_0)$, with $R$ arbitrarily large.  (As $R$ increases,
$\delta$ decreases.)  Let $U \subset S^m$ be the image of the embedding.  Let $g$ be the pushforward of the metric from $S^p(\delta) \times B^{m-p}(R)$
onto $U$.  So $(U,g)$ is isometric to $S^p(\delta) \times B^{m-p}(R)$.  Since we used a $k$-expanding embedding, $\Lambda^k g \le \Lambda^k g_0$ on
$U$.  Now we extend $g$ to a metric on all of $S^m$ with $\Lambda^k g \le \Lambda^k g_0$ by making $g$ very small outside of $U$.

We claim that the Uryson width $UW_{m-p-1}(S^m, g)$ is $\gtrsim R$.  To see this we will prove that $(S^m, g_0)$ contains
an undistorted copy of $B^{m-p}(R/2)$.  In other words, we will find an embedding $I: B^{m-p}(R/2) \rightarrow (S^m, g)$ so that for any
two points $x,y \in B^{m-p}(R/2)$, $|x-y| = dist_g (I(x), I(y))$.  Then it follows that $UW_{m-p-1}(S^m, g) \ge UW_{m-p-1}(B^{m-p}(R/2)) \gtrsim R$.

The embedding $I$ is very simple.  We just pick a point $\theta \in S^p(\delta)$, and we define $I(x) = I_1(\theta, x)$.  It only remains to check
that this embedding is undistorted.  Let $x,y \in B^{m-p}(R/2)$.  First we note that the distance in $S^p(\delta) \times B^{m-p}(R)$ from $(\theta, x)$
to $(\theta,y)$ is just $|x-y|$.  Now, let $\gamma$ be a path from $I(x)$ to $I(y)$ in $(S^m, g)$.  If the path $\gamma$ stays in $U$, then the length
of $\gamma$ is at least $|x-y|$, because $(U,g)$ is isometric to $S^p(\delta) \times B^{m-p}(R)$.  But if the path $\gamma$ leaves $U$, it must contain
an arc from $I(x)$ to $\partial U$ and another arc from $I(y)$ to $\partial U$.  Each of these arcs has length at least $R/2$.  So the total length of 
$\gamma$ is at least $R \ge |x-y|$.  

So we see that $UW_{m-p-1}(S^m, g) = UW_{m-c}(S^m, g)$ can be arbitrarily large.  \end{proof}

Based on these examples, the following conjecture looks plausible.

\newtheorem*{uwc}{Uryson width conjecture}
\begin{uwc} Let $1 \le c \le m$.  Let $g$ be a metric on $S^m$ with $\Lambda^k g \le \Lambda^k g_0$, where
$g_0$ is the unit sphere metric on $S^m$.  If $k \le m/c$, then $UW_{m-c}(S^m, g) \le C(m)$.
\end{uwc}

The conjecture is true when $c=1$ by the Uryson width inequality above.  It is trivially true
when $c=m$, since $\Lambda^1 g \le \Lambda^1 g_0$ implies $\Diam(g) \le \Diam(g_0)$,
and $UW_0(S^m,g)$ is just the diameter of $(S^m,g)$.  In the range
$2 \le c \le m-1$, the conjecture is open.  

The Uryson width conjecture has implications for the questions we considered
above, including our main question.  The first implication is that our
construction of thick tubes is optimal.

\newtheorem*{ttc}{Thick tube conjecture} 

\begin{ttc} If $I$ is a $k$-expanding embedding
from $S^p(\delta) \times B^{m-p}(R)$ into the unit $m$-ball, and if $k \le \frac{m}{p+1}$,
then $R \lesssim 1$.
\end{ttc}

The second implication of the Uryson width conjecture is a general conjecture about $k$-dilation
and contractibility of mappings.

\newtheorem*{nhc}{Null-homotopy conjecture}
\begin{nhc} If $k \le m/c$ and $n > m - c$,
then every non-contractible map $F$ from the unit $m$-sphere to the unit $n$-sphere has
$\Dil_k(F) \ge c(m,n) > 0$.
\end{nhc}

The Uryson width conjecture implies the null-homotopy conjecture by the following argument.
Let $h_0$ be the metric on the unit $n$-sphere.  Let $g$ be the pullback metric $F^*(h_0)$
and let $\tilde g = g + \epsilon g_0$.  
We know that $\Lambda^k g \le \Dil_k(F)^2 \Lambda^k(g_0)$, and if $\epsilon$ is small enough, we can assume that $\Lambda^k \tilde g \le 2 \Dil_k(F)^2 \Lambda^k g_0$.   Since $k \le (m/c)$, the Uryson width conjecture implies that $UW_{m-c}(S^m, \tilde g) \lesssim \Dil_k(F)^{1/k}$.  Now the map $F: (S^m, \tilde g) \rightarrow (S^n, h_0)$ has Lipschitz constant 1.  If $UW_{m-c}(S^m, \tilde g)$ is small
enough, then Proposition \ref{uwcontr} implies that $F$ is contractible.

For example, the null-homotopy conjecture says that if $F: S^m \rightarrow S^{m-1}$ has
tiny $m/2$-dilation, then $F$ is contractible.  This statement is actually true by our
Steenrod square inequality.  (The Steenrod square inequality is a little stronger than
this statement, because it also applies to $(m+1)/2$-dilation.)

The null-homotopy conjecture also says that if $F: S^m \rightarrow S^{m-2}$ has tiny
$m/3$ dilation, then $F$ should be contractible.  This is an open problem.

\section{Appendices}

\subsection{A probability lemma} \label{secproblemma}

In this section, we recall and prove a simple probability lemma that we used a couple times in the paper.

Suppose that $X = \prod_{i \in I} X_i$ is a (countable or finite) product of probability spaces.  Suppose that $B \subset X$ is a ``bad" set, consisting of a union $B = \cup B_\alpha$.  We would
like to find a not-bad element of $X$ i.e. an element $x \in X$ which is not in $B$.  We know that the measure (probability) of each $B_{\alpha}$ is less than $\epsilon$ a small number.  But,
we have no control over the number of sets $B_{\alpha}$.  Therefore, on average, an element of
$X$ may lie in over a thousand different $B_\alpha$.  We can still find an element outside of $B$ provided that the sets $B_\alpha$ are ``localized" in the following sense.

\begin{lemma} \label{problemma} Suppose that $B$ is the union of sets $B_\alpha$ each with probability less than $\epsilon$.  
Suppose that each set $B_{\alpha}$ depends on $< C_1$ different coordinates
$x_i$ of the point $x$.  Suppose that each variable is relevant for $< C_2$ different bad
sets $B_{\alpha}$.  If $\epsilon < (1/2) C_2^{- C_1}$, then $B$ is not all of $X$.
\end{lemma}

This lemma is an easy corollary of the Lovasz local lemma.  The hypotheses imply that each set $B_{\alpha_0}$ is independent of the other sets
except for $C_1 C_2$ of them.  Then the local lemma implies our lemma with a better estimate for $\epsilon$.  The local lemma is proven
in \cite{EL} and \cite{AS}.  

Our lemma is quite easy, and we give a short self-contained proof as well.

\proof The idea is that we just choose the coordinates $x_1, x_2, ... $ one at a time in a reasonable way.

Let $I(\alpha) \subset I$ be the set of coordinates that are relevant for the bad set $B(\alpha)$.  We
know that the number of elements $| I (\alpha) | < C_1$.  Similarly, we let $A(i)$ be the set of bad events $\alpha$
which depend on the coordinate $x_i$.  We know that the number of elements $| A(i) | < C_2$.  
We let $P(\alpha)$ be the measure of $B(\alpha)$.  After choosing $x_1, x_2, ..., x_i$, we let $P_i(\alpha)$ be the conditional probability of 
landing in $B(\alpha)$ after randomly making all other choices.  We let $I_i (\alpha)$ be the set of
coordinates $j \in I(\alpha)$ with $1 \le j \le i$.

When we choose $x_{i+1}$, we affect some of the probabilities.  
If $\alpha$ is not in $A(i+1)$, then $P_{i+1}(\alpha) = P_i(\alpha)$.   But if $\alpha \in A(i+1)$, then $P_{i+1}(\alpha)$ may be different
from $P_i(\alpha)$.  When we randomly pick $x_{i+1}$, the probability that $P_{i+1}(\alpha) > C_2 P_i(\alpha)$ is $\le C_2^{-1}$.  Since
$A(i+1)$ contains $< C_2$ values of $\alpha$, we can choose $x_{i+1}$ so that $P_{i+1}(\alpha) \le C_2 P_i(\alpha)$ for every $\alpha \in A(i+1)$.  

Hence by induction, we have $ P_i(\alpha) \le C_2^{| I_i(\alpha)| } \epsilon$ for each $\alpha$.

After we have chosen all the $x_i$, the conditional probability $P_\infty(\alpha)$ is either 0 or 1.  $P_\infty(\alpha) = 1$ if the point
$x = (x_1, x_2, ...)$ lies in $B(\alpha)$, and $P_\infty(\alpha) = 0$ if it doesn't.  Our inequality on $P_i(\alpha)$ becomes in the limit
$P_\infty(\alpha) \le C_2^{C_1} \epsilon \le 1/2 $, and so the point $x$ does not lie in any bad set $B(\alpha)$. \endproof

\subsection{Bilipschitz embeddings of rectangles} \label{bilipembed}

At several points in the paper we use a bilipschitz embedding from some rectangular solid into a unit ball.  These
embeddings can all be derived from the following basic lemma, which describes when there is a bilipschitz embedding
from one rectangular solid into another.

\begin{lemma} \label{bilipembedlemma} Suppose that $R$ and $S$ are $n$-dimensional rectangles.  Let $R = \prod_{j=1}^n [0, R_j]$ with $R_1 \le ... \le R_n$, and
let $S = \prod_{j=1}^n [0, S_j]$ with $S_1 \le ... \le S_n$.  If $\prod_{j=1}^p R_j \ge \prod_{j=1}^p S_j$ for all $p$ in the range $1 \le p \le n$, then there is a locally $C(n)$-bilipschitz embedding from $S$ into $R$.
\end{lemma}

Recall that an embedding $I: S \rightarrow R$ is called locally $L$-bilipschitz if it distorts the lengths of tangent vectors by at most a factor of $L$.  More precisely, if $v$ is any tangent vector in $S$, then $|v| / L \le | dI (v) | \le L |v|$.

The proof is by induction on the dimension.  Unfortunately, the algebra is a bit tedious.  It has the following corollary.

\newtheorem*{cor}{Corollary}
\begin{cor} Suppose that $A$ is an $n$-dimensional convex set in $\mathbb{R}^n$ with volume 1.  Then there is a locally $C(n)$-bilipschitz embedding into the unit $n$-ball or into the upper hemisphere of the unit $n$-sphere.
\end{cor}

\begin{proof} After a rotation, the set $A$ is a subset of a rectangle $R$ with volume $\le C(n)$, for some $C(n) > 1$.  
The rectangle has side lengths $R_1 \le ...
\le R_n$ and $\prod_{j=1}^n R_j \le C(n)$.  Since the $R_j$ are increasing, we have $\prod_{j=1}^p R_j \le C(n)$ also.
By Lemma \ref{bilipembedlemma}, this rectangle admits a
locally $C(n)$-bilipschitz embedding into the unit cube.  The unit cube has a $C(n)$-bilipschitz embedding into the unit ball or the upper hemisphere.
\end{proof}

Lemma \ref{bilipembedlemma} is sharp up to constant factors.  If there is an $L$-bilipschitz embedding from $S$ into $R$, then $R_1... R_p \ge c(n) L^{-p} S_1 ... S_p$ for each $p$ from 1 to $n$.  A proof is given in \cite{GUWV}.  The known proofs are surprisingly difficult.  All the proofs use homology theory.  It would be interesting to find a really elementary proof.  Now we give the proof of Lemma \ref{bilipembedlemma}.

\begin{proof} The proof is by induction on $n$.  The base case is $n=2$.

Suppose that $R_1 \ge S_1$ and $R_1 R_2 \ge S_1 S_2$.  If $R_2 \ge S_2$, then the identity map is an embedding from $S$ into $R$, and there is nothing to prove.  If $R_2 < S_2$, then $S$ is longer and thinner than $R$, and the area of $S$ is smaller than the area of $R$.  In this case, we can make a locally 10-bilipschitz embedding by folding $S$ back and forth inside of $R$.

For general $n$, we construct the bilipschitz embedding by using this construction repeatedly with different coordinates.  We know that $R_1 \ge 
S_1$.  If $R_j \ge S_j$ for all $j$, then the identity map is an embedding from $S$ into $R$, and there is nothing to prove.  Otherwise, let $a$ be the smallest value so that $R_a < S_a$.  We know that $a \ge 2$, and so $R_{a-1} \ge S_{a-1}$.  

We will define a rectangle $S'$ with $S_j' = S_j$ except for $j = a-1$ and $j=a$, and we will use the 2-dimensional case to find a 10-bilipschitz embedding from $S$ into $S'$.  Now $S'$ will have the property that either $S_a' = R_a$ or else $S_{a-1}' = R_{a-1}'$.  Then by induction, we will construct a $C(n-1)$-bilipschitz embedding from $S'$ into $R$.  Now we turn to the details.

We consider the ratios $R_{a-1} / S_{a-1} \ge 1$ and $S_a / R_a \ge 1$.  We proceed in two cases, depending on which ratio is larger.

Suppose first that $R_{a-1} / S_{a-1} \ge S_a/R_a$.  Define $S_a' = R_a$ and $S_{a-1}' = (S_a/R_a) S_{a-1} \le R_{a-1} \le R_a = S_a'$.  
We note that $S_{a-1}' \ge S_{a-1} \ge S_{a-2}$ and $S_a' \le S_a \le S_{a+1}$.  Now we define $S_j' = S_j$ for all $j$ except $a-1$ and $a$.  The inequalities we have proven show that $S_j'$ are in order: $S_1' \le S_2' \le ... \le S_n'$.  We let $S'$ be the corresponding rectangle $\prod_{j=1}^n [0, S_j']$.  By the 2-dimensional case, there is a 10-bilipschitz embedding from $[0, S_{a-1}] \times [0, S_a]$ into $[0, S_{a-1}'] \times [0, S_a']$.  Using the identity in the other coordinates, we get a 10-bilipschitz embedding from $S$ into $S'$.

We claim that $\prod_{j=1}^p S_j' \le \prod_{j=1}^p R_j$ for all $p$.  For $p \le a-2$, this follows because $\prod_{j=1}^p S_j' = \prod_{j=1}^p S_j$.  
We also note that $S_{a-1}' S_a' = S_{a-1} S_a$, and so for $p \ge a$, $\prod_{j=1}^p S_j' = \prod_{j=1}^p S_j \le \prod_{j=1}^p R_j$.  Finally, we have to consider $p = a-1$.  Since $S_{a-1}' \le R_{a-1}$, $\prod_{j=1}^{a-1} S_j' \le (\prod_{j=1}^{a-2} S_j') R_{a-1} \le (\prod_{j=1}^{a-2} R_j) R_{a-1} = \prod_{j=1}^{a-1} R_j$.  This proves the claim.  Now we let $R = \bar R \times [0, R_a]$ and $S' = \bar S' \times [0, R_a]$, where
$\bar R$ is the product of $[0, R_j]$ for $j \not= a$, and $\bar S'$ is the product of $[0, S_j']$ for $j \not=a$.  By induction, we see that there is a
$C(n-1)$-bilipschitz embedding from $\bar S'$ into $\bar R$.  Using the identity in the $a$-coordinate, we get a $C(n-1)$-bilipschitz embedding from $S'$ into $R$.  Composing our two embeddings, we get a $10 C(n-1)$-bilipschitz embedding from $S$ into $R$.

The other case is similar.  Suppose that $R_{a-1} / S_{a-1} \le S_a / R_a$.  Define $S_{a-1}' = R_{a-1}$ and $S_a' = (S_{a-1}/R_{a-1}) S_a$.  We note that
$R_a \le S_a' \le S_a$.  Therefore $S_{a-2} \le S_{a-1} \le S_{a-1}' = R_{a-1} \le R_a \le S_a' \le S_a \le S_{a+1}$.  We define $S_j ' = S_j$ for all $j$ 
except $a-1$ and $a$.  The inequalities we have proven show that the $S_j'$ are in order, and we define $S' = \prod_{j=1}^n [0, S_j']$  By the 
2-dimensional case there is a 10-bilipschitz embedding from $S$ into $S'$.  By the same arguments as above, we can check that $\prod_{j=1}^p S_j' \le \prod_{j=1}^p R_j$ for all $p$.  This time, $S_{a-1}' = R_{a-1}$.  We let $R = \bar R \times [0, R_{a-1}]$ and $S' = \bar S' \times [0, S'_{a-1}]$.  By induction, we see that there is a
$C(n-1)$-bilipschitz embedding from $\bar S'$ into $\bar R$.  Using the identity in the $a$-coordinate, we get a $C(n-1)$-bilipschitz embedding from $S'$ into $R$.  Composing our two embeddings, we get a $10 C(n-1)$-bilipschitz embedding from $S$ into $R$.  \end{proof}

\subsection{Basic facts about flat chains and flat equivalence} \label{appenflat}

We use some basic facts about flat chains and flat equivalence in Section \ref{newpfhopf} and Section \ref{cyclez(f)}.  In this appendix,
we review the basic facts.

The flat norm is usually defined for chains in a Riemannian manifold.  Here we have to work with chains in a finite CW complex with Lipschitz attaching maps.  This is only slightly harder.  If $X$ is a finite CW complex with Lipschitz attaching maps, then $X$ is a metric space in a natural way.  The complex $X$ is given by finitely many closed balls with some identifications.  We put the standard unit ball metric on each closed ball, and we define the metric on $X$ to be the quotient metric coming from the identifications.  So we can define Lipschitz maps into $X$ and Lipschitz chains.  The volume of a Lipschitz chain is defined by breaking the chain into pieces in each open cell, and the volume of each piece is defined in the usual way.

One fundamental result about the volumes of chains is that a cycle of small volume must be null-homologous.  We formulate this as a lemma and prove it by the standard Federer-Fleming deformation argument.

\begin{lemma} \label{defCW} If $X$ is a finite CW complex with Lipschitz attaching maps, then there is a constant $\epsilon > 0$ so that the following holds.  If
$z$ is a mod 2 Lipschitz cycle in $X$ with volume $< \epsilon$, then $z$ is homologically trivial. \end{lemma}

\begin{proof} Suppose that $z$ is a d-cycle.  We homotope $z$ into the d-skeleton of $X$ while keeping control of the volume.  We may assume that all the attaching maps have Lipschitz constant $< L$.  If $z$ initially lies in the N-skeleton of $X$ for some $N > d$, then we homotope it to the (N-1)-skeleton by picking a random point near the middle of each N-cell, and pushing out radially into the boundary of the cell.  By the Federer-Fleming averaging trick, we can choose a point so that this push out map increases volumes by at most a factor $C(N)$.  From the boundary of the cell, we map into the (N-1)-skeleton of $X$ using the attaching maps, which stretch volumes by at most a factor $L^d$.  Repeating this for each dimension, we homotope $z$ to a cycle $z'$ in the d-skeleton of $X$ with volume at most $C(X) \epsilon$.  If $\epsilon$ is small enough, then $z'$ doesn't cover any d-cell of the d-skeleton, and so $z'$ is null-homologous. \end{proof}

If $T$ is a mod 2 Lipschitz d-chain in a CW complex, then the flat norm of $T$ is defined to be the infimum over all
(d+1)-chains $U$ of $\Vol_{d+1}(U) + \Vol_d(T - \partial U)$.  In other words, a chain $T$ may have a small norm if it has small volume,
or if it is the boundary of a (d+1)-chain with small volume, or if it is the sum of pieces of these two types.  It's staightforward to check that the flat norm obeys the triangle inequality: $\FlatNorm( T_1 + T_2) \le \FlatNorm (T_1) + \FlatNorm (T_2)$.  

The flat distance between $T_1$ and $T_2$ is defined to be the flat norm of $T_1 - T_2$.  If the flat distance between two Lipschitz chains is zero, we say they are flat equivalent.  Because the flat norm obeys the triangle inequality, it follows that flat equivalence is an equivalence relation.  The flat norm defines a metric on the set of equivalence classes of Lipschitz chains.  

The resulting metric space is not complete, and the space of flat chains is the completion of this metric space.   However, in this paper, we only need the notion of flat equivalence.

Here are some examples of flat equivalence.  If $T_1$ and $T_2$ differ by a chain with volume zero, then they are flat equivalent.  Also, if $T_1$ and $T_2$ are two homologous d-dimensional cycles in a d-dimensional complex, then they are flat equivalent, because $T_1 - T_2$ is the boundary of a (d+1)-chain which must have zero (d+1)-volume.  The different flat equivalences that appear in Sections \ref{newpfhopf} and \ref{cyclez(f)} just come from these two observations.

The last small result that we need is that two flat equivalent Lipschitz cycles are homologous.

\begin{lemma} Suppose that $z_1$ and $z_2$ are flat equivalent Lipschitz cycles in a finite CW complex $X$ with Lipschitz attaching maps.
Then $z_1$ and $z_2$ are homologous.
\end{lemma}

\begin{proof} We know that the flat norm of $z_1 - z_2$ is zero.  So for any $\epsilon > 0$, we can find a Lipschitz chain $U$ so that $U$ has volume $< \epsilon$ and $\partial U - z_1 + z_2$ has volume $< \epsilon$.  Obviously, $\partial U$ is homologically trivial.  If $\epsilon$ is small enough, then $\partial U - z_1 + z_2$ is homologically trivial by Lemma \ref{defCW}.  Therefore, $z_1 - z_2$ is homologically trivial. \end{proof}

\subsection{Standard facts about the deformation operator} \label{fedflemreview}

In this section, we review some standard facts about the deformation operator, which we stated in Section \ref{defop}.  We work
with mod 2 chains and cycles.  (The statements here can be extended to other coefficients, but we don't need them and it takes extra work to keep track of the orientations.)

If $T$ is a d-chain, $s > 0$ is a scale, and $v \in \mathbb{R}^N$ is a vector, then we define the
deformation operator $D_v(T)$ by the following formula,

$$D_v T:= \sum_{F \subset \Sigma^d(s)} [\bar F_v \cap T] F.$$

\noindent In this formula, $[\bar F_v \cap T] \in \mathbb{Z}_2$ is the number of points in $\bar F_v \cap T$ taken mod 2. 
If $\bar \Sigma(s)$ is transverse to $T$, then $D_v T$ is well-defined.  The deformation $D_v T$ is a cubical d-chain in $\Sigma^d(s)$.   

The deformation operator has the following properties.

1. If $|v| < s/2$, and if $T$ is a cubical d-chain in $\Sigma(s)$, then $D_v(T) = T$.

\begin{proof} We just have to check that if $F, G$ are d-faces of $\Sigma(s)$, then $[\bar F_v \cap G]$ is equal to 1 if $F = G$ and 0 if $F \not= G$.  This holds for $v=0$.  The boundary of $\bar F$ lies at a distance $\ge s/2$ from the face $G$.  So as we continuously translate $\bar F$ to $\bar F_v$, the intersection number doesn't change.
\end{proof}

2. The deformation operator commutes with taking boundaries.  In other words, as long as $\bar \Sigma_v$ is transverse to
both $\partial T$ and $T$,  $ \partial D_v(T) = D_v( \partial T)$.

\begin{proof} From the formula for $D_v(T)$, we see that 

$$\partial D_v(T) = \sum_{F^d \subset \Sigma^d(s)} [\bar F_v \cap T] \partial F.$$

Consider a (d-1)-face $G$ in $\Sigma^{d-1}$.  Let $F_1(G), ..., F_{2(N-d+1)}(G)$ be the set of all the d-faces of $\Sigma^d(s)$ that contain $G$ in their boundary.  We can rewrite the formula for $\partial D_v(T)$ as follows:

$$ \partial D_v(T) = \sum_{G^{d-1} \subset \Sigma^{d-1}(s)} \left( \sum_{j=1}^{2(N-d+1)} [\overline{ F_j(G)}_v \cap T] \right) G. $$

Now the first key point is that $\sum_{j=1}^{2(N-d+1)} \overline{F_j(G)} = \partial \bar G$.  Plugging in, we get

$$ \partial D_v(T) = \sum_{G^{d-1} \subset \Sigma^{d-1}(s)} [\partial \bar G_v \cap T] G. $$

 Since $\bar \Sigma_v$ is transverse to $T$, $\bar G_v \cap T$ is a
1-chain, and the boundary of $\bar G_v \cap T$ consists of an even number of points.  Since $\bar \Sigma_v$ is transverse to
$T$ and $\partial T$, the boundary of $\bar G_v \cap T$ is the union of $\partial \bar G_v \cap T$ and $\bar G_v \cap \partial T$.
Therefore, $[\partial \bar G_v \cap T] = [\bar G_v \cap \partial T]$.  Substituting this identity into the last equation, we get

$$\partial D_v(T) =  \sum_{G^{d-1} \subset \Sigma^{d-1}(s)} [\bar G_v \cap \partial T] G = D( \partial T). $$

\end{proof}

3.  If we average over all $|v| < s/2$, then

$$ \Average_{v \in B(s/2)}  \Vol_d [ D_v(T) ] \le C(N) \Vol_d(T). $$

\begin{proof} This follows by integral geometry.  If $F$ is a face of $\Sigma^d(s)$, let $B[F]$ denote the ball around the center of $F$ with radius $Ns$.  If we take a random vector $v \in B(s/2)$, the probability that $[\bar F_v \cap T] = 1$
is at most $C(N) s^{-N} \Vol_d( T \cap B[F] )$.  Therefore the average volume on the left-hand side is 

$$ \le C(N) \sum_F \Vol_d(T \cap B[F]) \le C(N) \Vol_d(T). $$

\end{proof}

4. If $z$ is a d-cycle, then we can build a (d+1)-chain $A_v(z)$ in the $C(N) s$ neighborhood of $z$ with $\partial A_v(z) = z - D_v(z)$.  Moreover, if we average over all $|v| < s/2$, then

$$ \Average_v \Vol_{d+1} [ A_v(z)] \le C(N) s \Vol_{d}(z). $$

This estimate takes a little more work.  There are several variations of the deformation operator.  We begin by recalling
a different point of view about the deformation operator, where the chain $A$ appears more naturally.  Then we see how
the different points of view are connected.

Suppose that $z$ is a d-cycle in $\mathbb{R}^N$.  Federer and Fleming gave a procedure to homotope $z$ into $\Sigma^d_v(s)$ (which is well-defined for almost every $v$).  For each N-cube $Q^N$ of $\Sigma_v(s)$, we project $z$ outward from the center to the boundary of $Q$.  As long as $z$ doesn't intersect the center point, we get a homotopy into the (N-1)-skeleton of $\Sigma_v(s)$.  If $d < N-1$, we repeat this operation with each (N-1)-cube $Q^{N-1}$ of $\Sigma_v(s)$.  We continue in this way until we have homotoped $z$ into the d-skeleton of $\Sigma_v(s)$.  We can do this as long as, at each step of the homotopy, the image of $z$ does not include any of the center points of the cubes.

This procedure defines a homotopy $H_v: z \times [0,1] \rightarrow \mathbb{R}^N$, for $t \in [0,1]$, where $H_v$ at time 0 is the identity and $H_v$ at time 1 maps $z$ into the d-skeleton of $\Sigma_v(s)$.  (We will see below that the homotopy $H_v$ is defined for almost every $v$.)

We notice that $H_v(z, 1)$ is a d-cycle in $\Sigma^d_v(s)$.  Since $\Sigma^d_v(s)$ is a d-dimensional polyhedron, $H_v(z,1)$ is homologous to a sum of faces.  In other words, we have

$$H_v(z,1) = \sum_{F \in \Sigma^d(s)} c(F) F_v + \partial \nu, $$

\noindent where $c(F)$ are coefficients and $\nu$ is a (d+1)-chain in $\Sigma^d_v(s)$.  Note that $\nu$ is a (d+1)-chain with $\Vol_{d+1} (\nu) = 0$, so $H_v(z,1)$ is essentially equal to $\sum_{F \in \Sigma^d(s)} c(F) F_v$.  We now define the Federer-Fleming deformation of $z$ by

$$ \tilde D_v(z) := \sum_{F \in \Sigma^d(s)} c(F) F_v. $$

(The chain $\tilde D_v(z)$ is closely related to $D_v(z)$, as we explain below, but they are not identical.)

We now compute the constant $c(F)$.  The constant $c(F)$ measures the number of times that $H_v(z,1)$ covers the face $F_v$, taken mod 2.

If $Q$ is an e-face of $\Sigma_v(s)$ with center $x_Q$, let $\pi_Q: Q \setminus \{ x_Q \} \rightarrow \partial Q$ be the radial projection.  Notice that the center $x_Q$ is $Q \cap \bar \Sigma^{N-e}_v(s)$.   By applying the radial projection $\pi_Q$ in each e-face $Q$, we get a map $\pi_d: \Sigma^e_v(s) \setminus \bar \Sigma^{N-e}_v(s) \rightarrow \Sigma^{e-1}_v(s)$.   To get a map from $\mathbb{R}^N$ to $\Sigma^d_v(s)$, we use the composition $\pi:= \pi_{d+1} \circ ... \circ \pi_N$.  

The map $\pi$ is not defined on all of $\mathbb{R}^N$, but it is a well-defined map from $\mathbb{R}^N \setminus \bar \Sigma^{N-d-1}_v(s)$ to $\Sigma^d_v(s)$.  (To see this, we just have to check that for each $e \ge d+1$, $\pi_e$ maps $\Sigma^e_v(s) \setminus \bar \Sigma^{N-d-1}_v(s)$ into $\Sigma^{e-1}_v(s) \setminus \bar \Sigma^{N-d-1}_v(s)$.)  Therefore, the homotopy $H$ is well-defined as long as $z$ is disjoint from $\bar \Sigma^{N-d-1}_v(s)$, which happens for almost every $v$.

If $F_v$ is a d-face of $\Sigma_v(s)$ with center $x(F_v)$, then $\pi^{-1}(x(F_v))$ is just $\bar F_v$ - the perpendicular (N-d)-face of $\bar \Sigma_v(s)$.  If $z$ is transverse to $\bar \Sigma_v(s)$, then we see that the coefficient $c(F)$ is just the intersection number $c(F) = [z \cap \bar F_v]$.   Therefore, we get the following formula for $\tilde D_v(z)$:

$$ \tilde D_v(z) = \sum_{F \in \Sigma^n(s)} [z \cap \bar F_v] F_v. $$

So we see that $\tilde D_v(z)$ is just the translation of $D_v(z)$ by the vector $v$.

Now we can define the homology $A_v(z)$.  Since $\tilde D_v(z)$ is just a translation of $D_v(z)$ by a vector $v$, there is an obvious homotopy between them, given by translations.  This homotopy defines a chain $H'$ with $\partial H' = D_v(z) - \tilde D_v(z)$. The chain $A_v(z)$ is the sum of $H_v(z \times [0,1])$ and the chain $\nu$ and the chain $H'$.

\begin{lemma} If $|v| \le s$ and if we choose the zero-volume chain $\nu$ correctly, then $A_v(z)$ is contained in the $C(N) s$ neighborhood of $z$.

\end{lemma}

\begin{proof} By construction, the homotopy $H_v$ displaces points by $\le C(N) s$: in other words, $|H_v(x,t) - x| \le C(N) s$.  Therefore, $H_v(z \times [0,1])$ lies in the $C(N) s$ neighborhood of $z$.  Therefore, $\tilde D_v(z)$ lies in the $C(N) s$ neighborhood of $z$.  Now $H_v(z, 1) - \tilde D_v(z)$ is a null-homologous cycle in $\Sigma_v^d(s)$ lying in the $C(N) s$ neighborhood of $z$.  Therefore, we can fill it by a chain $\nu$ in $\Sigma_v^d(s)$ lying in the same neighborhood.  Finally, since $|v| \le s$, the homotopy $H'$ lies in the $s$-neighborhood of $\tilde D_v(z)$ and in the $C(N) s$-neighborhood of $z$.  \end{proof}

Next we turn to bounding the volume of $A_v(z)$.  Federer and Fleming observed that if we take a random vector $v$ in $B(s/2)$, then there are several useful volume estimates that hold on average.

\begin{prop} The following estimates hold for the average behavior of $H_v$ and $\tilde D_v(z)$:

$$ \Average_{v \in B(s/2)}  \Vol_d \tilde D_v(z)  \le C(N) \Vol_d z. $$

$$ \Average_{v \in B(s/2)}  \Vol_{d+1} H( z \times [0,1]) \le C(N) s \Vol_d z. $$

\end{prop}

We sketch the proof of the proposition.  For more details, see \cite{GFRM} pages 16-20.  By a direct computation, one shows that for any d-cycle $z$, 

$$\Vol_d \pi(z) \le C(N) \int_z \Dist(x, \bar \Sigma^{N-d-1}_v(s) )^{-d}  dvol_z (x) . $$

If we use a random translation $v \in B(s/2)$, then the average value of the last line is

$$ C(N) s^{-N} \int_{B(s/2)} \left( \int_z \Dist(x, \bar \Sigma^{N-d-1}_v(s) )^{-d}  dvol_z (x)  \right) dv. $$

The key insight of Federer-Fleming is to estimate this double integral using Fubini.  It is equal to

$$ C(N) \int_z \left( s^{-N} \int_{B(s/2)} \Dist(x, \bar \Sigma^{N-d-1}_v(s) )^{-d} dv \right) dvol_z(x). $$

Now the expression in the large parentheses does not depend on $z$, and it is bounded $\le C(N)$ uniformly 
in $x$.  Therefore, the whole last line is $\le C(N) \Vol_d z$.  

By another direct computation, the (d+1)-volume of the homotopy $H$ from $z$ to $\pi(z)$ is bounded by 

$$\Vol_{d+1} H_v(z \times [0,1] ) \le s C(N) \int_z \Dist(x, \bar \Sigma^{N-d-1}_v(s) )^{-d}  dvol_z (x) . $$

And the same argument shows that $s^{-N} \int_{B(s/2)} \Vol_{d+1} H_v(z \times [0,1]) \le C(N) s \Vol_d z$. 

The last estimate is the main term in bounding the volume of $A_v(z) = H_v(z \times [0,1]) + \nu + H'$.  The chain
$\nu$ has zero volume.  The chain $H'$ is given by translating $\tilde D_v(z)$ to $D_v(z)$, and so it has volume
at most $|v| \Vol_{d} \tilde D_v(z) \le C(N) s \Vol_d z$.  Therefore, for an average $v \in B^N(s/2)$, 
the chain $A_v(z)$ has (d+1)-volume at most $C(N) s \Vol_d z$.


\begin{thebibliography}{5}

\vskip.125in

\bibitem[AS]{AS} Alon, Noga, Spencer, Joel,  {\it The Probabilistic Method} Third edition. With an appendix on the life and work of Paul Erd{\H o}s. Wiley-Interscience Series in Discrete Mathematics and Optimization. John Wiley and Sons, Inc., Hoboken, NJ, 2008.

\bibitem[Ar]{Ar} Arnold, V. I. The asymptotic Hopf invariant and its applications. Selected translations. Selecta Math. Soviet. 5 (1986), no. 4, 327-345. 

\bibitem[BS]{BS} Bechtluft-Sachs, Stefan, Infima of universal energy functionals on homotopy classes. 
Math. Nachr. 279 (2006), no. 15, 1634-1640. 

\bibitem[BT]{BT} Bott, Raoul, Tu, Loring, {\it Differential Forms in Algebraic Topology} 
Graduate Texts in Mathematics, 82. Springer-Verlag, New York-Berlin, 1982.

\bibitem[DGS]{DGS} DeTurck, D., Gluck H., and Storm, P.,  Lipschitz minimality of Hopf fibrations and Hopf vector fields, arXiv:1009.5439.

\bibitem[EL]{EL} Erd{\H o}s, P.; Lov\'asz, L. Problems and results on 3-chromatic hypergraphs and some related questions. Infinite and finite sets (Colloq., Keszthely, 1973; dedicated to P. Erd?s on his 60th birthday), Vol. II, pp. 609-627. Colloq. Math. Soc. Janos Bolyai, Vol. 10, North-Holland, Amsterdam, 1975.

\bibitem[EM]{EM} Eliashberg, Y., and Mishachev, N. , {\it Introduction to the h-principle}  Graduate Studies in Mathematics, 48. American Mathematical Society, Providence, RI, 2002.

\bibitem[FH]{FH} Freedman, Michael H.; He, Zheng-Xu Divergence-free fields: energy and asymptotic crossing number. Ann. of Math. (2) 134 (1991), no. 1, 189-229.

\bibitem[Ge]{Ge} Gehring, F. W. Inequalities for condensers, hyperbolic capacity, and extremal lengths. Michigan Math. J. 18 1971 1-20. 

\bibitem[GCC]{GCC} Gromov, M., Carnot-Carathéodory spaces seen from within. in {\it Sub-Riemannian geometry}, 79-323,
Progr. Math., 144, Birkhäuser, Basel, 1996. 

\bibitem[GFRM]{GFRM} Gromov, M., Filling Riemannian manifolds. J. Differential Geom. 18 (1983), no. 1, 1-147. 

\bibitem[GHED]{GHED} Gromov, M.; Homotopical effects of dilatation, J. Differential Geom. 13 (1978), no. 3, 303-310. 

\bibitem[GMS]{GMS} Gromov, Misha; {\it Metric Structures on Riemannian and Non-Riemannian Spaces}, Modern Birkhäuser Classics. 
Birkhauser Boston, Inc., Boston, MA, 2007.

\bibitem[GPDR]{GPDR} Gromov, M., {\it Partial Differential Relations}, Ergebnisse der Mathematik und ihrer Grenzgebiete (3)
 [Results in Mathematics and Related Areas (3)], 9. Springer-Verlag, Berlin, 1986.

\bibitem[GWR]{GWR} Gromov, M. Width and related invariants of Riemannian manifolds. On the geometry of differentiable manifolds (Rome, 1986). 
Astérisque No. 163-164 (1988), 6, 93-109, 282 (1989).

\bibitem[GUT]{GUT} Guth, L., Area-contracting maps between rectangles, PhD thesis

\bibitem[GURH]{GURH} Guth, L., Isoperimetric inequalities and rational homotopy invariants, arXiv:0802.3550.

\bibitem[GUUW]{GUUW} Guth, L., Uryson width and volumes of balls, preprint

\bibitem[GUWV]{GUWV} Guth, L., The width-volume inequality, Geom. Funct. Anal. 17 (2007), no. 4, 1139-1179.

\bibitem[H]{H} Hatcher, A., {\it Algebraic Topology}, Cambridge University Press, Cambridge, 2002. 

\bibitem[K]{K} Kaufman, R., A singular map of a cube onto a square. J. Differential Geom. 14 (1979), no. 4, 593-594.

\bibitem[M]{M} Milnor, J., {\it Topology from the Differentiable Viewpoint}, Based on notes by David W. Weaver. 
Revised reprint of the 1965 original. Princeton Landmarks in Mathematics. Princeton University Press, Princeton, NJ, 1997.

\bibitem[R]{R} Riviere, T. , Minimizing fibrations and p-harmonic maps in homotopy classes from $S^3$ into $S^2$,
Comm. Anal. Geom. 6 (1998), no. 3, 427-483. 

\bibitem[T]{T} Toda, H., {\it Composition Methods in Homotopy Groups of Spheres}, Annals of Mathematics Studies, No. 49
 Princeton University Press, Princeton, N.J. 1962.

\bibitem[TW]{TW} Tsui, Mao-Pei; Wang, Mu-Tao,
 Mean curvature flows and isotopy of maps between spheres. Comm. Pure Appl. Math. 57 (2004), no. 8, 1110-1126. 

\bibitem[WY]{WY} Wenger, S. and Young, R., Lipschitz homotopy groups of the Heisenberg group, arXiv:1210.6943

\bibitem[W]{W} Whitney, H., A function not constant on a connected set of critical points.
Duke Math. J. 1 (1935), no. 4, 514-517. 

\bibitem[YL]{YL} Young, L. C., Some extremal questions for simplicial complexes. V. The relative area of a Klein bottle.
Rend. Circ. Mat. Palermo (2) 12 1963 257-274. 

\bibitem[Y]{Y} Young, R., Filling inequalities for nilpotent groups, arXiv:math/0608174.





\end{thebibliography}
\end{document}